\documentclass[oneside,reqno]{amsart}
\usepackage{amssymb, amsmath,mathtools}
\usepackage{mathrsfs}
\usepackage{amscd} 
\usepackage{bbm}
\usepackage{verbatim}
\usepackage{stmaryrd}
\usepackage[dvipsnames]{xcolor}
\usepackage{xfrac}
\usepackage{a4wide}
\usepackage[main=english,polish]{babel}
\usepackage{polski}
\usepackage{enumitem}
\usepackage[backend=biber,maxbibnames=199,doi=false,isbn=false,url=false]{biblatex}
\usepackage[colorlinks,linkcolor={blue},citecolor={blue},urlcolor={purple},]{hyperref}
\usepackage[colorinlistoftodos,prependcaption,textsize=tiny]{todonotes}
\newcommand{\wtxi}{%
\mspace{2mu}%
  \tilde{\mspace{-2mu}\rule{0pt}{1.48ex}\smash[t]{\xi}}%
}
\newcommand{\wteta}{%
\tilde{\eta}
}
\newcommand{\wtmu}{%
\mspace{2mu}%
  \tilde{\mspace{-2mu}\rule{0pt}{0.4ex}\smash[t]{\mu}}%
}
\newcommand{\wtu}{%
\mspace{2mu}%
  \tilde{\mspace{-1mu}\rule{0pt}{0.4ex}\smash[t]{u}}%
}

\newcommand{\dd}{\,\mathrm{d}}
\newcommand{\ddd}{\mathrm{d}}

\newcommand{\Id}{\mathrm{Id}}
\newcommand{\Tr}{\mathrm{Tr}}
\newcommand{\BM}{\mathrm{BM}} 

\newcommand{\one}{\mathbbm{1}}
\newcommand{\<}{\langle}
\newcommand{\?}{\rangle}

\newcommand{\into}{\hookrightarrow}

\newcommand{\law}{\mathtt{Law}}

\newcommand{\supp}{\mathrm{supp}}
\newcommand{\loc}{\mathrm{loc}}
\newcommand{\ceqq}{\coloneqq}
\newcommand{\eqqc}{\eqqcolon}
\newcommand{\col}{\colon}
\newcommand{\indep}{\perp\!\!\!\perp} 
\newcommand{\wien}{\mathrm{P}^{\infty}}
\newcommand{\as}{a.s.\ }
\newcommand{\aev}{a.e.\ }
\newcommand{\eg}{e.g.\ }
\newcommand{\ie}{i.e.\ }
\newcommand{\dotcup}{\,\dot{\cup}\,}

\newcommand{\eq}{({\scriptstyle\mathrm{SPDE}})}
\newcommand{\cond}{\mathrm{C}}

\theoremstyle{plain}
\newtheorem{theorem}{Theorem}[section]
\theoremstyle{remark}
\newtheorem{remark}[theorem]{Remark}
\newtheorem{example}[theorem]{Example}
\theoremstyle{plain}
\newtheorem{corollary}[theorem]{Corollary}
\newtheorem{lemma}[theorem]{Lemma}

\newtheorem{definition}[theorem]{Definition}
\newtheorem{assumption}[theorem]{Assumption}
\numberwithin{equation}{section}

\newcommand{\N}{\mathbb{N}}

\newcommand{\R}{\mathbb{R}}
\newcommand{\C}{\mathbb{C}}

\newcommand{\E}{\mathbb{E}}
\newcommand{\B}{\mathbb{B}}
\newcommand{\W}{\mathbb{W}}
\newcommand{\BB}{\mathcal{B}}

\newcommand{\LL}{\mathcal{L}} 
\newcommand{\F}{\mathcal{F}}
\newcommand{\G}{\mathcal{G}}
\renewcommand{\P}{\mathbb{P}}
\newcommand{\PP}{\mathcal{P}}
\newcommand{\A}{\mathcal{A}}
\renewcommand{\SS}{\mathcal{S}}
\newcommand{\CC}{\mathcal{C}}

\newcommand{\X}{X_{1/2}}
\newcommand{\com}{\angle_\CC}

\newcommand{\om}{\omega}
\newcommand{\Om}{\Omega}

\newcommand{\guy}{\gamma(U,Y)}
\newcommand{\guz}{\gamma(U,Z)}
\newcommand{\gluy}{\gamma(L^2(I_T;U),Y)}

\allowdisplaybreaks

\begin{document}

\author[E. S. Theewis]{Esm\'ee Theewis}
\address[E. S. Theewis]{Delft Institute of Applied Mathematics\\
Delft University of Technology \\ P.O. Box 5031\\ 2600 GA Delft\\The
Netherlands}
\email{e.s.theewis@tudelft.nl}

\thanks{The author is supported by the VICI subsidy VI.C.212.027 of the Netherlands Organisation for Scientific Research (NWO)}

\date{June 19, 2026}

\title[The Yamada--Watanabe--Engelbert theorem]{The Yamada--Watanabe--Engelbert theorem  \\ for SPDEs in Banach spaces} 

\keywords{Yamada--Watanabe theorem, Engelbert theorem, stochastic evolution equations, stochastic partial differential equations, cylindrical Brownian motion, stochastic integral}

\subjclass[2020]{Primary: 60H15, Secondary: 60H05}

\begin{abstract}
We give a unified proof of the Yamada--Watanabe--Engelbert theorem for various notions of solutions for SPDEs in Banach spaces with cylindrical Wiener noise. We use Kurtz' generalization of the theorems of Yamada, Watanabe and Engelbert. 
In addition, we deduce the classical Yamada--Watanabe theorem for SPDEs, with a slightly different notion of `unique strong solution' than that corresponding to the result of Kurtz. 

Our setting  includes analytically strong solutions, analytically weak solutions and mild solutions. Moreover, our approach offers flexibility with regard to the function spaces and integrability conditions that are chosen in the solution notion (and affect the meaning of existence and uniqueness). All results hold in Banach spaces which are either martingale type 2 or UMD. For analytically weak solutions, the results hold in arbitrary Banach spaces. 

In particular, our results extend the Yamada--Watanabe theorems of Ondrej\'at for mild solutions in 2-smooth Banach spaces, of R\"ockner et al.\ for the variational framework and  of Kunze for analytically weak solutions, and cover many new settings.   

As a tool, and of interest itself, we construct a measurable representation $I$ of the stochastic integral in a martingale type 2 or UMD Banach space, in the sense that for any stochastically integrable process $f$ and cylindrical Brownian motion $W$, we have $I(f(\om),W(\om),\law(f,W))=(\int_0^\cdot f\dd W)(\om)$ for almost every $\om$.  
\end{abstract}

\maketitle
\addtocontents{toc}{\protect\setcounter{tocdepth}{2}}

\tableofcontents

\section*{Introduction} 

The \emph{Yamada--Watanabe theorem} was originally proven by Yamada and Watanabe for stochastic differential equation (SDE) solutions in \cite[Cor.\ 1]{yamadawatanabe71}, 
and can be roughly summarized as:
\begin{align}\label{eq:YWclassical}
\text{weak existence $+$ pathwise uniqueness}\iff \text{ unique strong existence. }
\end{align}
A detailed proof can be found in  \cite[Th.\ 1.1 Chap.\ 4]{ikedawatanabe81}. 
In \eqref{eq:YWclassical}, `weak' is meant in a probabilistic sense, and `unique strong existence' is a comprehensive notion indicating that there exist probabilistically strong solutions, all uniquely determined by a single solution map $F$ with several measurability properties, mapping the initial data and the Brownian motion to the solution to the SDE. In particular,  `$\Leftarrow$' in the above is trivial and `$\Rightarrow$' is the useful  implication.  

Afterwards, it was observed that the proof would also extend  to infinite dimensional SPDE settings, first by Brze\'zniak and Gątarek in \cite[p.\ 212]{brzezniakgatarek99}. 
A first rigorous proof was provided by Ondrej\'at in his thesis \cite{ondrejat04},   
where the Yamada--Watanabe theorem was proved for mild solutions to SPDEs in 2-smooth Banach spaces. 
Other important contributions to SPDE settings were made by R\"ockner, Schmuland and Zhang \cite{rock08} (variational framework) and by Kunze \cite{kunze13} (analytically weak solutions), who all used the  proof strategy from SDEs as in \cite{ikedawatanabe81}.

Meanwhile, a reverse of the Yamada--Watanabe theorem was provided by Jacod \cite{jacod80} and by Engelbert \cite[Th.\ 2]{engelbert91}, leading to the equivalence result:
\begin{align}\label{eq:YWE}
\begin{split}
&\text{weak existence $+$ pathwise uniqueness}\\
&\iff \text{ existence of a strong solution $+$ joint weak uniqueness, } 
\end{split}
\end{align}
commonly referred to as the \emph{Yamada--Watanabe--Engelbert theorem}.
Here, the last line is much weaker than unique strong existence (see also \cite[Cor.\ p.\ 209]{engelbert91}).  Thus `$\Leftarrow$' provides new information compared to \eqref{eq:YWclassical} and can be used for proving pathwise uniqueness.  
A slight difference   is that in the classical Yamada--Watanabe result, the left-hand side of \eqref{eq:YWclassical} is assumed for all initial distributions, whereas in the Yamada--Watanabe--Engelbert result, the equivalence \eqref{eq:YWE} holds for a fixed initial distribution.  

A much more general  version of the  Yamada--Watanabe--Engelbert theorem \eqref{eq:YWE} was proved by Kurtz \cite[Prop.\ 3.14]{kurtz07}. The result applies to general stochastic equations with a convexity structure and a compatibility structure capturing adaptedness properties, and is suited for S(P)DEs. 
Subsequently, in \cite[p.\ 7, Th.\ 3.1, Th.\ 3.4]{kurtz14}, Kurtz extended this result by entirely dropping the convexity assumption and by treating more general compatibility structures, including larger classes of SDEs, \eg those with distribution-dependent coefficients, which were not included in the former version \cite[Ex.\ 2.14]{kurtz14}. 

In this paper, we use the results of \cite{kurtz07,kurtz14} to prove both the Yamada--Watanabe--Engelbert theorem (Theorem \ref{th:YW SPDE}) and the classical Yamada--Watanabe theorem (Corollary \ref{cor: single F}) for very general SPDE settings. 
To the best of our knowledge, with the exception of \cite{de_bouard19,hausenblas25}, all rigorous proofs of the Yamada--Watanabe theorem for SPDE settings have not been based on Kurtz' results, but have followed the classical approach of \cite{ikedawatanabe81}. Also the reverse, \ie `$\Leftarrow$'  of \eqref{eq:YWE}, has been proved mostly without the use of Kurtz' results (see \eg \cite[Th.\ 3.1]{rehmeier21}). Thus our approach is rather new. The advantage of using the results of Kurtz is that it reveals in a more direct way which structure is exactly needed of the equation and underlying spaces.

No examples of applications to SPDEs were given in \cite{kurtz07,kurtz14}. Still, Kurtz' result \cite[Prop.\ 3.14]{kurtz07} has been cited to deduce strong existence for various SPDEs \cite{rocknerzhuzhu14,rocknerzhuzhu17,flandoliluo20,flandoliluo21,flandoligaleatiluo21,tahraouicipriano24}. 
Our work in particular provides a detailed explanation of how Kurtz' result can be applied in such settings.  
When doing so, the main issues that need to be handled are the following:
\begin{itemize}
  \item The solution notion, including the actual stochastic equation that a solution needs to  satisfy,  
  has to be characterized by the joint distribution of the solution and the noise. 
  \item Suitable measurability conditions and (Polish) function spaces have to be chosen for the solution and the coefficients in the SPDE.
\end{itemize}
For SPDEs, these issues are more delicate than for SDEs. In fact, one of the crucial parts of Ondrej\'at's proof concerned the first issue 
\cite[Th.\ 6]{ondrejat04}. 
To solve the first issue in our setting, essentially without extra effort, we prove existence of a measurable representation for stochastic integrals in martingale type 2 or UMD Banach spaces, see Theorem \ref{th:stoch int rep type 2 or UMD}.  
That is, there exists a measurable map $I$ such that for any stochastically integrable (possibly discontinuous) process $f$ and cylindrical Brownian motion $W$, we have 
\[
I\big(f(\om),W(\om),\law(f,W)\big)=\Big(\int_0^\cdot f\dd W\Big)(\om) \quad\text{for almost every $\om$.}
\]   
Such representations are well known in finite dimensions, but as far as we know,  
the above has not been proved for infinite dimensions (in other proofs of Yamada--Watanabe results, one fixes $\law(f,W)$). The UMD case is especially interesting here, and a bit more involved.

The generality of the setting in this paper is such that our Yamada--Watanabe(--Engelbert) theorem can be applied to many different notions of solutions, in particular to analytically strong, analytically weak and mild solutions.   
Furthermore, our method is  flexible in several respects:  
\begin{itemize}
\item The solution and the coefficients in the equation take values in and are defined on arbitrary separable Banach spaces.      
\item The solutions may be defined on a bounded time interval $[0,T]$ or on $[0,\infty)$.    
\item The space $\B$ in which solutions are required to live (pointwise in $\om\in\Om$) may be chosen flexibly, as long as it is Polish and embeds continuously into some space $C([0,T];Z)$ or $C([0,\infty);Z)$, with $Z$ a separable Banach space.  In particular, (weighted) Sobolev spaces and more exotic spaces can be used in the choice of $\B$. Also, weak continuity and positivity conditions can be included (see Example \ref{ex:applic}). 
\item The coefficients in the SPDE are allowed to depend on time and on the whole path of the solution. No specific (\eg semilinear) structure is assumed.
\item Moment conditions (with respect to $\om$) can be added to the notion of solution.  
 \end{itemize} 
Some minor aspects are taken into account as well:   
 \begin{itemize}
  \item  Treating the cylindrical Brownian motion as $\R^\infty$-valued or as trace class noise in a larger Hilbert space through a Hilbert--Schmidt embedding (Remarks \ref{rem: mathbbW choice} and \ref{rem: mathbbW choice YW}).  
  \item Allowing for all filtered probability spaces or only those which are complete or satisfy the usual conditions (Remark \ref{rem:sol on completion}). 
  \end{itemize}

As a consequence, we cover the well-known Yamada--Watanabe results of Ondrej{\'a}t \cite{ondrejat04} (Example  \ref{ex:ondrejat setting}), R\"ockner, Schmuland and Zhang \cite{rock08} (Example \ref{ex:variational}) and Kunze \cite{kunze13} (Example \ref{ex:kunze setting}). 
More importantly, we obtain the Yamada--Watanabe--Engelbert theorem for various classes of solutions to SPDEs for which the  result was not available yet, such as:
\begin{enumerate}[label=\alph*),ref=\alph*)]
  \item\label{it:intro1} Analytically weak solutions in Banach spaces that are neither UMD nor of M-type 2, without structural assumptions on the SPDE.  
  \item\label{it:intro2} Mild solutions to SPDEs with time-dependent differential operators. 
    \item \label{it:intro3} 
    Quasilinear SPDEs that cannot be treated in a variational setting. Examples are those in    \cite{agrestisauerbrey24, bechtelveraar23}. The framework of the present paper was recently applied in \cite{BT} to extend the weak existence result of \cite[Th.\ 6.6]{bechtelveraar23} to strong existence. 
  \item\label{it:intro4} Semilinear SPDEs in critical spaces. Examples are the following equations in critical spaces with Gaussian transport noise, in an $L^p(L^q)$-framework ($p>2$):  
      Allen-Cahn \cite[\S 7.1, (7.12)]{AV22nonlinear1}, Navier-Stokes \cite{AV_NS24}, Cahn-Hilliard \cite[\S 7.3, (7.24)]{AV22nonlinear1}, reaction-diffusion equations \cite{AV23reacdiff}. 
\end{enumerate}
The above examples do not fit in \cite{ondrejat04,rock08,kunze13} for different reasons. For \ref{it:intro2}, one has to use evolution families instead of a semigroup, which had not been treated yet.  Concerning \ref{it:intro1} and \ref{it:intro3}: the settings of  \cite[p.\ 8, (0.2)]{ondrejat04} and \cite[Hyp.\ 4.1, Def.\ 3.3]{kunze13} are well-suited for semilinear equations, but not for quasilinear equations, due to the structural assumptions. For \ref{it:intro4}, note that  the variational setting of \cite{rock08} does not apply since $p>2$. 
Moreover, the reason that \cite{kunze13} cannot be applied with critical spaces is that the smallest space in which solutions are continuous is a real interpolation space $(X_0,X_1)_{1-\frac{1+\kappa}{p},p}$. Now, the coefficients in the equations are not bounded on bounded subsets of this interpolation space. Therefore, the space $E$ of \cite[Hyp.\  3.1, Def.\ 3.3]{kunze13}, in which solutions are continuous, cannot be chosen suitably. Similarly, the underlying spaces cannot be chosen in a way fitting  \cite{ondrejat04}, here due to the transport noise. 
The  restrictions are overcome in our setting by the flexibility of  Assumption \ref{ass}.

\subsection*{Overview of the paper}
Section \ref{sec:prelim} consists of preliminaries. In Subsection \ref{sub: YW Kurtz}, we revisit Kurtz' version of the Yamada--Watanabe--Engelbert theorem and provide some extra probabilistic details for the more analysis-oriented reader. In Subsection \ref{sub:stoch integration}, some relevant aspects of stochastic integration in Banach spaces are discussed.  In Section \ref{sec:setting}, we explain the settings and our solution notions for SPDEs (Definition \ref{def: SPDE C-sol}), and provide examples of settings that fit into the framework (Example \ref{ex:applic}).  In Section \ref{sec:YW SPDE}, we prove existence of a measurable representation of the stochastic integral (Theorem \ref{th:stoch int rep type 2 or UMD}). This will then be used to  prove the Yamada--Watanabe--Engelbert theorem (Theorem \ref{th:YW SPDE}) for SPDEs. Lastly, we prove the classical Yamada--Watanabe theorem with a stronger notion of `unique strong solution' (Corollary \ref{cor: single F}). 

\vspace{.4cm}
During the final phase of the completion of this manuscript, it came to our attention that related results have been obtained simultaneously and independently by Fahim, Hausenblas and Karlsen \cite{hausenblas25}. Their setting generalizes the Yamada--Watanabe--Engelbert theorem in a different direction: by including noise coming from a Poisson random measure, besides a Wiener process. On the other hand, the framework is only variational, with  coefficients defined on Hilbert and reflexive Banach spaces that are part of a Gelfand triple. 
In contrast, our work focuses on Gaussian noise and provides a Yamada--Watanabe(--Engelbert) theorem in Banach spaces,  for various solution notions. The variational framework is a special case here. The coefficients may be Markovian as in \cite{hausenblas25} (see Lemma  \ref{lem:coeffs markov}), but may also depend on the whole time paths of the solution.

\subsubsection*{Acknowledgements}  
The author is grateful to Mark Veraar for many helpful discussions during the research process, and  thanks Jan van Neerven, Markus Kunze and the anonymous referees for their valuable comments on the manuscript. 
Lastly, discussions with Max Sauerbrey, Klaas Pieter Hart and Joris van Winden on related topics  were much appreciated, as was a short exchange with Martin Ondrej\'at on stochastic integral representations. 

\subsection*{Notation}\label{sub:notation}
The following notation will be used throughout the paper. 

For topological spaces $S$, $S'$ and $A,B\subset S$, we write
\begin{itemize}\setlength\itemsep{.3em}
\item $A \dotcup B$: union of $A$ and $B$ with $A\cap B=\varnothing$,
\item $S\into S'$: $S$ embeds continuously into $S'$,
\item $\BB(S)$:  Borel $\sigma$-algebra on $S$, 
\item $\PP(S)$: set of Borel probability measures on $S$,
\item $\Id_S$: identity map $S\to S$.
\end{itemize} 
Let $(S_1,{\A}_1,\mu_1)$,  $(S_2,\A_2,\mu_2)$  be $\sigma$-finite measure spaces. Let $\phi\col S_1\to S_2$ be measurable and let ${\tilde{\A}_1} $ be another $\sigma$-algebra on $S_1$. We use the following notations:
\begin{itemize}\setlength\itemsep{.3em}
\item $(S_1\times S_2,\mathcal{A}_1\otimes \mathcal{A}_2,\mu_1\otimes\mu_2)$:  product measure space,
  \item $\PP(S_1,\A_1)$: probability measures on $(S_1,\A_1)$,
\item $\A_1\vee\tilde{\A}_1   \ceqq \sigma(\A_1\cup \tilde{\A}_1 )$,
\item $\phi\#\mu_1\in\PP(S_2,\A_2)$: pushforward measure of $\mu_1$ under $\phi$,
\item $\BM(S_1,\A_1)$: bounded, measurable, real-valued functions $(S_1,\A_1)\to(\R,\BB(\R))$.
\end{itemize}
For a probability space  $(\Om,\F,\P)$  and random variables $\xi\col\Om\to S_1$, $\eta\col \Om\to S_2$, we write  
\begin{itemize}\setlength\itemsep{.3em} 
\item $\xi\indep \eta$: independent random variables, 
\item $\law(\xi)\ceqq \xi\#\P$ if the underlying probability measure $\P$ is clear,
\item $(\Om,\overline{\F},\overline{\P})$: completion of $(\Om,\F,\P)$.
\end{itemize}
If  $(\F_t)$ is a filtration on $(\Om,\F,\P)$, we let
\begin{itemize}
  \item $\overline{\F_t}$:   $\sigma$-algebra generated by $\F_t$ and all $(\Om,\F,\P)$-null sets.
\end{itemize}
For $T\in(0,\infty]$, we write
\begin{itemize}\setlength\itemsep{.3em}
\item $I_T\ceqq (0,T)$,
\item $\bar{I}_T\ceqq\overline{I_T}$,
\item $\R_+\ceqq [0,\infty)$. 
\end{itemize}

For a Hilbert space $U$ and a Banach space $Y$, we write 
\begin{itemize}\setlength\itemsep{.3em}
\item $\<\cdot,\cdot\?_U$: inner product of $U$, 
\item $\<\cdot,\cdot\?$: duality pairing between $Y$ and $Y^*$, 
  \item $\LL(U,Y)$: bounded linear operators from   $U$ to $Y$,  $\LL(U)\ceqq \LL(U,U)$, 
  \item $\gamma(U,Y)\subset \LL(U,Y)$: $\gamma$-radonifying operators.  
\end{itemize}  
We reserve the letter $k$ for indices in $\N\ceqq\{1,2,\ldots\}$ and the letter $t$ for indices in $\R_+$:
\begin{itemize}\setlength\itemsep{.3em}
  \item $(x_k)\ceqq(x_k)_{k\in\N}$,
  \item $(\F_t)\ceqq (\F_t)_{t\in\R_+}$.
\end{itemize}

\section{Preliminaries}\label{sec:prelim}

\subsection{The Yamada--Watanabe--Engelbert theorem by Kurtz}\label{sub: YW Kurtz}

To begin, we will discuss the Yamada--Watanabe--Engelbert theorem for $\CC$-compatible solutions to  stochastic equations given in \cite[\S2]{kurtz14}. In Section \ref{sec:YW SPDE}, it will be used to obtain a Yamada--Watanabe--Engelbert theorem for solutions to SPDEs. 

Throughout this paper, we call a map from a measurable space to a metric space $S$ a \emph{random variable} if it is measurable, \ie every inverse image of a Borel set is a measurable set. 

Within this section, we assume that $S_1$ and $S_2$ are arbitrary but fixed complete metric spaces and unless specified otherwise, we let them be equipped with their Borel $\sigma$-algebra.  

The following abstract definitions are identical to those in \cite{kurtz14}. In our later application, $\xi$ will correspond to an SPDE solution $u$  and $\eta=(u(0),W)$ will consist of the initial value and the noise. 

\begin{definition}\label{def:compatible} 
Let $\A$ be any index set. Let $\BB_\alpha^{S_i}$ be a sub-$\sigma$-algebra of $\BB(S_i)$ for $i=1,2$ and $\alpha\in\A$. 
We call $ \CC\ceqq\{(\BB_\alpha^{S_1},\BB_\alpha^{S_2}):\alpha\in\A\}$ a \emph{compatibility structure}. 

Let $\xi,\xi_1,\xi_2\col \Om\to S_1$ and $\eta\col \Om\to S_2$ be random variables. We define as in  \cite[p.\ 5]{kurtz14}: 
\[
\F_\alpha^\xi\ceqq   \overline{\sigma(\xi^{-1}(B):B\in\BB_\alpha^{S_1})},  
 \quad\F_\alpha^\eta\ceqq  \overline{\sigma(\eta^{-1}(B):B\in\BB_\alpha^{S_2})}.  
\]
We say that $\xi$ is \emph{$\CC$-compatible with $\eta$}, which we denote by `$\xi\com \eta$',  
  if for any $h\in \BM(S_2,\BB(S_2))$ and $\alpha\in\A$:
  \begin{equation}\label{eq:compatible}
  \E[h(\eta)|\F_\alpha^\xi\vee \F_\alpha^\eta]=\E[h(\eta)| \F_\alpha^\eta].
  \end{equation}
Furthermore, we say that $(\xi_1,\xi_2)$ is \emph{jointly $\CC$-compatible with $\eta$}, which we denote by  `$(\xi_1,\xi_2)\com \eta$', if for any $h\in\BM(S_2,\BB(S_2))$ and  $\alpha\in\A$:
\begin{equation}\label{eq:compatiblejoint}
\E[h(\eta)|\F_\alpha^{\xi_1}\vee\F_\alpha^{\xi_2}\vee \F_\alpha^\eta]=\E[h(\eta)| \F_\alpha^\eta].
\end{equation}
  
Let $\nu\in\PP(S_2)$ and $\Gamma\subset \PP(S_1\times S_2)$. For $i=1,2$, let $\pi_i\col S_1\times S_2\to S_i$ be the canonical projection. We define
\begin{align}\label{eq: def Kurtz set}
&\PP_\nu(S_1\times S_2)\ceqq\{\mu\in\PP(S_1\times S_2):\pi_2\#\mu=\nu\}, \notag\\ 
\begin{split}
&\SS_{\Gamma,\CC,\nu}\ceqq \{\mu\in \PP_\nu(S_1\times S_2)\cap \Gamma:\, \exists \text{ probability space }(\Om,\F,\P),\, \exists\text{ measurable }\xi\col\Om\to S_1\\
&\hspace{5cm} \text{ and }\eta\col\Om\to S_2,   \text{ such that }\xi \com \eta \text{ and }\mu=\law(\xi,\eta)\}.
\end{split}
\end{align} 
\end{definition} 

The compatibility conditions \eqref{eq:compatible} and \eqref{eq:compatiblejoint} allow us to capture adaptedness and independence properties, to be made precise in Lemmas \ref{lem:SPDE sol is compatible} and \ref{lem:joint compatible indep incr}. Furthermore,  $\Gamma$ will capture the specific stochastic equation and moment conditions required for the solution and coefficients. Let us stress that $\Gamma$ is allowed to be any subset of  $\PP(S_1\times S_2)$. 
 
\begin{definition}\label{def: kurtz solution}
Let $\Gamma,\CC$ and $\nu$ be as in Definition \ref{def:compatible}.
Let $\xi\col\Om\to S_1$ and $\eta\col\Om \to S_2$ be random variables on a probability space $(\Om,\F,\P)$. 

We call $(\xi,\eta)$ a \emph{weak solution for $(\Gamma,\CC,\nu)$} if $\law(\xi,\eta)\in \SS_{\Gamma,\CC,\nu}$. 

We call    
$(\xi,\eta)$ a \emph{strong solution for $(\Gamma,\CC,\nu)$} if $\law(\xi,\eta)\in \SS_{\Gamma,\CC,\nu}$ and $\xi=F(\eta)$ a.s. for some Borel measurable map $F\col S_2\to S_1$. 

We say that \emph{pointwise uniqueness holds} if whenever $(\xi,\eta)$ and $(\wtxi,\eta)$ are weak solutions for $(\Gamma,\CC,\nu)$ and defined on the same probability space: $\xi=\wtxi$ a.s.

We say that \emph{joint uniqueness in law holds} if $\#S_{\Gamma,\CC,\nu}\leq 1$.  
   
We say that \emph{uniqueness in law holds} if 
for all $\mu,\wtmu\in\SS_{\Gamma,\CC,\nu}$: $\pi_1\#\mu=\pi_1\#\wtmu$.   
\end{definition}

In SPDE terminology, pointwise uniqueness corresponds to pathwise uniqueness, joint uniqueness in law corresponds to weak joint uniqueness and uniqueness in law corresponds to weak uniqueness. This will be made precise in Theorem \ref{th:YW SPDE}.

\begin{remark}
Recalling the definition of $\SS_{\Gamma,\CC,\nu}$ (Definition \ref{def:compatible}), it is a priori unclear whether all weak solutions for $(\Gamma,\CC,\nu)$ are $\CC$-compatible.  Fortunately, the latter holds, as observed in \cite[(2.3)]{kurtz14}. Indeed, \eqref{eq:compatible} holds if and only if 
\[
\inf_{f\in \BM(S_1\times S_2,\BB_\alpha^{S_1}\otimes\BB_\alpha^{S_2})}\E[(h(\eta)-f(\xi,\eta))^2]=\inf_{f\in\BM(S_2,\BB_\alpha^{S_2})}\E[(h(\eta)-f(\eta))^2], 
\]
using that $h$ is bounded, so the conditional expectation is an $L^2$-projection. 

Now if $\xi\col\Om\to S_1$, $\eta\col\Om\to S_2$, ${\wtxi}\col\tilde{\Om}\to S_1$ and ${\wteta}\col\tilde{\Om}\to S_2$ are random variables on probability spaces $(\Om,\F,\P)$ and $(\tilde{\Om},\tilde{\F},\tilde{\P})$ such that   $\law(\xi,\eta)=\law({\wtxi},{\wteta})$ and $\xi\com \eta$, then 
we may replace $\xi$ and $\eta$ by  $\wtxi$ and $\wteta$ in the identity above, proving that 
${\wtxi}\com {\wteta}$.  
 \end{remark}

The next lemma, which is similar to \cite[Lem.\ 1.3c)]{kurtz14}, will be used. Note that \eqref{it:a.s. equal}  is true  if $E_1$ and $E_2$ are separable metric spaces and $\A_i=\BB(E_i)$.

\begin{lemma}\label{lem:cond expec dist}
Let $(\Om,\F,\P)$   be a probability space and let $(E_1,\mathcal{A}_1)$, $(E_2,\mathcal{A}_2)$ be measurable spaces. Let $\xi\col\Om\to E_1$ and $\eta\col\Om\to E_2$ be measurable. 
Suppose that 
\begin{equation}\label{it:a.s. equal}
 D\ceqq \{(z,z):z\in E_1\times E_2\}\in (\A_1\otimes \A_2) \otimes (\A_1\otimes \A_2).
\end{equation}   
If $\law(\xi,\eta)=\law(F(\eta),\eta)$ for a measurable map $F\col (E_2,\A_2)\to (E_1,\A_1)$, then $\xi=F(\eta)$ a.s.  
\end{lemma} 
\begin{proof} 
Let $\nu\ceqq\law(\eta)$. We have $\law(\xi,\eta)=\law(F(\eta),\eta)=\delta_{F(y)}(\ddd x)\nu(\ddd y)$.  
Let $C\ceqq \{(F(y),y):y\in E_2\}$. Note that $\one_C(x,y)=\one_D((F(y),y),(x,y))$. Since $D\in (\A_1\otimes \A_2) \otimes (\A_1\otimes \A_2)$ and $(x,y)\mapsto ((F(y),y),(x,y))$ is $\A_1\otimes \A_2/(\A_1\otimes \A_2) \otimes (\A_1\otimes \A_2)$-measurable, it holds that $C\in \A_1\otimes \A_2$. Thus we may compute 
\begin{align*}
\P(\xi=F(\eta))
=\int_{E_2}\int_{E_1} \one_C(x,y)\delta_{F(y)}(\ddd x)\nu(\ddd y)&=\int_{E_2} \one_C(F(y),y)\nu(\ddd y)
=\int_{E_2}\nu(\ddd y)=1.
\end{align*}  
\end{proof}

The following result can be found in \cite[Lem.\ 2.10]{kurtz14}. For convenience, we provide some  details. 

\begin{corollary}\label{cor:compatible}   
Pointwise uniqueness for jointly $\CC$-compatible solutions is equivalent to pointwise uniqueness for $\CC$-compatible solutions.  
\end{corollary}
\begin{proof}   
Let $\xi,\wtxi\col\Om\to S_1$ and $\eta,\wteta\col\Om\to S_2$ be random variables. 
Note that $(\xi,{\wtxi})\com \eta$ implies $\xi\com \eta$ and ${\wtxi}\com \eta$. Thus pointwise uniqueness for $\CC$-compatible solutions implies pointwise uniqueness for jointly $\CC$-compatible solutions. For the other direction, assume that we have pointwise uniqueness for jointly $\CC$-compatible solutions and suppose that $\law(\xi,\eta),\law({\wtxi},\eta)\in \SS_{\Gamma,\CC,\nu}$. We have to show that $\xi={\wtxi}$ a.s. 

As in the proof of \cite[Th.\ 1.5 a)$\Rightarrow$b)]{kurtz14}: by \cite{blackwelldubins83}, there exist Borel measurable maps $G,\tilde{G}\col S_2\times[0,1]\to S_1$, such that for any random variables $\chi$, $\tilde{\chi}$ and $\zeta$ with $\chi$, $\tilde{\chi}$ uniformly distributed, $\law(\zeta)=\nu$,  $\zeta\indep \chi$ and $\zeta\indep \tilde{\chi}$, 
we have 
\begin{equation}\label{eq:G_i}
  \law(G(\zeta,\chi),\zeta)=\law(\xi,\eta)\in\SS_{\Gamma,\CC,\nu},\quad \law(\tilde{G}(\zeta,\tilde{\chi}),\zeta)=\law({\wtxi},\eta)\in\SS_{\Gamma,\CC,\nu}.
\end{equation}
In particular, the above holds for the random variables $\chi,\tilde{\chi}, \zeta$ on $(\Om',\F',\P')\ceqq(\Om\times [0,1]\times [0,1],\F\otimes\BB([0,1])\otimes \BB([0,1]),\P\otimes \lambda\otimes\lambda)$ (with $\lambda$ the Lebesgue measure on $[0,1]$) defined by 
\begin{align}\label{eq: zeta chi}
  \zeta(\om,x,\tilde{x})\ceqq\eta(\om), \quad\chi(\om,x,\tilde{x})\ceqq x, \quad \tilde{\chi}(\om,x,\tilde{x})\ceqq \tilde{x}. 
\end{align}
Furthermore, these $\zeta, \chi$ and $\tilde{\chi}$ are jointly independent, so application of 
  \cite[Lem.\ 2.11]{kurtz14} yields $(G(\zeta,\chi),\tilde{G}(\zeta,\tilde{\chi}))\com \zeta$. Combined with \eqref{eq:G_i}, we conclude that $(G(\zeta,\chi),\zeta)$ and $(\tilde{G}(\zeta,\tilde{\chi}),\zeta)$ are jointly $\CC$-compatible weak solutions for $(\Gamma,\CC,\nu)$.  
  Thus, the assumed pointwise uniqueness yields $G(\zeta,{\chi})=\tilde{G}(\zeta,\tilde{\chi})$ a.s. Recalling the independence of $\zeta$, $\chi$ and $\tilde{\chi}$,  \cite[Lem.\ A.2]{kurtz07} implies that $G(\zeta,\chi)=\tilde{G}(\zeta,\tilde{\chi})=F(\zeta)$ a.s. for a Borel measurable map $F\col S_2\to S_1$. Together with  \eqref{eq:G_i} and the definitions of $\zeta$ and $\P'$, we thus obtain 
  \[
  \law(\xi,\eta)=\law({\wtxi},\eta)=\law(F(\zeta),\zeta)
  =\law(F(\eta),\eta).
  \] 
At last, Lemma \ref{lem:cond expec dist}  gives $\xi=F(\eta)={\wtxi}$ a.s., completing the proof. 
\end{proof}

The next theorem is an extension of the Yamada--Watanabe--Engelbert theorem  \cite[Th.\ 1.5]{kurtz14} to  $\CC$-compatible solutions, as mentioned in \cite[p.\ 7]{kurtz14}. With the  preparations done, the proof is essentially the same as that of \cite[Th.\ 1.5]{kurtz14} (without $\CC$-compatibility) and \cite[Th.\ 3.14]{kurtz07} ($\Gamma$ convex), but we include it and add the equivalent statement \ref{it:YW3Kurtz} to the formulation.

\begin{theorem}\label{th:YW}
The following are equivalent:
\begin{enumerate}[label=\textup{(\alph*)},ref=\textup{(\alph*)}] 
\item \label{it:YW1Kurtz} $S_{\Gamma,\CC,\nu}\neq\varnothing$ and pointwise uniqueness holds.
\item \label{it:YW2Kurtz} There exists a strong solution for $(\Gamma,\CC,\nu)$ and joint uniqueness in law holds. 
\item \label{it:YW3Kurtz} $S_{\Gamma,\CC,\nu}\neq\varnothing$ and joint uniqueness in law holds, and there exists a Borel measurable map $F\col S_2\to S_1$ s.t. for any random variable $\eta\col\Om\to S_2$ with $\law(\eta)=\nu$,  $(F(\eta),\eta)$ is a strong solution for $(\Gamma,\CC,\nu)$. Lastly, any weak solution $(\xi,\eta)$ for $(\Gamma,\CC,\nu)$ satisfies $\xi=F(\eta)$ a.s. 
    \end{enumerate}
\end{theorem}
\begin{proof}
\ref{it:YW1Kurtz}$\Rightarrow$\ref{it:YW3Kurtz}: Let  ${\mu},{\wtmu}\in \SS_{\Gamma,\CC,\nu}$ be arbitrary, not necessarily distinct.  
By \cite{blackwelldubins83}, there exist Borel measurable maps $G,\tilde{G}\col S_2\times[0,1]\to S_1$ such that for any random variable $\zeta\col\Om'\to S_2$ with $\law(\zeta)=\nu$ and for any uniformly distributed $\chi,\tilde{\chi}\col\Om'\to[0,1]$ with $\chi\indep \zeta$, $\tilde{\chi}\indep \zeta$ and ${\chi}\indep \tilde{\chi}$: 
\begin{equation}\label{eq:G_i2}
 \law(G(\zeta,\chi),\zeta)=\mu\in\SS_{\Gamma,\CC,\nu}, \quad 
 \law(\tilde{G}(\zeta,\tilde{\chi}),\zeta)=\wtmu\in\SS_{\Gamma,\CC,\nu}. 
\end{equation} 
Let $\eta\col \Om\to S_2$ be measurable with $\law(\eta)=\nu$. We apply the above with $(\Om',\F',\P')$, $\zeta$, $\chi$ and $\tilde{\chi}$ defined by \eqref{eq: zeta chi}.   
Pointwise uniqueness yields $(G(\zeta,\chi),\zeta)=\tilde{G}(\zeta,\tilde{\chi}),\zeta)$ $\P'$-a.s. Combining the latter with independence of ${\chi}$ and $\tilde{\chi}$ and \cite[Lem.\ A.2]{kurtz07}, we obtain  
\begin{equation}\label{eq: F(eta)}
F(\zeta)=G(\zeta,{\chi})=\tilde{G}(\zeta,\tilde{\chi})\quad\P'\text{-a.s. } 
\end{equation}
for some Borel measurable map $F\col S_2\to S_1$. The construction of $F$ in \cite[Lem.\ A.2]{kurtz07} does not depend on $\zeta$ (or $\eta$), only on $\tilde{G}$ and $\law(\tilde{\chi})$. Recalling the construction of $\P'$ and $\zeta$ and combining \eqref{eq:G_i2} and \eqref{eq: F(eta)}, we find $
\law(F(\eta),\eta)=\law(F(\zeta),\zeta)=\mu=\wtmu\in S_{\Gamma,\CC,\nu}$, \ie $(F(\eta),\eta)$ is a strong solution. Also, ${\mu}={\wtmu}$ yields $\#S_{\Gamma,\CC,\nu}=1$, proving joint uniqueness in law. 

For the last part of \ref{it:YW3Kurtz}, let $(\xi,\eta)$ be any weak solution for $(\Gamma,\CC,\nu)$. Then $\law(\eta)=\nu$, so by the above, $(F(\eta),\eta)$ is a weak solution for $(\Gamma,\CC,\nu)$, living on the same probability space as $(\xi,\eta)$. Pointwise uniqueness yields  $\xi=F(\eta)$ a.s.   

\ref{it:YW3Kurtz}$\Rightarrow$\ref{it:YW2Kurtz}: Trivial.

\ref{it:YW2Kurtz}$\Rightarrow$\ref{it:YW1Kurtz}: Clearly, $S_{\Gamma,\CC,\nu}\neq\varnothing$. It remains to prove pointwise uniqueness. Let  $(\xi_*,\eta_*)$ be a strong solution for $(\Gamma,\CC,\nu)$ with  measurable  $F\col S_2\to S_1$ such that $\xi_*=F(\eta_*)$ a.s. Put $\mu_*\ceqq \law(\xi_*,\eta_*)\in S_{\Gamma,\CC,\nu}$.   
We have $\mu_*(\ddd x,\ddd y)=\delta_{F(y)}(\ddd x)\nu(\ddd y)$. 
Now let $({\xi},\eta)$ and $({\wtxi},\eta)$ be solutions for $(\Gamma,\CC,\nu)$ on the same probability space $(\Om,\F,\P)$. By \ref{it:YW2Kurtz}, $\#S_{\Gamma,\CC,\nu}\leq 1$, so $\law({\xi},\eta)=\law({\wtxi},\eta)=\mu_*$. 
Moreover, $\mu_*=\law(F(\eta),\eta)$, since for all $A\in \BB(S_1)$, $B\in \BB(S_2)$, we have  
\begin{align*}
\P((F(\eta),\eta)\in A\times B)=\!\int_{S_2}\one_{F^{-1}(A)\cap B}(y)\nu(\ddd y)&=\!\int_{S_2}\int_{S_1}\one_{A\times B}(x,y)\delta_{F(y)}(\ddd x)\nu(\ddd y)=\mu_*(A\times B).
\end{align*}
Therefore, $\law({\xi},\eta)=\law({\wtxi},\eta)=\mu_*=\law(F(\eta),\eta)$  and Lemma \ref{lem:cond expec dist} yields $ {\xi}=F(\eta)={\wtxi} $ a.s., proving pointwise uniqueness. 
\end{proof}

\subsection{Stochastic integration in Banach spaces}\label{sub:stoch integration}

Here, we discuss some relevant aspects of stochastic integration. 
Throughout this paper, we will use the following function space notations. 
If $X$ is a normed space and  $(S_1,{\A}_1,\mu_1)$ is a measure space, then $L^0(S_1;X)$ and $L^0_{\A_1}(S_1;X)$ denote the linear space of strongly $\mu_1$-measurable  
functions $f\col (S_1,\A_1)\to X$, with identification of $\mu_1$-\aev equal functions, and equipped with the topology of convergence in measure. 
For $p\in (0,\infty)$:
$$
L^p(S_1;X)\ceqq\{f\in L^0(S_1;X): \|f\|_{L^p(S_1;X)}<\infty\} , \quad \|f\|_{L^p(S_1;X)}\ceqq
({\textstyle\int_{S_1} \|f\|_X^p \dd \mu_1)^{\frac{1}{p}}}$$ 
and we put $L^p(S_1)\ceqq L^p(S_1;\R)$.  
For a metric space $M$ and $T\in(0,\infty]$,  $C(\bar{I}_T;M)$ denotes the space of continuous functions $\bar{I}_T\to M$ equipped with the topology of uniform convergence on compact subsets of $\bar{I}_T$. We put $C(\bar{I}_T)\ceqq C(\bar{I}_T;\R)$. 
We will use the following notation for point evaluation maps:
\begin{equation}\label{eq: point evaluation}
  \pi_M^t\col C(\bar{I}_T;M)\to M\col f\mapsto f(t),\qquad t\in\bar{I}_T.
\end{equation} 

Now, for stochastic integration, let us begin by defining a $U$-cylindrical Brownian motion.  
\begin{definition}\label{def: cylindrical BM} Let $(\Om,\F,\P,(\F_t))$ be a filtered probability space, let $U$ be a real separable Hilbert space and let $W\in\mathcal{L}(L^2(\R_+;U),L^2(\Om))$. Then $W$ is called a \emph{$U$-cylindrical Brownian motion} with respect to $(\Om,\F,\P,(\F_t))$ if for all $f,g\in L^2(\R_+;U)$ and $t\in\R_+$:
  \begin{enumerate}[label=\emph{(\roman*)},ref=\textup{(\roman*)}]
    \item \label{it:bm1} $Wf$ is normally distributed with mean zero and $\E[Wf Wg]=\<f,g\?_{L^2(\R_+;U)}$,
    \item \label{it:bm2} if $\supp(f)\subset [0,t]$, then $Wf$ is $\F_t$-measurable,
    \item \label{it:bm3} if $\supp(f)\subset [t,\infty)$, then $Wf$ is independent of $\F_t$.
  \end{enumerate}
\end{definition} 

Recall that we call a map from a measurable space to a metric space $S$ a random variable if it is measurable. 
For the Yamada--Watanabe(--Engelbert) theorem, we will need to view a $U$-cylindrical Brownian motion  $W$ as a random variable $\Om\to\W$ for some Polish space $\W$, but $W$ in Definition \ref{def: cylindrical BM} is not a random variable. For this reason, we will use an associated $\R^\infty$-Brownian motion.   

\begin{definition}\label{def: sequence independent BM}
Let $(\Om,\F,\P,(\F_t))$ be a filtered probability space. 
An \emph{$\R^\infty$-Brownian motion} with respect to $(\Om,\F,\P,(\F_t))$ is a sequence $(\beta_k)$ consisting of independent standard real-valued $(\F_t)$-Brownian motions. 
\end{definition}

Suppose now that an orthonormal basis $(e_k)$ for $U$ is given. Then the above definitions can be related as follows.   
Any $U$-cylindrical Brownian motion $W$  with respect to $(\Om,\F,\P,(\F_t))$ induces an everywhere continuous  $\R^\infty$-Brownian motion $(W^k)$ with respect to  $(\Om,\overline{\F},\overline{\P},(\overline{\F_t}))$ (the completion defined in \nameref{sub:notation}), through: 
\begin{equation}\label{eq: def Rinfty BM}
W^k(t)=W(\one_{(0,t]}\otimes e_k),\quad k\in\N, t\in\R_+. 
\end{equation} 
Here, for each $k\in\N$, we choose a continuous modification $W^k$ of $W(\one_{(0,\cdot]}\otimes e_k)$  using the Kolmogorov continuity theorem and \ref{it:bm1}. Since $W(\one_{(0,t]}\otimes e_k)=W^k(t)$ $\P$-\as for all $t\in\R_+$,  the claimed properties of $(W^k)$ follow from \ref{it:bm1}-\ref{it:bm3}.   
 
Conversely, for any $\R^\infty$-Brownian motion $(W^k)$ with respect to $(\Om,\F,\P,(\F_t))$, one can construct a $U$-cylindrical Brownian motion $W$ with respect to $(\Om,{\F},{\P},({\F_t}))$ satisfying \eqref{eq: def Rinfty BM}. 
Since linear combinations of functions $\one_{(0,t]}\otimes e_k$ are dense in $L^2(\R_+;U)$, and since we have $W\in \mathcal{L}(L^2(\R_+;U);L^2(\Om))$, we see that for any $f\in L^2(\R_+;U)$, $Wf$ (as an element of $L^2(\Om)$) is uniquely determined by $(W^k)$. That is, the $U$-cylindrical Brownian motion $W$ satisfying \eqref{eq: def Rinfty BM} is unique.  

Therefore, in contexts where an orthonormal basis $(e_k)$ is fixed, we identify 
\[
W=(W^k). 
\]
With this identification, and restricting to $\bar{I}_T$ ($T\in(0,\infty]$), $W$ is an actual random variable taking values in 
\[
\W\ceqq C(\bar{I}_T;\R^\infty),
\]  
where $\R^\infty$($=\R^{\N}$)  is equipped with 
$d_\infty(x,y)\ceqq \sum_{j=1}^\infty 2^{-j}(|x_j-y_j|\wedge 1)$, metrizing the product topology. 
We equip $\W$ with the topology of uniform convergence on compact subsets, so that it becomes a Polish space which is completely metrized by 
\[
\textstyle{
d(w,z)\ceqq \sum_{k=1}^\infty 2^{-k}\big[\big(\sup_{t\in \bar{I}_{k\wedge T}}d_\infty(w(t),z(t))\big)\wedge 1 \big].}
\]    
Now indeed, $W\col\Om\to\W$ is a random variable: by Definition \ref{def: cylindrical BM} and \eqref{eq: def Rinfty BM}, $W^k(t)\in L^2(\Om)$ for each $k$ and $t$, hence $W^k(t)\col \Om\to\R$ is measurable. Measurability of $W$ thus follows from Lemma \ref{lem: generating maps}, noting that $\R^\infty$ is completely Hausdorff  and the continuous maps $\W\to \R\col w\mapsto w_k(t)$, for $k\in\N$ and $t\in\bar{I}_T$ separate the points in $\W$. 
Alternatively, one may also use a $Q$-Wiener process induced by $W$, see Remark \ref{rem: Q-wiener}. 

In Section \ref{sec:YW SPDE}, we will write ${\wien}\in\PP(\W)$ for the law of a continuous $\R^\infty$-Brownian motion.  
\begin{remark}\label{rem:wien measure} $\wien$ is unique: let $(\beta_k)$ be any sequence of continuous independent standard real-valued Brownian motions with respect to any filtration.  
For   $n\in\N$, $(\beta_1,\ldots,\beta_n)$ is an $n$-dimensional standard Brownian motion, as can be seen from the L\'evy characterization \cite[Th.\ 3.16]{karatzas98}. Also, $
\BB(\W)=\sigma(\Pi_n:n\in\N)$ with $\Pi_n\col \W\to \W\col w\mapsto (w_1,\ldots,w_n,0,0,\ldots)$. As $\{\Pi_n^{-1}(A):A\in\BB(\W),n\in\N\}$ is a $\pi$-system generating $\BB(\W)$, we find that  $\law((\beta_k))$ is entirely determined by 
$\{\law(\beta_1,\ldots,\beta_n):n\in\N\}$, \ie the law of $n$-dimensional Brownian motions, which is unique. 
\end{remark} 
 
We will consider stochastic integration in M-type 2  spaces and in UMD spaces. For the definitions of these classes of Banach spaces, see \cite[Def.\ 4.4, Def.\ 5.2]{NVW15}. We recall that every Hilbert space is both UMD and of M-type 2. Moreover, $L^p(S_1)$   is UMD for $p\in(1,\infty)$ and has M-type 2 for $p\in[2,\infty)$. There also exist M-type 2 spaces which are not UMD. 

Let us mention that a Banach space has M-type 2 if and only if it has an equivalent 2-smooth norm \cite{pisier75}. Also, if a Banach space is UMD and has type 2, then it has M-type 2 \cite[Prop.\ 5.3]{NVW15}

Now let $Y$ be of M-type 2 or UMD.  Given a $U$-cylindrical Brownian motion with respect to $(\Om,\F,\P,(\F_t))$, one can define the stochastic integral $\int_0^\cdot f\dd W\in L^0(\Om;C(\bar{I}_T;Y))$ for  certain processes $f\col \bar{I}_T\times \Om \to \LL(U,Y)$. 
Conditions for stochastic integrability of $f$ will be given below.  For more details on stochastic integration in M-type 2 spaces and UMD spaces, we refer to   \cite[\S 4.2, \S 5.4]{NVW15}.  
For elementary adapted rank one processes, the stochastic integral is defined by  
\begin{equation}\label{eq: def stoch int elementary}
 \int_0^t\one_{(t_1,t_2]\times A}\otimes(u\otimes y)\dd W\ceqq \one_Ay\big(W(\one_{(0,t_2\wedge t]}\otimes u)-W(\one_{(0,t_1\wedge t]}\otimes u)\big), 
\end{equation}
for $0\leq t_1<t_2$, $A\in \F_{t_1}$, $u\in U$ and $y\in Y$, where $u\otimes y\in\LL(U,Y)$ is given by $v\mapsto \<u,v\?_Uy$. 

Definition \eqref{eq: def stoch int elementary} is extended linearly to \emph{elementary adapted processes}, \ie linear combinations of elementary adapted rank one processes.   
It is then further extended by continuity (from Burkholder's inequality for the M-type 2 case, and from the It\^o isomorphism for the UMD case), by localization and by existence of a  progressively measurable modification \cite[Th.\ 0.1]{ondrejatseidler13}. These procedures lead to a class of \emph{stochastically integrable processes}. 

Before we describe the latter, let us recall  that a  map $f\col S\to\LL(U,Y)$ is called $U$-strongly measurable if $f(\cdot)x\col S\to Y$ is strongly measurable for all $x\in U$. Likewise, $f\col \bar{I}_T\times \Om\to\LL(U,Y)$ is called $U$-strongly $(\F_t)$-adapted if $f(t,\cdot)x$ is strongly $\F_t$-measurable for all $t\in \bar{I}_T$ and $x\in U$. We now define stochastic integrability as in \cite[p.\ 2]{NVW07conditions}

\begin{definition}\label{def: UMD process} 
Let $T\in(0,\infty)$, let $(\Om,\F,\P,(\F_t))$ be a filtered probability space, let $U$ be a separable real Hilbert space and let $Y$ be a real Banach space. Let $W$ be a $U$-cylindrical Brownian motion on $(\Om,\F,\P,(\F_t))$. Let $\phi\col \bar{I}_T\times\Om\to \LL(U,Y)$ be $U$-strongly measurable. 

We say that $\phi$ is \emph{stochastically integrable} with respect to $W$ if there exist elementary adapted processes $\phi_k\col \bar{I}_T\times\Om\to \LL(U,Y)$ ($k\in\N$) and a $\zeta\col \Om\to C(\bar{I}_T;Y)$, such that 
\begin{itemize}
  \item $\<\phi_k(\cdot)u,y^*\?\to \<\phi(\cdot) u,y^*\?$ in measure for all $u\in U$ and $y^*\in Y^*$,  
  \item $\int_0^\cdot \phi_k \dd W\to \zeta$ in probability in $C(\bar{I}_T;Y)$.
\end{itemize} 
\end{definition} 

Let us discuss some   alternative conditions for stochastic integrability. To this end,   $\gamma$-radonifying operators are useful.  
For a Hilbert space $U$ and a Banach space  $Y$, $U\otimes Y\subset \LL(U,Y)$ denotes the linear  space of finite rank operators. Let  $(\gamma_k)$ be a sequence of independent real-valued standard Gaussian random variables. Then $\guy$, the Banach space of $\gamma$-radonifying operators, is defined as the completion of $U\otimes Y$ with respect to the norm 
\[
\textstyle{\|\sum_{n=1}^N u_n\otimes y_n\|_{\guy}\ceqq \big(\E\big[\|\sum_{n=1}^N \gamma_ny_n\|_Y^2\big]\big)^{1/2}, }
\] 
where $u_1,\ldots, u_N$ are orthonormal in $U$. This norm is well-defined, as the chosen decomposition on the left-hand side does not affect the value on the right-hand side. We refer to \cite[\S 9.1]{HNVWvolume2} for further details and useful properties. Here, let us only mention some simple  examples:
\begin{itemize}
  \item $\guy=\LL_2(U,Y)$ (the Hilbert--Schmidt operators) if $Y$ is Hilbert \cite[Prop.\ 9.1.9]{HNVWvolume2},
  \item $\gamma(U,L^p(S_1))\cong L^p(S_1;U)$ if $U$ is real \cite[p.\ 285]{HNVWvolume2}. Similar identifications can be made for Sobolev spaces. 
\end{itemize} 

If $Y$ has M-type 2, then a sufficient condition for stochastic integrability of  $\phi\col \bar{I}_T\times\Om\to \LL(U,Y)$ is that it is  $U$-strongly measurable, $U$-strongly $(\F_t)$-adapted and $f\in L^2(\bar{I}_T;\guy)$ a.s., see  \cite[p.\ 306]{NVW15}.  
If $Y$ is UMD, then we have different conditions, for which we need one more definition. 

\begin{definition}\label{def: UMD process2} 
Let the assumptions of Definition \ref{def: UMD process} hold and let $\psi\col \bar{I}_T\to \LL(U,Y)$ and $\phi\col \bar{I}_T\times\Om\to \LL(U,Y)$ be $U$-strongly measurable. 

We say that $\psi$ \emph{belongs scalarly to}  $L^2(I_T;U)$ if for all $y\in Y^*$, $\psi(\cdot)^*y^*\in L^2(I_T;U)$. We say that such $\psi$ \emph{represents} $R\in \gamma(L^2(I_T;U),Y)$ if 
\begin{equation}\label{eq: def represent determ}
\<R h,y^*\?=\int_0^T \<h(t),\psi(t)^*y^*\?_U\dd t\qquad \text{ for all }h\in L^2(I_T;U),y^*\in Y^*.
\end{equation} 

We say that $\phi$ \emph{belongs scalarly \as to} $L^2(I_T;U)$ if for all $y\in Y^*$, $\phi(\cdot,\om)^*y^*\in L^2(I_T;U)$ a.s.  
If $\phi$ belongs scalarly \as to $L^2(I_T;U)$ and $\xi\col \Om\to \gluy$ is strongly measurable, we say that $\phi$ \emph{represents} $\xi$ if for all $h\in L^2(I_T;U)$ and $y^*\in Y^*$:
\begin{equation}\label{eq: def represent}
\<\xi(\om) h,y^*\?=\int_0^T \<h(t),\phi(t,\om)^*y^*\?_U\dd t\qquad\text{for \aev }\om\in\Om. 
\end{equation}  
\end{definition} 

If  $Y$ is UMD, stochastic integrability   can now be characterized  as follows:  $\phi\col \bar{I}_T\times\Om\to \LL(U,Y)$ is stochastically integrable  if and only if it is $U$-strongly measurable, $U$-strongly $(\F_t)$-adapted, $\phi$ belongs scalarly \as to $L^2(I_T;U)$ and it represents an element $\xi\in L^0(\Om,\gamma(L^2(I_T;U),Y))$. 
More equivalent conditions for stochastic integrability  can be found in \cite[Th.\ 5.9]{NVW07}.

Lastly, let us  mention that  if $Y$ is UMD and has type 2, hence M-type 2, then we have  embeddings $L^2(I_T;\guy)\into \gamma(L^2(I_T;U),Y)$ for  $T\in(0,\infty)$.  

\begin{remark}\label{rem: represent}
Let $\psi$ and $\phi$ be as in Definition \ref{def: UMD process2}. If $\psi$ represents $R\in \gamma(L^2(\bar{I}_T;U),Y)$, then  $R$ is trivially unique by \eqref{eq: def represent determ}. Conversely, $\psi$ is $\mathrm{d}t$-\aev determined by $R$ through \eqref{eq: def represent determ}, 
so we may identify $\psi$ (as an element of $L^0({I}_T;\LL(U,Y))$) with $R$ and simply write
\[
\psi\in\gamma(L^2(I_T;U),Y). 
\]
This notation will be used to indicate that $\psi$ belongs scalarly to $L^2(I_T;U)$. 

Moreover, by \cite[Lem.\ 2.7, Rem.\ 2.8]{NVW07}, $\phi$ represents a strongly measurable $\xi\col \Om\to \gluy$ if and only if for \aev $\om\in\Om$ there exists a $\xi_{\om}\in \gamma(L^2(I_T;U),Y)$ such that $\phi(\cdot,\om)$ represents $\xi_{\om}$. 
\end{remark}

\section{Assumptions and solution notion}\label{sec:setting}

For our Yamada--Watanabe--Engelbert result, we will consider the following stochastic evolution equation in Banach spaces:
\begin{equation}\label{SPDE}
\dd u(t)=b(t,u)\dd t+\sigma(t,u)\dd W(t),\quad t\in \bar{I}_T. 
\end{equation}  
Moreover, for mild solutions,  we will study  
\begin{equation}\label{SPDE mild}
u(t)=S(t,0)u(0)+\int_0^t S(t,s){b}(s,u)\dd s+\int_0^t S(t,s){\sigma}(s,u)\dd W(s),\quad t\in \bar{I}_T,  
\end{equation}
for evolution families $(S(t,s))_{s,t\in \bar{I}_T, s\leq t}$ on Banach spaces.  In both equations, 
$W$ will be a $U$-cylindrical Brownian motion as defined in Definition \ref{def: cylindrical BM}.

In this section, we specify the underlying spaces and assumptions belonging to \eqref{SPDE} and \eqref{SPDE mild}. Furthermore, we will define our notion of solution in Definition \ref{def: SPDE C-sol}. This solution notion and the assumptions are formulated in such a way, that many settings are covered at once.
 
\begin{assumption}\label{ass}
The following hold:
  \begin{itemize}\setlength\itemsep{.1em}
  \item $T\in(0,\infty]$ 
  \item $Y$ and $Z$ are real separable Banach spaces and $Y\into Z$, 
  \item $\B$ is a Polish space such that $\B\into C({\bar{I}_T};Z)$,  
  \item $U$ is a real separable Hilbert space,  
  \item $(e_k)$ is a fixed orthonormal basis for $U$,  
  \item 
    \begin{align}
       &\begin{cases}
     b\col I_T\times\B\to Z \text{ is }\BB(I_T)\otimes\BB(\B)/\BB(Z)\text{-measurable},\\
     \sigma\col I_T\times\B\to \guy \text{ is }\BB(I_T)\otimes\BB(\B)/\BB(\guy)\text{-measurable},
    \end{cases}\label{eq:b sigma mble}\\[.2cm]
       &\begin{cases} 
     b(t,\cdot) \text{ is }\BB_t(\B)/\BB(Z)\text{-measurable for all }t\in I_T,\\
    \sigma(t,\cdot) \text{ is }\BB_t(\B)/\BB(\guy)\text{-measurable for all }t\in I_T,
    \end{cases}\label{eq:b sigma adapted}
          \end{align}
 \end{itemize}
 where
\begin{equation}\label{eq:def B_t(B^T)}
  \BB_t(\B)\ceqq \sigma(\pi_1^s:s\in \bar{I}_t), \qquad \pi^s_1\col \B\to Z\col v\mapsto v(s),
\end{equation}
identifying $v\in\B$ with its image in $C(\bar{I}_T;Z)$ under $\B\into C(\bar{I}_T;Z)$.  
\end{assumption}

The codomains of $b$ and $\sigma$ are separable, thus `measurable' in \eqref{eq:b sigma mble} and \eqref{eq:b sigma adapted} is equivalent to `strongly measurable', due to Pettis' lemma \cite[Prop.\ I.1.9]{vakhania87}. For $\sigma$, the measurability conditions in \eqref{eq:b sigma mble} and \eqref{eq:b sigma adapted} are also equivalent to having $\BB(I_T)\otimes\BB(\B)/\BB(Y)$-measurability of $\sigma(\cdot)u$ and $\BB_t(\B)/\BB(Y)$-measurability of $\sigma(t,\cdot)u$ for all $u\in U$, due to \cite[Lem.\ 2.5]{NVW07}.  

The time paths of solutions to the SPDE will lie \as in the space $\B$ (see Definition \ref{def: SPDE C-sol}). Typically, but not necessarily, $\B$ could be $C(\bar{I}_T;Z_1)\cap L^p_\loc(\bar{I}_T;Z_2)$ with $Z_1,Z_2\into Z$. Note that $b$ and $\sigma$ are defined on the whole space $\B$. However, in most applications, $b$ and $\sigma$ are of the \emph{Markovian form} of the next lemma. Here, the space $\tilde{X}$ is introduced to allow the coefficients to be only defined on a subspace of $Z$, provided that  $\B\subset L^0(I_T;\tilde{X})$ (\ie  solutions are \as $\tilde{X}$-valued).

\begin{lemma}\label{lem:coeffs markov} 
Let $T$, $Y$, $Z$, $\B$ and $U$ be as in Assumption \ref{ass}. Let $\tilde{X}$ be a Banach space such that $\tilde{X}\into Z$ and $\B\subset L^0(I_T;\tilde{X})$. 
Let 
\begin{equation}\label{eq: b sigma markov}
\begin{split}
 & b\col I_T\times \B\to Z,\, b(t,u)\ceqq \bar{b}\big(t,u(t)\one_{\tilde{X}}(u(t))\big),\\ 
 & \sigma\col I_T\times \B\to \guy,\, \sigma(t,u)\ceqq \bar{\sigma}\big(t,u(t)\one_{\tilde{X}}(u(t))\big), 
 \end{split}
\end{equation}
where
\begin{equation}\label{eq: bar b bar sigma}
\begin{split}
&\bar{b}\col I_T\times \tilde{X} \to Z\text{ is }\BB(I_T)\otimes \BB(\tilde{X})/\BB(Z)\text{-measurable},\\
&\bar{\sigma}\col I_T\times \tilde{X} \to \guy\text{ is }\BB(I_T)\otimes \BB(\tilde{X})/\BB(\guy)\text{-measurable}.
\end{split}
\end{equation}
Then  $b$ and $\sigma$ satisfy \eqref{eq:b sigma mble} and \eqref{eq:b sigma adapted}. 
\end{lemma}
\begin{proof}
We prove the claim only for $b$, as the proof for $\sigma$ goes the same. 

\eqref{eq:b sigma adapted}: Note that $b(t,u)=\bar{b}(t,\cdot\one_{\tilde{X}}(\cdot))\circ \pi_t^1(u)$. By Fubini and the joint measurability of $\bar{b}$, $\bar{b}(t,\cdot)$ is $\BB(\tilde{X})/\BB(Z)$-measurable. Since $\tilde{X}\into Z$, Kuratowski's theorem \cite[Th.\ 15.1]{kechris95} gives that $\bar{b}(t,\cdot\one_{\tilde{X}}(\cdot))$ is $\BB(Z)/\BB(Z)$-measurable. 
Also, $\pi^1_t$ is $\BB_t(\B)/\BB(Z)$-measurable, since $\BB_t(\B)$ is generated by $\pi^1_s$ for $s\in\bar{I}_t$. 
Hence $u\mapsto b(t,u(t)\one_{\tilde{X}}(u(t)))=\big(\bar{b}(t,\cdot\one_{\tilde{X}}(\cdot))\circ \pi_t^1(u)$ is $\BB_t(\B)/\BB(Z)$-measurable. 

\eqref{eq:b sigma mble}: Note that $b=\bar{b}\circ \phi$, where $\phi(t,u)\ceqq (t,u(t)\one_{\tilde{X}}(u(t)))$. Thus, by \eqref{eq: bar b bar sigma}, it suffices to prove that $\phi$ is $\BB(I_T)\otimes\BB(\B)/\BB(I_T)\otimes\BB(\tilde{X})$-measurable. 
Now, $\phi_1\col(t,u)\mapsto(t,u(t))$ is $\BB(I_T)\otimes\BB(\B)/\BB(I_T)\otimes\BB(Z)$-measurable (continuous even, using that $\B\into C({\bar{I}_T};Z)$).  
Furthermore, since $\tilde{X}\into Z$, Kuratowski's theorem gives that $\phi_2\col(t,x)\mapsto(t,x\one_{\tilde{X}}(x))$ is $\BB(I_T)\otimes\BB(Z)/\BB(I_T)\otimes\BB(\tilde{X})$-measurable.  
We conclude that $\phi=\phi_2\circ\phi_1$ is indeed $\BB(I_T)\otimes\BB(\B)/\BB(I_T)\otimes\BB(\tilde{X})$-measurable.
\end{proof}

We will define our notion of solution to \eqref{SPDE} and \eqref{SPDE mild} in Definition \ref{def: SPDE C-sol}. It permits the flexibility to choose a collection of conditions $\cond$ for the solution. These conditions will encode analytically strong, analytically weak, mild or weakly mild solutions, possibly with extra conditions regarding integrability in time or $\om$. 

Recall that $L^p_{\loc}(\bar{I}_T;X)=\{u\in L^0(\bar{I}_T; X) : u|_{{I}_t}\in L^p({I}_t;X) \,\forall\, t\in \bar{I}_T\}$, for $T\in(0,\infty]$. Of course, if $T<\infty$, then $L^p_{\loc}(\bar{I}_T;X)=L^p(\bar{I}_T;X)$.  

Let $D$ be any fixed subset of $Z^*$. The conditions that one may choose to require for solutions to   \eqref{SPDE} are as follows. 
\begin{enumerate}[label=(\arabic*),ref=(\arabic*)]  
 \item\label{it:sol integrability} $Y$ has M-type 2 and $\P$-\as $$ 
   b(\cdot,u)\in L^1_\loc({\bar{I}_T};Z), \quad \sigma(\cdot,u)\in L^2_\loc({\bar{I}_T};\guy). $$ 
 \item\label{it:sol integrability UMD} $Y$ is UMD and $\P$-a.s.: for all $t\in\bar{I}_T$,
 $$
  b(\cdot,u)\in L^1_\loc({\bar{I}_T};Z), \quad \sigma(\cdot,u)\in \gamma(L^2(I_t;U),Y)$$
   \item\label{it:sol integrability ana weak} For all $z^*\in D\subset Z^*$: $\P$-a.s.
 \begin{equation*} 
   \<b(\cdot,u),z^*\?\in L^1_\loc({\bar{I}_T}), \quad z^*\circ\sigma(\cdot,u)\in L^2_\loc({\bar{I}_T};U^*).
   \end{equation*}
   \item\label{it:integrated eq} \ref{it:sol integrability} or \ref{it:sol integrability UMD} holds and $\P$-a.s. 
 \begin{equation*} 
     u(\cdot) =u(0)+\int_0^\cdot b(s,u)\dd s+\int_0^\cdot \sigma(s,u)\dd {W}(s)\,\text{ in $C({\bar{I}_T};Z)$.}
 \end{equation*} 
 \item\label{it:integrated eq ana weak} \ref{it:sol integrability} or \ref{it:sol integrability UMD} or \ref{it:sol integrability ana weak} holds and for all $z^*\in D\subset Z^*$, $\P$-a.s. 
\begin{equation*} 
  \<u(\cdot),z^*\? =\<u(0),z^*\?+\int_0^\cdot \<b(s,u),z^*\?\dd s+\int_0^\cdot z^*\circ \sigma(s,u)\dd {W}(s)\,\text{ in $C({\bar{I}_T})$.}
  \end{equation*}
   \end{enumerate}  

Besides equation \eqref{SPDE}, we will study \eqref{SPDE mild} for mild solutions. 
For either equation, one may include moment conditions and \as conditions for solutions. Let $f\col\B\to \R$ be  measurable and let $p\in(0,\infty)$. Let $E\in\BB(\B)$. 
\begin{enumerate}[label=(\arabic*),ref=(\arabic*),resume]
\item\label{it:moment cond} $f(u)\in L^p(\Om)$.  
\item\label{it:as cond} $\P(u\in E)=1$.
\end{enumerate}   
A typical example of \ref{it:moment cond} is the condition `$u\in L^p(\Om;\tilde{\B})$', for a normed space $\tilde{\B}$ with $\B\into \tilde{\B}$, choosing $f(u)=\|u\|_{\tilde{\B}}$. See Example \ref{ex:schrodinger}. 

For $S$ appearing in \eqref{SPDE mild}, we assume the following. 
\begin{definition}\label{def: evol fam}  
Let $Z$ be a Banach space and let $T\in(0,\infty]$. We call $S=(S(t,s))_{s,t\in \bar{I}_T, s\leq t}$ a \emph{strongly measurable evolution family} on $Z$ if $S\col \Lambda\to \LL(Z)$ and
\begin{align}
 &S(\cdot,\cdot)x\col \Lambda\to Z\text{ is  Borel measurable for all }x\in Z, \qquad \Lambda\ceqq\{(t,s)\in \bar{I}_T\times \bar{I}_T:s\leq t\}, \label{eq: mble evol fam} \\
&S(s,s)=\Id_Z,\, S(t,r)S(r,s)=S(t,s)  \text{ for all }s\leq r\leq t. \notag
\end{align} 
\end{definition}
If $S$ is a strongly measurable evolution family on $Z$ and if Assumption \ref{ass} holds, then we understand $S(t,s)\sigma(s,u)$ as the composition $S(t,s)|_Y\circ \sigma(s,u)\in \gamma(U,Z)$. 

Let $S$ be a strongly measurable evolution family and let $D$ be any fixed subset of $Z^*$. 
The following conditions can be used to work with mild solutions. 
\begin{enumerate}[label=(\arabic*),ref=(\arabic*),resume]
\item\label{it:sol integrability mild} $Y$ has M-type 2 and $\P$-a.s.: for all $t\in \bar{I}_T$,
\begin{align*} 
S(t,\cdot){b}(\cdot,u) \in L^1(I_t;Z), \quad S(t,\cdot){\sigma}(\cdot,u) \in  L^2(I_t;\guz).
\end{align*} 
   \item\label{it:sol integrability mild UMD} $Y$ is UMD and $\P$-a.s.: for all $t\in \bar{I}_T$,
\begin{align*} 
S(t,\cdot){b}(\cdot,u) \in L^1(I_t;Z), \quad S(t,\cdot){\sigma}(\cdot,u) \in \gamma(L^2(I_t;U),Z).
\end{align*} 
 \item\label{it:sol integrability mild weak} For all $z^*\in D\subset Z^*$, $\P$-a.s.: for all $t\in \bar{I}_T$,
\begin{align*} 
\<S(t,\cdot){b}(\cdot,u),z^*\? \in L^1({I_t}), \quad z^*\circ S(t,\cdot){\sigma}(\cdot,u) \in L^2(I_t;U^*).
\end{align*} 
\item\label{it:integrated eq mild} \ref{it:sol integrability mild} or \ref{it:sol integrability mild UMD} holds and $\P$-a.s.: for all $t\in \bar{I}_T$, 
    \begin{equation*} 
    u(t)=S(t,0)u(0)+\int_0^t S(t,s){b}(s,u)\dd s+\int_0^t S(t,s){\sigma}(s,u)\dd W(s) \text{ in $Z$}.  
    \end{equation*}
\item\label{it:integrated eq mild weak} \ref{it:sol integrability mild} or \ref{it:sol integrability mild UMD} or \ref{it:sol integrability mild weak}  holds and for all $z^*\in D\subset Z^*$, $\P$-a.s.: for all $t\in \bar{I}_T$, 
    \begin{equation*} 
    \<u(t),z^*\?=\<S(t,0)u(0),z^*\?+\int_0^t \<S(t,s){b}(s,u),z^*\?\dd s+\int_0^t \<S(t,s){\sigma}(s,u),z^*\?\dd W(s). 
    \end{equation*}
\end{enumerate}  

\begin{remark}
Of course, there are several implications amongst the above conditions. For example, strong integrability implies weak integrability: \ref{it:sol integrability}$\vee$\ref{it:sol integrability UMD}$\Rightarrow$\ref{it:sol integrability ana weak}. Moreover, strong integrability and satisfaction of the weak equations with $D=Z^*$ implies satisfaction of the strong equation:  [\ref{it:sol integrability}$\vee$\ref{it:sol integrability UMD}]$\wedge$\ref{it:integrated eq ana weak}$\Rightarrow$\ref{it:integrated eq}, by taking a sequence in $Z^*$ that separates the points, as we will do in the proof of Theorem \ref{th:stoch int rep type 2 or UMD}. 
If $Y$ is a Hilbert space, then \ref{it:sol integrability} and \ref{it:sol integrability UMD} are equivalent. Lastly,  if $Y$ is both UMD and of type 2 (hence of M-type 2), then we have the implication \ref{it:sol integrability}$\Rightarrow$\ref{it:sol integrability UMD}. 
Analogous implications hold amongst \ref{it:sol integrability mild}--\ref{it:integrated eq mild weak}.
\end{remark} 

The reason for separating the conditions is that our (variable) notion of solution will directly correspond to standard notions of solutions, while taking hardly any extra effort in the proofs. These standard notions are not always equivalent (see Remark \ref{rem:not equiv}).

Let us now introduce our variable notion of solution to \eqref{SPDE} and \eqref{SPDE mild}. Here,  $\cond$ is in theory allowed to be any of the indicated subcollections. However, to make sense as a `solution',   $\cond$ should contain at least one of the equations \ref{it:integrated eq}, \ref{it:integrated eq ana weak}, \ref{it:integrated eq mild} or \ref{it:integrated eq mild weak}. 
 
\begin{definition}\label{def: SPDE C-sol}
 Let Assumption \ref{ass} hold. Let $\cond$ be a subcollection of conditions \ref{it:sol integrability}-\ref{it:as cond} or of \ref{it:moment cond}-\ref{it:integrated eq mild weak}. 
 We call $(u,W,(\Om,\F,\P,(\F_t))$ a \emph{$\cond$-weak solution} to \eqref{SPDE}, respectively \eqref{SPDE mild}, if:  
 \begin{itemize}
   \item $(\Om,\F,\P,(\F_t))$ is a filtered probability space,
   \item $W$ is a $U$-cylindrical Brownian motion on $(\Om,\F,\P,(\F_t))$,
   \item $u\col {\bar{I}_T}\times \Om\to Z$ is $\BB({\bar{I}_T})\otimes\F/\BB(Z)$-measurable and $(\F_t)$-adapted,  
   \item $\P$-\as $u\in\B$,
   \item $\cond$ holds. 
 \end{itemize}

We call $(u,W,(\Om,{\F},{\P},({\F_t})))$ a \emph{$\cond$-strong solution} if it is a $\cond$-weak solution and if $u=F(u(0),W)$ $\P$-\as for some measurable map $F\col Z\times\W\to \B$, where $W=(W^k)$ through \eqref{eq: def Rinfty BM} and $\W\ceqq C(\bar{I}_T;\R^\infty)$. Moreover, we say that:
\begin{itemize}
\item \emph{pathwise uniqueness} holds if for every two 
$\cond$-weak solutions 
$(u,W,(\Om,\F,\P, (\F_t)))$ and $(v,W,(\Om,\F,\P, (\F_t)))$ with $u(0)=v(0)$ $\P$-a.s.: $\P$-a.s. $u=v$ in $\B$,
\item \emph{joint weak uniqueness} holds if for every two $\cond$-weak solutions 
$(u,W,(\Om,\F,\P,(\F_t)))$ and $(v,W',(\Om',\F',\P', (\F'_t)))$ with $\law(u(0))=\law(v(0))$: $\law(u,W)=\law(v,W')$,
\item\emph{weak uniqueness} holds if for every two $\cond$-weak solutions $(u,W,(\Om,\F,\P,(\F_t)))$ and

\noindent
 $(v,W',(\Om',\F',\P', (\F'_t)))$ with $\law(u(0))=\law(v(0))$: $\law(u)=\law(v)$.  
\end{itemize} 
\end{definition}

Many different definitions of `strong solution' appear in the literature. 
The above notion of `$\cond$-strong solution' is analogous to the notion of `strong solution' used by Kurtz. The analogy makes the proof of the Yamada--Watanabe--Engelbert theorem (Theorem \ref{th:YW SPDE}) more transparant. Still, we will obtain `strong existence' in a much stronger sense in Theorem \ref{th:YW SPDE} and Corollary \ref{cor: single F}.  

\begin{example}\label{ex:choices C}
The following choices of $\cond$ correspond to standard notions of solutions. In the lines with two options, the first is suited if $Y$ has M-type 2 and the second is suited if $Y$ is UMD.
\begin{itemize}
  \item $\cond=\{\ref{it:sol integrability},\ref{it:integrated eq}\}$ or $\cond=\{\ref{it:sol integrability UMD},\ref{it:integrated eq}\}$: analytically strong solution.
  \item $\cond=\{\ref{it:sol integrability ana weak},\ref{it:integrated eq ana weak}\}$, $D=Z^*$: analytically weak solution.
  \item $\cond=\{\ref{it:sol integrability},\ref{it:integrated eq ana weak}\}$ or $\cond=\{\ref{it:sol integrability UMD},\ref{it:integrated eq ana weak}\}$, $D=Z^*$: analytically weak solution, stronger integrability.   
  \item $\cond=\{\ref{it:sol integrability mild}, \ref{it:integrated eq mild}\}$ or $\cond=\{\ref{it:sol integrability mild UMD}, \ref{it:integrated eq mild}\}$: mild solution.
  \item $\cond=\{\ref{it:sol integrability mild weak},\ref{it:integrated eq mild weak}\}$, $D=Z^*$: weakly mild solution.
  \item $\cond=\{\ref{it:sol integrability mild},\ref{it:integrated eq mild weak}\}$ or $\cond=\{\ref{it:sol integrability mild UMD},\ref{it:integrated eq mild weak}\}$, $D=Z^*$: weakly mild solution, stronger integrability.
  \item $\cond=\{\ref{it:sol integrability},\ref{it:integrated eq}\}$  
  with $Z=V^*$, $\B\into L^0(I_T;V)\cap C(\bar{I}_T;H)$, $H$ a Hilbert space and $(V,H,V^*)$ a Gelfand triple: variational framework solution. 
\end{itemize}  
\end{example} 

\begin{remark}\label{rem:not equiv}
We included all these different solution notions because they cannot be used interchangeably in general. 
For example, the analytically weak and weakly mild notions can be used in a Banach space $Y$ without additional geometrical structure, but for the analytically strong and mild solution notions one does need structure of $Y$ so that the stochastic integrals can be defined. 
Secondly, the notions of mild and weakly mild solutions only make sense if part of the coefficient $b$ has  some linear structure, while this is not needed for analytically strong and analytically weak solutions.  
Lastly, even for semilinear equations in an M-type 2 or UMD space $Y$, to get equivalence between analytically weak and analytically strong or mild solutions, one may need to assume extra structure of the evolution family/semigroup or generator thereof, see \eg \cite[Prop.\ 6.9,  Hyp.\ 6.5]{kunze13} and  \cite[Prop.\ 7.5.12, (C2)]{veraarphdthesis}.
\end{remark}

Let us emphasize that many `\as conditions' can be modeled by the choice of the Polish space $\B$.  Recall that a $\cond$-weak solution has to satisfy $u\in \B$ a.s. The space $\B$ need not be linear.  Time integrability conditions can be put into $\B$, but also positivity conditions, by taking an intersection with a space of  positive valued functions (Example \ref{ex:positive}). In general, it may be useful to recall that an open or closed subset of a Polish space is again Polish.  Additional `\as conditions' can be added through \ref{it:as cond}, such as weak continuity conditions (Example \ref{ex:schrodinger}). 

Note that colorings of the noise can be included through $\sigma$ (see \eg Example \ref{ex:ondrejat setting}).

A few technical comments on the notion of $\cond$-weak solution are in order. 

\begin{remark}
The notations `$\in\gamma(L^2(I_t;U),Y)$' and `$\in\gamma(L^2(I_t;U),Z)$' in conditions \ref{it:sol integrability UMD} and \ref{it:sol integrability mild UMD} are meant in the sense of Remark \ref{rem: represent}.  The second part of Remark \ref{rem: represent} gives that if \ref{it:sol integrability UMD} or  \ref{it:sol integrability mild UMD} holds, then the stochastic process $\sigma(\cdot,u)$,  respectively  $S(t,\cdot)\sigma(\cdot,u)$,  represents a strongly measurable $\xi\col \Om\to \gluy$. 

The UMD requirement for $Y$ in \ref{it:sol integrability UMD} and \ref{it:sol integrability mild UMD}  can be weakened a bit,   requiring that $Y$ is one of the slightly more general spaces studied in \cite{coxveraar11,coxgeiss21} (including $L^1$, which is not UMD \cite[Ex.\ 4.7]{coxveraar11}). 
The stochastic integration aspects needed for the results in this paper carry over to such spaces. 
\end{remark}

\begin{remark}
In conditions \ref{it:integrated eq}, \ref{it:integrated eq ana weak}, \ref{it:integrated eq mild} and \ref{it:integrated eq mild weak}, the stochastic integrals denote those constructed on $(\Om,\overline{\F},\overline{\P},(\overline{\F_t}))$. Note that $\sigma(s,u)$ is not everywhere defined on $\Om$, but it is indistinguishable from $\sigma(s,u\one_{\B}(u))$, which is  joint measurable and $(\overline{\F_t})$-adapted by the upcoming Lemma \ref{lem: mble L^p subsets}. By Lemma \ref{lem:sol on completion}, $W$ is also a $U$-cylindrical Brownian motion on the completed filtered probability space. Combined with the assumed integrability of the coefficients and using Corollary \ref{cor: mble L^p subsets}, we conclude that all stochastic integrals are well-defined on $(\Om,\overline{\F},\overline{\P},(\overline{\F_t}))$.  
\end{remark}

\begin{remark}
Sometimes, solutions $u$ are defined to be progressively measurable, so that $b(\cdot,u)$ and $\sigma(\cdot,u)$ are progressively measurable. However, the weaker condition in Definition \ref{def: SPDE C-sol} of $u$ being only measurable and adapted is equivalent.  
Again, one can turn to $u\one_\B(u)$, which is continuous everywhere and $(\overline{\F_t})$-adapted and thus  $(\overline{\F_t})$-progressively measurable (Lemma \ref{lem: mble L^p subsets}).
\end{remark}

\begin{remark}\label{rem: L(U,Y) valued}
  In the case of analytically weak solutions and mild weak solutions with only weak integrability conditions for the coefficients (\ie \ref{it:sol integrability ana weak} or \ref{it:sol integrability mild weak}), Assumption \ref{ass} can be weakened a bit: $\sigma$ is allowed to take values in $\LL(U,Y)$ instead of $\gamma(U,Y)$ and the measurability for $\sigma\col I_T\times\B\to\LL(U,Y)$ in \eqref{eq:b sigma mble} and \eqref{eq:b sigma adapted} needs to hold $U$-strongly, since then, $z^*\circ \sigma\col I_T\times \B\to U^*$ has the relevant measurability properties. 
\end{remark}

\begin{remark}\label{rem: integrability coeffs flex} 
One can add more integrability conditions to $\cond$ besides \ref{it:sol integrability}--\ref{it:sol integrability ana weak} and \ref{it:sol integrability mild}--\ref{it:sol integrability mild weak}. For example, if $S_b$ and $S_\sigma$ are Polish spaces such that $S_b\into L^0({\bar{I}_T};Z)$ and $S_\sigma\into L^0({\bar{I}_T};\guy)$, one can add the condition that $\P$-a.s.: $b(\cdot,u)\in S_b$ and $\sigma(\cdot,u)\in S_\sigma$. Analogously, one can require that $\P$-a.s.:  $S(t,\cdot)b(\cdot,u)\in S_b^t\into L^0({\bar{I}_t};Z)$ and $S(t,\cdot)\sigma(\cdot,u)\in S_\sigma^t\into L^0({\bar{I}_t};\guy)$ for $t\in \bar{I}_T$ and Polish spaces $S_b^t$, $S_\sigma^t$.  
The proofs of Lemma \ref{lem: mble L^p subsets} and Corollary \ref{cor: mble L^p subsets}, hence their application in the proof of Theorem \ref{th:YW SPDE}, remain valid.  
\end{remark} 

Now let us give some examples of settings that fit in our framework.  
Of course, many more settings are possible. New applications were more generally described in  \ref{it:intro1}--\ref{it:intro4} in the introduction. 

\begin{example}\label{ex:applic}  
The following settings fit Assumption \ref{ass}  and Definition \ref{def: SPDE C-sol}. Hence, the Yamada--Watanabe--Engelbert theorem (Theorem \ref{th:YW SPDE}) and the classical Yamada--Watanabe theorem (Corollary \ref{cor: single F}) apply. 
\begin{enumerate}[label=\alph*),ref=\ref{ex:applic}\alph*)]
\item\label{ex:variational} The variational setting of the Yamada--Watanabe theorem \cite{rock08}. Assumption \ref{ass} holds with   
$T\ceqq \infty$,  $\B\ceqq C(\bar{I}_\infty;H)\cap L_{\loc}^1(\bar{I}_\infty;V)$, $Y\ceqq H$, $Z\ceqq E$, and we choose $\cond\ceqq\{\ref{it:sol integrability},\ref{it:integrated eq}\}$. 
  \item\label{ex:ana strong}   The strong analytic settings of \cite[\S3, \S4]{AV22nonlinear2} for nonlinear parabolic stochastic evolution equations in Banach spaces. Therein, $X_0$ and $X_1$ are UMD Banach spaces with type 2 such that $X_1\into X_0$ and $T\in(0,\infty]$. In practice, one can additionally assume that $X_0$ and $X_1$ are separable, since solutions are defined to be strongly measurable. 
Now, for \cite[Ass.\ 3.1, Ass.\ 3.2, Def.\ 3.3]{AV22nonlinear2}, we put $\cond\ceqq\{\ref{it:sol integrability},\ref{it:integrated eq}\}$,  $\X\ceqq (X_0,X_1)_{1/2,2}$ and
\begin{align*}
&\B\ceqq C(\bar{I}_T;X_0)\cap L_{\loc}^2(\bar{I}_T;X_1),\quad  Y\ceqq \X, \quad Z\ceqq X_0 
\end{align*} 
Similarly, for \cite[\S4 Hyp.\ (H), Def.\ 4.3]{AV22nonlinear2}, we put $X^{\Tr}_{\kappa,p}\ceqq(X_0,X_1)_{1-\frac{1+\kappa}{p},p}$ and
\begin{align}\label{eq: choices weighted Lp}
&\B\ceqq C(\bar{I}_T;X^{\Tr}_{\kappa,p})\cap L_{\loc}^p(\bar{I}_T,w_\kappa;X_1),\quad Y\ceqq \X, \quad Z\ceqq X_0,
\end{align}  
where  $p\in(1,\infty)$, $\kappa\in[0,p-1)$ and $w_\kappa(s)\ceqq s^\kappa$ is a weight. In the setting here, more choices of $p,\kappa$ and type of weights are allowed. 
We refer to \cite[\S2.2, \S3]{AV22nonlinear1} for details on the real interpolation spaces  $\X$ and $X^{\Tr}_{\kappa,p}$.  

The Banach spaces $Y$ and $Z$ are separable by the separability of $X_1$ and $X_0$. 
Note that $\B$ is a Polish space, in the second case completely metrized by  
\[
\textstyle{\rho(u,v)\ceqq \sum_{k=1}^\infty 2^{-k}\big[ \big(\|u-v\|_{C(\bar{I}_{k\wedge T};X^{\Tr}_{\kappa,p})}+\|u-v\|_{L^p(I_{k\wedge T},w_{\kappa};X_1)} \big)\wedge 1 \big], }
\]
and similarly for the first case. The coefficients in \cite[\S4 Hyp.\ (H)]{AV22nonlinear2} are of the Markovian form of Lemma \ref{lem:coeffs markov} (with $\tilde{X}\ceqq X_1$), thus fit in our framework.  
 
In \cite[Def.\ 4.3]{AV22nonlinear2}, the integrability conditions in the solution notion  are formulated slightly differently, \eg one requires $F(\cdot,u)\in L_\loc^p(\bar{I}_T;w_\kappa;Z)$ and $G(\cdot,u)\in L_\loc^p(\bar{I}_T;w_\kappa; \guy)$. Such conditions are covered by Remark \ref{rem: integrability coeffs flex}, and can be added to $\cond=\{\ref{it:sol integrability},\ref{it:integrated eq}\}$.

\item\label{ex:schrodinger} SPDEs with less classical solution spaces, such as wave equations, Schr\"odinger equations and other equations for which Strichartz estimates are used.  
When the solution space is not standard, there is often no directly applicable reference for the Yamada--Watanabe theorem.

For example, for the Schr\"odinger equation in \cite[(2.3), Def.\ 2.2]{brzezniak22}, solutions  are  required to be weakly continuous. Our setting and results apply directly. We let  $\B\ceqq C(\bar{I}_T;Z)\cap L^2(\bar{I}_T;Y)$, with $T<\infty$ and $Z\ceqq H^{-1}(M;\C)$ and $Y\ceqq H^1(M;\C)$ real Banach spaces (as in Assumption \ref{ass}), see also \cite[\S 2.5, \S 7]{cazenaveharaux}. 
As before,  $b$ and $\sigma$ are of the form \eqref{eq: b sigma markov}. 

Since   solutions in \cite{brzezniak22} are required to be in $L^2(\Om;L^2(\bar{I}_T;Y))$, we use condition \ref{it:moment cond} with $p=2$ and $f(u)=\|u\|_{\tilde{\B}}$, where $\B\into\tilde{\B}\ceqq L^2(\bar{I}_T;Y)$ . 

To encode weak continuity, we use condition \ref{it:as cond} with $E\ceqq C(\bar{I}_T;Y_{\mathrm{w}})\cap \B$, where $Y_{\mathrm{w}}$ denotes $Y$ equipped with the weak topology. We need to argue why $E\in \BB(\B)$. 
Note that $C(\bar{I}_T;Y_{\mathrm{w}})$ is not metrizable with respect to the compact-open topology, but for each $a>0$, 
\[
\textstyle{K_a\ceqq \{f\in C(\bar{I}_T;Y_{\mathrm{w}}):\sup_{t\in[0,T]}\|f(t)\|_Y\leq a\}}
\]
is Polish with the subspace topology, since the closed unit ball in $Y$ with the weak topology is Polish, see also \cite[Rem.\ 4.2]{brzezniakondrejatseidler}. 
Moreover, using Rellich's theorem ($M$ is a compact smooth manifold), one can show that 
\[
K_a\into L^0(\bar{I}_T;Z) \; \text{ and }\; C(\bar{I}_T;Y_{\mathrm{w}})=\cup_{a\in\N}K_a.
\]
The above implies $C(\bar{I}_T;Y_{\mathrm{w}})\in \BB(L^0(\bar{I}_T;Z))$, applying Kuratowski's theorem to each Polish space $K_a$ and using Remark \ref{rem:L0}. 
Also, $\B\into L^0(\bar{I}_T;Z)$, which gives $E=C(\bar{I}_T;Y_{\mathrm{w}})\cap \B\in\BB(\B)$ (as a preimage under the embedding), as required. 
Put $\cond=\{\ref{it:sol integrability},\ref{it:integrated eq},\ref{it:moment cond},\ref{it:as cond}\}$. 
  
\item\label{ex:positive}
SPDEs with positivity conditions.  Condition \ref{it:as cond} can be used to encode positivity conditions, or the Polish space $\B$ can be chosen as an intersection with a space of positive functions. 
For the second strategy, recall that an open or closed subset of a Polish space is again Polish.  
One may \eg consider the positive solutions to the stochastic thin-film equations in  \cite{agrestisauerbrey24}. The setting fits by choosing spaces as in \eqref{eq: choices weighted Lp}, but now we intersect $\B$ with $C([0,T];C(\mathbb{T}^d;(0,\infty))$, using that the latter is an  open subset of the Polish space $C([0,T];C(\mathbb{T}^d))$,   and  $\B\into C([0,T];C(\mathbb{T}^d))$ in this setting.

More generally, for  $T\in(0,\infty)$, $K\subset \R^d$ compact and $A\subset \R$ open, the set $C(\bar{I}_T;C(K;A))$ is open in in the Polish space $C(\bar{I}_T;C(K))$. Now, if ${\B}$ is a Polish space such that  ${\B}\into C([0,T];C(K))$, then one can use the new Polish space ${\B}\cap C(\bar{I}_T;C(K;A))$ (as an open subset of ${\B}$). 
 
Similarly, for  $T\in(0,\infty]$, $\mathcal{O}\subset \R^d$ any subset and $A\subset \R$ closed, $C(\bar{I}_T;C(\mathcal{O};A))$ is closed in $C(\bar{I}_T;C(\mathcal{O}))$ with the compact-open topology. 
Indeed, convergence of $(f_n)\subset C(\bar{I}_T;C(\mathcal{O};A))$ in $C(\bar{I}_T;C(\mathcal{O}))$ implies convergence of $(f_n(t)(x))$ for all $t\in \bar{I}_T, x\in \mathcal{O}$, thus limits are $A$-valued since $A$ is closed. 
If ${\B}$ is a Polish space such that ${\B}\into C(\bar{I}_T;C(\mathcal{O}))$, then one may use the Polish space ${\B}\cap  C(\bar{I}_T;C(\mathcal{O};A))$ (as a closed subset of $\B$). If $\mathcal{O}\subset \R^d$ is locally compact (e.g.\ closed or open), then $C(\bar{I}_T;C(\mathcal{O};A))$ is itself Polish (see \cite[Ex.\ A.10]{kerrli}).  
  
\item Analytically strong solutions to SPDEs with space-time white noise. The spaces $\B$, $Y$ and $Z$ in Assumption \ref{ass} also allow us to work with fractional Sobolev spaces: we may use $X_0\ceqq H^{-1-s,q}(\mathbb{T})$, $X_1 \ceqq H^{1-s,q}(\mathbb{T})$ with $s\in(0,1)$ and $q\in[2,\infty)$,  
      and define $\B,Y,Z$  as in \eqref{eq: choices weighted Lp} with $p\in(2,\infty)$ and $\kappa\in[0,\frac{p}{2}-1)$. We let  $\cond\ceqq\{\ref{it:sol integrability},\ref{it:integrated eq}\}$ and may add additional integrability conditions for the coefficients by Remark \ref{rem: integrability coeffs flex}. This covers \eg the equations with space-time white noise  in  \cite[\S 5.5]{AV22nonlinear1}. 
\item Analytically weak solutions such as those of Kunze's work \cite{kunze13}. See Example  \ref{ex:kunze setting}.  
\item Mild solutions to SPDEs in 2-smooth Banach spaces as in Ondrej\'at's work \cite{ondrejat04}. See  Example \ref{ex:ondrejat setting}.   
\item Mild solutions to SPDEs in Banach spaces with space-time white noise. For examples in M-type 2 spaces, see \cite[\S6.2]{brzezniak97}. Examples in UMD spaces are the parabolic equations in \cite[p.\ 983]{NVW08} and the general setting therein \cite[p.\ 969, p.\ 957]{NVW08}.  
  Let $E$ be a separable UMD Banach space and put $\cond\ceqq\{\ref{it:sol integrability mild UMD}, \ref{it:integrated eq mild}\}$ and 
  \begin{align*}
  \B\ceqq C(\bar{I}_{T_0};E)\cap L_{\loc}^1(\bar{I}_{T_0};E_\eta),\quad  Y\ceqq E_{-\theta_B}, \quad Z\ceqq E_{-(\theta_F\vee\theta_B)}. 
\end{align*}
Here, $Z$ is chosen such that the setting fits Assumption \ref{ass}. In \cite{NVW08} the coefficient $b$ is originally measurable as an $E_{-\theta_F}$-valued map, but since  $E_{-\theta_F}\into Z$ is Borel measurable, it is also measurable as a $Z$-valued map.  
\end{enumerate} 
\end{example}

We conclude this section with some words on the choice of path space $\W$ for the $U$-cylindrical Brownian motion. First of all, one may use $\W_0\ceqq \{w\in\W:w(0)=0\}$ instead of $\W$ in the notion of $\cond$-strong solution. This is unimportant: $W\in\W_0$ \as and  $\W_0\into \W$. Continuity of the embedding implies that $F|_{ Z\times \W_0}$ is measurable whenever $F\col Z\times\W\to\B$ is measurable. Conversely, if $F_0\col Z\times \W_0\to\B$ is measurable, then so is the trivial extension $F_0\one_{Z\times\W_0}\col Z\times\W\to \B$, by Kuratowski's theorem. A slightly more important choice is the following.

\begin{remark}\label{rem: Q-wiener}
Alternatively, one can view $W$ as a random variable in $\W_1\ceqq C(\bar{I}_T;U_1)$, for some larger separable Hilbert space $U_1\supset U$, such that $U\into U_1$ through a Hilbert--Schmidt embedding $J\in\mathcal{L}_2(U,U_1)$. 
If we fix an orthonormal basis  $(e_k)$ for  $U$, and we let $W^k$ be a continuous modification of $W(\one_{(0,\cdot]}\otimes e_k)$ (see \eqref{eq: def Rinfty BM} and below), then 
\begin{equation}\label{eq: def W_1}
  {W_1}(t)\ceqq \sum_{k=1}^{\infty}W^k(t)Je_k
\end{equation}
defines a $Q_1$-Wiener process in $U_1$ with $Q_1=JJ^*\in\LL(U_1)$ nonnegative definite, symmetric 
and of trace class, see \cite[Prop.\ 2.5.2]{liurockner15}.  
Using Doob's maximal inequality, one can show that for all $T_0\in(0,\infty)$, the series in \eqref{eq: def W_1} converges in $L^2(\Om;C([0,T_0];U_1))$ (see \cite[p.\ 51]{liurockner15}). This ensures that $W_1$ is a.s.\ continuous and that the series converges in probability in $C([0,T_0];U_1)$. Since each $W^k(\cdot)Je_k$ is a symmetric $C([0,T_0];U_1)$-valued random variable,  the series in \eqref{eq: def W_1} even converges a.s.\ in $C([0,T_0];U_1)$ by the It\^o--Nisio theorem \cite[Cor.\ 6.4.2]{HNVWvolume2}, and hence a.s.\ in $\W_1=C(\bar{I}_T;U_1)$ (taking $T_0\in\N$ if $T=\infty$). 

Equipping $\W_1$ with the topology of uniform convergence on compact sets, it becomes a Polish space, completely metrized by $
d(w,z)\ceqq \sum_{k=1}^\infty 2^{-k}\left[\|w-z\|_{C(\bar{I}_{k\wedge T};U_1)} \wedge 1 \right]$. 
Now, $W_1\col\Om\to \W_1$ is a random variable by Lemma \ref{lem: generating maps}: $W_1(t)\col \Om\to U_1$ is a random variable for each $t\in\bar{I}_T$ 
and the continuous maps $\pi_{U_1}^t$ (see \eqref{eq: point evaluation}) separate the points in $\W_1$. 

The results in this paper also hold for $W,\W$ replaced by $W_1,\W_1$ respectively, as we will explain in Remark \ref{rem: mathbbW choice}.  
\end{remark}

\section{The Yamada--Watanabe--Engelbert theorem for SPDEs}\label{sec:YW SPDE}

We now turn to the main goal of this paper: to  prove the Yamada--Watanabe--Engelbert theorem for SPDE solutions as defined in Definition \ref{def: SPDE C-sol}. Our result is stated below.  Here, in part  \ref{it:YW_Fmu}, $\wien$ denotes the unique law of a continuous $\R^\infty$-Brownian motion (Remark \ref{rem:wien measure}) and the indicated $\sigma$-algebras will be defined in \eqref{eq:temporal}.  

As before, $\cond$ is allowed to be any  subcollection of  \ref{it:sol integrability}-\ref{it:as cond} or of \ref{it:moment cond}-\ref{it:integrated eq mild weak}. However, the theorem is only useful if $\cond$ contains at least one equation, \ie condition \ref{it:integrated eq}, \ref{it:integrated eq ana weak}, \ref{it:integrated eq mild} or \ref{it:integrated eq mild weak}. 
Applications are explained in Examples \ref{ex:choices C} and \ref{ex:applic}.

\begin{theorem}\label{th:YW SPDE}
Suppose that Assumption \ref{ass} holds. Let $\cond$ be a subcollection of conditions \ref{it:sol integrability}-\ref{it:as cond}. Let $\mu\in\PP(Z)$. 
The following are equivalent:  
\begin{enumerate}[label=\textup{(\alph*)},ref=\textup{(\alph*)}] 
\item \label{it:YW1SPDE} 
    There exists a $\cond$-weak solution to $\eqref{SPDE}$ with $\law(u(0))=\mu$ and pathwise uniqueness holds for $\cond$-weak solutions with inital law $\mu$.  
\item \label{it:YW2SPDE} There exists a $\cond$-strong solution to $\eqref{SPDE}$ with $\law(u(0))=\mu$ and joint weak uniqueness holds for $\cond$-weak solutions with initial law $\mu$.
\item \label{it:YW3SPDE} Joint weak uniqueness holds for  $\cond$-weak solutions with initial law $\mu$ and `strong existence holds': for any filtered probability space $(\Om,\F,\P, (\F_t))$ and any $U$-cylindrical Brownian motion $W$ on $(\Om,\F,\P, (\F_t))$, there exists $u$ such that $(u,W,(\Om,\overline{\F},\overline{\P}, (\overline{\F_t}))$ is a $\cond$-weak solution with initial law $\mu$.   
    
    Even more: there exists a single measurable map $F_\mu\col Z\times \W\to\B$ such that for any filtered probability space $(\Om,\F,\P, (\F_t))$, for any $U$-cylindrical Brownian motion $W$ on $(\Om,\F,\P, (\F_t))$ and for any $u_0\in  L^0_{\F_0}(\Om;Z)$ with $\law(u_0)=\mu$, $(F_{\mu}(u_0,W),W,(\Om,\overline{\F},\overline{\P}, (\overline{\F_t})))$ is a $\cond$-strong  solution to $\eqref{SPDE}$ with $F_{\mu}(u_0,W)(0)=u_0$ a.s.  
\end{enumerate}  
Moreover, 
\vspace{-.2cm}
\begin{enumerate}[label=\textup{($\ast$)},ref=\textup{($\ast$)}] 
\item\label{it:YW_Fmu}
 {the map $F_\mu$ in \ref{it:YW3SPDE} is $\overline{\BB_t^{S_2}}^{\mu\otimes\wien}/\BB_t^{S_1}$-measurable for all $t\in\bar{I}_T$.}
\end{enumerate}
The same statements hold if $\cond$ is a subcollection of conditions \ref{it:moment cond}-\ref{it:integrated eq mild weak} and equation \eqref{SPDE} is replaced by \eqref{SPDE mild}.  
\end{theorem}

The above theorem contains the analog of Theorem \ref{th:YW}. Furthermore,  \ref{it:YW_Fmu} states    additional measurability of the map $F_\mu$. In words, it tells us that if $u=F_\mu(u_0,W)$, then the paths of $u$ up till time $t$ are determined by the paths of $(u_0,W)$ up till time $t$. 

More properties of $F_\mu$ will be stated and proved  in Corollary \ref{cor: single F}. The latter   contains an analog of the classical Yamada--Watanabe theorem,  which is partly stronger and concerns all initial measures together,  instead of one fixed measure $\mu$ as in   Theorem \ref{th:YW SPDE}.  

In our notion of $\cond$-weak solution and in Theorem \ref{th:YW SPDE}, we work with arbitrary filtered probability spaces. 
Sometimes, it is preferable to consider only complete filtered probability spaces with complete, right-continuous filtrations, hereafter called \emph{usual filtered probability spaces}.  
Fortunately, the Yamada--Watanabe--Engelbert theorem remains true, by the next remark.

\begin{remark}\label{rem:sol on completion}  
Any $\cond$-weak solution induces a $\cond$-weak solution on a usual filtered probability space: if $(u,W,(\Om,\F,\P,(\F_t)))$ is a $\cond$-weak solution to \eqref{SPDE} or \eqref{SPDE mild}, then so is $(u,W,(\Om,\overline{\F},\overline{\P},(\bar{\F}_{t^+})))$, where $\bar{\F}_{t^+}\ceqq \cap_{h>0} \sigma(\overline{\F_{t+h}})$.

Indeed,  $u$ is  $(\bar{\F}_{t^+})$-adapted since $\F_t\subset \bar{\F}_{t^+}$, and by Lemma \ref{lem:sol on completion}, $W$ is a $U$-cylindrical Brownian motion with respect to $(\Om,\overline{\F},\overline{\P}, (\bar{\F}_{t^+}))$. Also, the conditions in $\cond$ do not depend on the filtration. 

Consequently, conditions \ref{it:YW1SPDE} and  \ref{it:YW2SPDE} of 
Theorem \ref{th:YW SPDE}  are equivalent to respectively:
\begin{enumerate}[label=\textup{(\alph*')},ref=\textup{(\alph*')}] 
  \item  \label{it:YW1usualSPDE}  
    On some usual probability space, there exists a $\cond$-weak solution to \eqref{SPDE}, respectively \eqref{SPDE mild}, with $\law(u(0))=\mu$ and pathwise uniqueness holds for $\cond$-weak solutions with inital law $\mu$ on usual filtered probability spaces.  
  \item  \label{it:YW2usualSPDE}  
    On some usual probability space, there exists a $\cond$-strong solution to $\eqref{SPDE}$ with $\law(u(0))=\mu$ and joint weak uniqueness holds for $\cond$-weak solutions with initial law $\mu$ on usual filtered probability spaces.
\end{enumerate} 
Combined with Theorem \ref{th:YW SPDE}, we obtain the useful implications of the Yamada--Watanabe--Engelbert theorem: \ref{it:YW1usualSPDE}$\Rightarrow$\ref{it:YW3SPDE} and   \ref{it:YW2usualSPDE}$\Rightarrow$\ref{it:YW1SPDE}. 

Of course, one may also replace `usual' by `complete' in the above, interpolating between \ref{it:YW1SPDE} and \ref{it:YW1usualSPDE} and between \ref{it:YW2SPDE} and \ref{it:YW2usualSPDE}.  
\end{remark}

Our main result, Theorem \ref{th:YW SPDE}, will be derived from Theorem \ref{th:YW}. To apply the latter, the crucial step is to define $\Gamma\subset \PP(S_1\times S_2)$ suitably, so that it captures  the set of conditions $\cond$ for the $\cond$-weak solutions to the SPDE. 
To this end, we need to prove that all possible conditions in $\cond$, \ie \ref{it:sol integrability}-\ref{it:integrated eq mild weak}, can be expressed in terms of the joint distribution of $u$ and $W$.  

In Subsection \ref{sub:integrability}, we make the necessary preparations for the integrability conditions \ref{it:sol integrability}-\ref{it:sol integrability ana weak} and \ref{it:sol integrability mild}-\ref{it:sol integrability mild weak}. In Subsection \ref{sub:equation}, we will cover the stochastic equations of conditions \ref{it:integrated eq},\ref{it:integrated eq ana weak},\ref{it:integrated eq mild} and \ref{it:integrated eq mild weak}. Then, in Subsection \ref{sub:main results}, we prove the Yamada--Watanabe--Engelbert theorem (Theorem \ref{th:YW SPDE}) as well as the Yamada--Watanabe theorem with a more extensive notion of `unique strong solution' (Corollary \ref{cor: single F}). In two final remarks, we relate our results to the Yamada--Watanabe results of Ondrej\'at and of Kunze. 

\subsection{Integrability conditions}\label{sub:integrability}

To include the integrability conditions amongst \ref{it:sol integrability}-\ref{it:integrated eq mild weak} in $\Gamma$, we need to prove Borel measurability of certain subsets of $\B$. We will prove this in the following lemma and corollary.  

\begin{lemma}\label{lem: mble L^p subsets} Let $T\in(0,\infty]$, let $\B$ be a Polish space, let $U$ be a separable Hilbert space, let $Y$ and $Z$ be separable Banach spaces with  $Y\into Z$ and let $b$ and $\sigma$ satisfy \eqref{eq:b sigma mble} and \eqref{eq:b sigma adapted}. 

Let $u\col \bar{I}_T\times \Om\to Z$ be $\BB(\bar{I}_T)\otimes\F/\BB(Z)$-measurable. 
Then $\{\om\in\Om:u(\cdot,\om)\in \B\}\in \F$.  Moreover
\begin{align*}
&B\ceqq\{v\in \B:b(\cdot, v)\in L^1_\loc(\bar{I}_T;Z)\}\in \BB(\B),\\ 
&D\ceqq\{v\in \B:\sigma(\cdot, v)\in L^2_\loc(\bar{I}_T;\gamma(U,Y))\}\in \BB(\B),\\
&D^\gamma_t\ceqq\{v\in \B:\sigma(\cdot, v)\in \gamma(L^2(I_t;U),Y)\}\in\BB(\B)\; \text{ for all }  t\in\bar{I}_T.
\end{align*} 
Suppose in addition that $u$ is $({\F_t})$-adapted and $u\in \B$ a.s. 

Then $u$ is indistinguishable from $u\one_{\B}(u)\col \bar{I}_T\times \Om\to Z$, which is everywhere in $\B$, $\F/\BB(\B)$-measurable  and  $(\overline{\F_t})$-progressively measurable. In particular, $u\one_{\B}(u)$ is $\BB(\bar{I}_T)\otimes\overline{\F}/\BB(Z)$-measurable and $(\overline{\F_t})$-adapted. 
Also, $b(\cdot,u\one_{\B}(u))$ and $\sigma(\cdot,u\one_{\B}(u))$ are $\BB({I}_T)\otimes {\F}$-measurable and $(\overline{\F_t})$-adapted. 
\end{lemma}
\begin{proof}
Lemma \ref{lem:L^0} yields strong measurability of $\hat{u}\col\Om\to L^0({I}_T;Z)\col\om\mapsto u(\cdot,\om)$. Note that $\B\into C(\bar{I}_T;Z)\into L^0(I_T;Z)$, so by Kuratowski's theorem  (and Remark \ref{rem:L0}), $\B\in\BB(L^0(I_T;Z))$. Hence $\{\om\in\Om:u(\cdot,\om)\in \B\}=\hat{u}^{-1}(\B)\in \F$.

Similarly, by  \eqref{eq:b sigma mble}, $b$ and $\sigma$ are  $\BB(I_T)\otimes \BB(\B)/\BB(Z)$- and $\BB(I_T)\otimes \BB(\B)/\BB(\gamma(U,Y))$-measurable respectively, hence strongly measurable, as $Z$ and $\gamma(U,Y)$ are separable. Thus Lemma \ref{lem:L^0} yields   
\begin{equation}\label{eq:tilde sigma}
\hat{\sigma}\col \B\to L^0(I_T;\gamma(U,Y))\col v\mapsto  \sigma(\cdot,v) \text{ is strongly measurable}
\end{equation}
and similarly for $\hat{b}\col \B\to L^0(I_T;Z)\col v\mapsto b(\cdot,v)$.  
Next, note that $L^1_\loc(\bar{I}_T;Z)\into L^0(I_T;Z)$ and $L^2_\loc(\bar{I}_T;\gamma(U,Y))\into L^0(I_T;\gamma(U,Y))$, so by Kuratowski's theorem: 
\begin{equation}\label{eq:Lp kuratowski}
    L^1_\loc(\bar{I}_T;Z)\in \BB(L^0(I_T;Z)), \quad L^2_\loc(\bar{I}_T;\gamma(U,Y))\in \BB(L^0(I_T;\gamma(U,Y))).
\end{equation}
Combined with the measurability of $\hat{b},\hat{\sigma}$, this gives $B=\hat{b}^{-1}(L^1_\loc(\bar{I}_T;Z))\in\BB(\B)$ and $D = \hat{\sigma}^{-1}(L^2_\loc(\bar{I}_T;\gamma(U,Y)))\in \BB(\B)$.  
   
Now let $t\in\bar{I}_T$ be arbitrary. Define    
\begin{align*}
&L^{0,\gamma}\ceqq \{\psi\in L^0(I_t;\guz) : \psi^*y^*\in L^2(I_t;U) \text{  $\forall y\in Y^*$, } \psi \text{ represents an }R_\psi\in \gamma(L^2({I}_t;U),Z)\}, \\
&d(\psi,\varphi)\ceqq d_0(\psi,\varphi)+\|R_\psi-R_\varphi\|_{\gamma(L^2({I}_t;U),Z)},
\end{align*} 
where $d_0$ is any metric metrizing the topology of $L^0(I_t;\guz)$. By Remark \ref{rem: represent}, $d$ is well-defined. Observe that $(L^{0,\gamma},d)$ is Polish, $(L^{0,\gamma},d)\into L^0(I_t;\guz)$ and $L^1(I_t;Z)\into L^0(I_t;Z)$. Hence Kuratowski's theorem gives   
$L^{0,\gamma}\in \BB(L^0(I_t;\guz))$.  
Thus $D_t^\gamma=\hat{\sigma}^{-1}(L^{0,\gamma})\in\BB(\B)$. 
     
For the next part, recall that $u$ is $\F/\BB(L^0(I_T;Z))$-measurable by Lemma \ref{lem:L^0} and $\BB(\B)\subset \BB(L^0(I_T;Z))$. 
Therefore, $\bar{u}\ceqq u\one_{\B}(u)$ is $\F/\BB(\B)$-measurable. 
Moreover,  $\bar{u}(t,\cdot)=u(t,\cdot)$  a.s. in $Z$ and $u$ is $(\F_t)$-adapted, so $\bar{u}
$ is trivially $(\overline{\F_t})$-adapted. 
Also, $\bar{u}\in \B\into C({\bar{I}_T};Z)$ everywhere on $\Om$, so it has continuous paths in $Z$. In particular, $\bar{u}$ is $(\overline{\F_t})$-progressively measurable.  
 
For the last part, we only give the proof for $\sigma$, since the same arguments apply  to $b$. Again, let $\bar{u}=u\one_{\B}(u)$. We start with adaptedness of $\sigma(\cdot,\bar{u})$. Let $t\in I_T$. By \eqref{eq:b sigma adapted}, $\sigma(t,\cdot)\col\B\to \gamma(U,Y)$ is $\BB_t(\B)/\BB(\gamma(U,Y))$-measurable. So it suffices to show that $\bar{u}\col \Om\to \B$ is $\overline{\F_t}/\BB_t(\B)$-measurable. By the definition \eqref{eq:def B_t(B^T)}, this is equivalent to showing that $\pi^1_s\circ \bar{u}=\bar{u}(s,\cdot)$ is $\overline{\F_t}/\BB(Z)$-measurable for all $s\in\bar{I}_t$. The latter holds by the $(\overline{\F_t})$-adaptedness of $\bar{u}$, thus $\sigma(\cdot,\bar{u})$ is $(\overline{\F_t})$-adapted. 
For the joint measurability, recall that $\bar{u}$ is $\F/\BB(\B)$-measurable and note that $\Id_T\col I_T \to I_T\col t\mapsto t$ is Borel measurable.  
Hence $(\Id_T,\bar{u})$ is $\BB(I_T)\otimes\F/\BB(I_T)\otimes \BB(\B)$-measurable. Now \eqref{eq:b sigma mble} gives $\BB(I_T)\otimes\F$-measurability of $\sigma\circ(\Id_T,\bar{u})=\sigma(\cdot,\bar{u})$, as desired. 
\end{proof}

\begin{corollary}\label{cor: mble L^p subsets}
  Let Assumption \ref{ass} hold. Then for all $z^*\in Z^*$,
\begin{align}\label{eq: ana weak mble set}
\begin{split}
&B_{z^*}\ceqq\{v\in \B:z^*\circ b(\cdot, v)\in L^1_\loc(\bar{I}_T)\}\in \BB(\B),\\  
&D_{z^*}\ceqq\{v\in \B:z^*\circ \sigma(\cdot, v)\in L^2_\loc(\bar{I}_T;U^*)\}\in \BB(\B).
\end{split}
\end{align}
Let $(S(t,s))_{s,t\in \bar{I}_T, s\leq t}$ be a $Z$-strongly measurable evolution family on $Z$ in the sense of Definition \ref{def: evol fam}. 
Then for all $t\in\bar{I}_T$ and $z^*\in Z^*$,
\begin{align} 
\begin{split}\label{eq: mble sets mild}
&B_t^S\ceqq\{v\in \B:S(t,\cdot)b(\cdot, v)\in L^1({I}_t;Z)\}\in \BB(\B),\\ 
&D_t^S\ceqq\{v\in \B:S(t,\cdot)\sigma(\cdot, v)\in L^2({I}_t;\guz)\}\in \BB(\B), \\
&D_{t}^{\gamma,S}\ceqq\{v\in \B:S(t,\cdot)\sigma(\cdot, v)\in \gamma(L^2({I}_t;U),Z)\}\in \BB(\B),
\end{split}\\
\begin{split}\label{eq: mble sets weak mild}
&B_{z^*,t}^S\ceqq\{v\in \B:z^*\circ S(t,\cdot)b(\cdot, v)\in L^1({I}_t)\}\in \BB(\B),\\ 
&D_{z^*,t}^S\ceqq\{v\in \B:z^*\circ S(t,\cdot)\sigma(\cdot, v)\in L^2({I}_t;U^*)\}\in \BB(\B). 
\end{split}
\end{align}
Moreover, if $u$ is as in Lemma \ref{lem: mble L^p subsets}, then for all $t\in\bar{I}_T$,   $S(t,\cdot)b(\cdot,u\one_{\B}(u))\col {I}_t\times\Om\to Z$ and $S(t,\cdot)\sigma(\cdot,u\one_{\B}(u))\col {I}_t\times\Om\to \guz$ are $\BB({I}_t)\times\F$-measurable and $(\overline{\F_s})$-adapted.   
\end{corollary}
\begin{proof}
Note that $z^*|_Y\in Y^*$ and $S(t,s)|_Y\in\LL(Y,Z)$ since $Y\into Z$. 
Recall that by Assumption \ref{ass}, $b$ and $\sigma$ satisfy \eqref{eq:b sigma mble} and \eqref{eq:b sigma adapted}. 

For \eqref{eq: ana weak mble set}, apply  Lemma \ref{lem: mble L^p subsets} with new coefficients $b_{z^*}\ceqq z^*\circ b$, $\sigma_{z^*}\ceqq z^*|_Y\circ \sigma$, which satisfy \eqref{eq:b sigma mble} and \eqref{eq:b sigma adapted}  with $Y=Z=\R$. 

For \eqref{eq: mble sets mild} and the very last statement, we apply Lemma \ref{lem: mble L^p subsets} with $T=t$,  $Y=Z$, 
${b}^t(\cdot,v)=S(t,\cdot)b(\cdot,v)$, ${\sigma}^t(\cdot,v)=S(t,\cdot)|_Y\circ\sigma(\cdot,v)$. Let us prove that these coefficients satisfy \eqref{eq:b sigma mble} and \eqref{eq:b sigma adapted}. For \eqref{eq:b sigma mble}, we apply Lemma \ref{lem: mble composition}\ref{it: mble comp 3}   with  $(\tilde S,\tilde\A)=({I}_t\times \B,\BB({I}_t)\otimes\BB(\B))$,   
$\Phi^1\col \tilde S\times Z\to Z$, $\Phi^1(s,v,z)\ceqq S(t,s)z$, $\phi^1\ceqq b\col \tilde S\to Z$ and $\Phi^2\col \tilde S\times \guy\to \guz$,  $\Phi^2(s,v,L)=S(t,s)|_Y\circ L$, $\phi^2\ceqq\sigma\col \tilde S\to \guy$ respectively. Note that $b^t=\Phi^1(\cdot,\phi^1(\cdot))$ and $\sigma^t=\Phi^2(\cdot,\phi^2(\cdot))$.
By \eqref{eq:b sigma mble}, $\phi^1$ and $\phi^2$ are measurable. Moreover, $\Phi^1(s,v,\cdot)$ is continuous since $S(t,s)\in\LL(Z)$ and $\Phi^1(\cdot,z)$ is measurable by \eqref{eq: mble evol fam}. 
Since $S(t,s)|_Y\in \LL(Y,Z)$, the ideal property  \cite[Th.\ 9.1.10]{HNVWvolume2} gives that $\Phi^2$ is well-defined and $\Phi^2(s,v,\cdot)$ is continuous. Lastly,  for fixed $x\in U$, $\Phi^2(\cdot,L)x\col \tilde S\to Z$ is measurable  by \eqref{eq: mble evol fam}, hence $\Phi^2(\cdot,L)\col \tilde S\to \guz$ is measurable \cite[Lem.\ 2.5]{NVW07}. Thus indeed, $b^t$ and $\sigma^t$ satisfy \eqref{eq:b sigma mble} by Lemma \ref{lem: mble composition}. 
For \eqref{eq:b sigma adapted}, fix $0\leq s\leq t$. Then $S(t,s)\col Z\to Z$ and $S(t,s)|_Y\circ \col \gamma(U,Y)\to \gamma(U,Z)$ are fixed, continuous maps, hence Borel measurable. Thus $b^t(s,\cdot)=S(t,s)b(s,\cdot)$ and $\sigma^t(s,\cdot)=S(t,s)|_Y\circ\sigma(s,\cdot)$ are $\F_s$-measurable as $b$ and $\sigma$ satisfy \eqref{eq:b sigma adapted}, concluding the proof.  

Lastly, for \eqref{eq: mble sets weak mild}, apply Lemma \ref{lem: mble L^p subsets} with 
${b}_{z^*}^t=z^*\circ b^t$, ${\sigma}_{z^*}^t=z^*\circ \sigma^t$:  $Y=Z=\R$, $T=t$. Since ${b}^t$ and ${\sigma}^t$ satisfy \eqref{eq:b sigma adapted} and \eqref{eq: mble sets mild}, the same holds for ${b}_{z^*}^t$ and ${\sigma}_{z^*}^t$. 
\end{proof}

\subsection{Representation of stochastic equations}\label{sub:equation}
 
We now turn to the conditions amongst \ref{it:sol integrability}-\ref{it:integrated eq mild weak} concerning the actual equations \eqref{SPDE} and \eqref{SPDE mild}. To apply Theorem \ref{th:YW}, we need that satisfaction of these equations is a property of $\law(u,u(0),W)$. To this end, we will prove an even stronger result, namely a functional representation of the stochastic integral in  M-type 2  or UMD Banach spaces, given by Theorem \ref{th:stoch int rep type 2 or UMD}. First, we make some preparations.

For topological spaces $S$, we will henceforth equip the set $\PP(S)$ with the following $\sigma$-algebra:  
\begin{align}\label{eq: def Sigma_S}
\Sigma_S\ceqq \sigma({\Pi}_S^B:B\in\BB(S)),\qquad {\Pi}_S^B\col \PP(S)\to[0,1]\col\mu\mapsto\mu(B).  
\end{align}
\begin{lemma}\label{lem: Sigma_S}
  If $S$ and $S'$ are topological spaces and $\phi\col S\to S'$ is Borel measurable, then $\PP(S)\to \PP(S')\col \mu\mapsto \phi\#\mu$ is $\Sigma_S/\Sigma_{S'}$-measurable. 
\end{lemma}
\begin{proof}
For all $B'\in\BB(S')$ and $A\in\BB[0,1]$, we have
    $B\ceqq\phi^{-1}(B')\in \BB(S)$, so $\{\mu:\phi\#\mu\in ({\Pi}^{B'}_{S'})^{-1}(A)\} 
= ({\Pi}_S^{B})^{-1}(A)\in {\Sigma}_S$. 
\end{proof} 

The next lemma contains the intended functional representation theorem for the special case $Y=\R$. After that, the desired representation for $Y$ of M-type 2 or UMD will be derived from it. 
Recall that  $\W=C(\bar{I}_T;\R^\infty)$ and $C(\bar{I}_T;Y)$ are equipped with the topology of uniform convergence on compact subsets of ${\bar{I}_T}$. See Remark \ref{rem: mathbbW choice} for the result with $C(\bar{I}_T;U_1)$ instead of $\W$. 

The proof extends arguments from \cite[Lem.\ 18.23, Prop.\ 18.26]{kallenberg21} to infinite dimensions. The main ingredient is a functional representation for limits in probability given in  \cite[Prop.\ 5.32]{kallenberg21}.

\begin{lemma}\label{lem:stoch int rep weak} 
Let $U$ be a separable Hilbert space, let $(e_k)$ be an orthonormal basis for $U$ and let $T\in(0,\infty]$. 
There exists a  measurable map 
$${I}\col L^2_\loc(\bar{I}_T;U)\times \W\times \PP(L^2_\loc(\bar{I}_T;U)\times \W)\to C({\bar{I}_T})$$ 
such that for any filtered probability space $(\Om,\F,\P,(\F_t))$, for any  $U$-cylindrical Brownian motion $W$ on $(\Om,\F,\P,(\F_t))$, and for any  $\BB(I_T)\otimes\F/\BB(U)$-measurable, $(\F_t)$-adapted process $f\col I_T\times \Om \to U$ satisfying $f\in L^2_\loc(\bar{I}_T;U)$ a.s., we have
\begin{equation}\label{eq:tildeI weak}
{I}\big(f(\om),W(\om),\law(\bar{f},W)\big)=\big(\int_0^\cdot \<f, \dd W\?_U \big)(\om)\; \text{ in }C({\bar{I}_T}),\; \text{ for $\P$-\aev } \om\in\Om,
\end{equation}
where   $\bar{f}\ceqq f\one_{L^2_\loc(\bar{I}_T;U)}(f)$ and $W=(W^k)$  through \eqref{eq: def Rinfty BM}.  
Here, all spaces are equipped with their Borel $\sigma$-algebra, except for  $\PP(L^2_\loc(\bar{I}_T;U)\times \W)$, which is equipped with the $\sigma$-algebra \eqref{eq: def Sigma_S}. 
\end{lemma}

\begin{proof}
Observe that $\bar{f}$ is $\F$-measurable, since $f\in L^0(\Om;L^0(I_T;U))$ by Lemma \ref{lem:L^0}, and $L^2_\loc(\bar{I}_T;U)\in\BB(L^0(I_T;U))$ by Kuratowski's theorem (and Remark \ref{rem:L0}).    Also,  $W$ is $\overline{\F}$-measurable, so $\law(\bar{f},W)$ is well-defined as $\overline{\P}\#(\bar{f},W)$. 
Moreover, $\bar{f}$ is $\BB(I_T)\otimes\overline{\F}/\BB(U)$-measurable and $(\overline{\F_t})$-adapted and $\P$-a.s.:  $(f,W)=(\bar{f},W)$ and $\int_{0}^{\cdot}{f}\dd W=\int_{0}^{\cdot}\bar{f}\dd W$. Therefore, with no loss of generality, we assume that $f(\om)\in L^2_\loc(\bar{I}_T;U)$ for every $\om\in \Om$. 

We will be using the functional representation for limits in probability in \cite[Prop.\ 5.32]{kallenberg21} three times, for approximation of:
\begin{itemize}
\item $U$-cylindrical Brownian motion by finite dimensional Brownian motions, 
\item $C({\bar{I}_T};U)$-valued process by step processes with an equidistant time grid, 
\item $L^2_\loc(\bar{I}_T;U)$-valued process by $C({\bar{I}_T};U)$-valued processes.
\end{itemize}
We emphasize that each of these approximations should be independent of the underlying probability space, since the same should hold for the eventual map $I$. 

First we construct a measurable map 
\[
\tilde{I}\col C({\bar{I}_T};U)\times \W\times \PP(C({\bar{I}_T};U)\times \W)\to C({\bar{I}_T})
\] 
such that for any filtered probability space $(\Om,\F,\P,(\F_t))$, for any  $U$-cylindrical Brownian motion $W$ on $(\Om,\F,\P,(\F_t))$, and for any $\BB(I_T)\otimes\F/\BB(U)$-measurable, $(\F_t)$-adapted process $f\col {\bar{I}_T}\times \Om \to U$ satisfying $f\in C({\bar{I}_T};U)$, we have
\begin{equation}\label{eq: tildeI}
  \tilde{I}(f(\om),W(\om),\law(f,W)) = \big(\int_0^\cdot \<f, \dd W\?_U \big)(\om)\; \text{ in }C({\bar{I}_T}),\; \text{ for $\P$-\aev } \om\in\Om.
\end{equation}
We equip  $\PP(C({\bar{I}_T};U)\times\W)$ with $\sigma$-algebra \eqref{eq: def Sigma_S} and we 
equip all other spaces with the Borel $\sigma$-algebra. 
For $w\in\W$, we will write $w_k\in C(\bar{I}_T)$ for the $k$-th  component of $w$. 
For $n,m\in\N$ and $j\in\N_0$, put $t_j^m\ceqq \frac{j}{m}\wedge T$ and  
\begin{align*}
&\tilde{I}_{n,m}\col C({\bar{I}_T};U)\times \W\times \PP(C({\bar{I}_T};U)\times \W)\to C({\bar{I}_T}),\\ 
&\tilde{I}_{n,m}(h,w)(t)\ceqq \sum_{j\in\N_0} \sum_{k=1}^n \big(w_k(t\wedge t_{j+1}^m)-w_k(t\wedge t_j^m)\big)\<h(t_j^m),e_k\?_U.
\end{align*}  
Each map $\tilde{I}_{n,m}$ is Borel measurable. Indeed, $\BB(C({\bar{I}_T}))=\sigma({\pi}_{\R}^ t:t\in {\bar{I}_T})$ (see \eqref{eq: point evaluation}) due to Lemma \ref{lem: generating maps}, so it suffices to show that $\tilde{I}_{n,m}(\cdot,\cdot)(t)= {\pi}_{\R}^t\circ \tilde{I}_{n,m}$ is Borel measurable for all $n,m\in\N$ and $t\in {\bar{I}_T}$. 
Now, for any fixed $h\in C({\bar{I}_T};U)$, $w\in \W$ and $t\in\bar{I}_T$, the maps  $\tilde{I}_{n,m}(h,\cdot)(t)\col \W\to  \R$ and $\tilde{I}_{n,m}(\cdot,w)(t)\col C({\bar{I}_T};U)\to  Y$ are continuous. 
Thus Lemma \ref{lem: mble composition}(i)$\Rightarrow$(ii) yields Borel measurability of $\tilde{I}_{n,m}(\cdot,\cdot)(t)$ as required. 

Now, $\tilde{I}_{n,m}$ represents the stochastic integral of approximating step processes: for any filtered probability space $(\Om,\F,\P,(\F_t))$ with $U$-cylindrical Brownian motion $W$ and for any $\BB(I_T)\otimes\F/\BB(U)$-measurable, $(\F_t)$-adapted process $f\col {\bar{I}_T}\times \Om \to U$ satisfying $f\in C({\bar{I}_T};U)$, we have
\begin{equation}\label{eq: step integral}
\tilde{I}_{n,m}(f(\om),W(\om))=\big(\int_0^\cdot \<{P}_nf^m,\dd W\?_U\big)(\om) \text{ for $\P$-\aev } \om\in\Om,
\end{equation}
where $P_n\in \mathcal{L}(U)$ be the orthogonal projection onto $\mathrm{span}(\{e_1,\ldots,e_n\})$ and 
\[
f^m(t,\om)\ceqq \sum_{j\in\N_0}\one_{(t_j^m,t_{j+1}^m]}(t)f(t_j^m,\om). 
\]
An elementary computation shows that \eqref{eq: step integral} holds for $f=\one_{(t_1,t_2]}\otimes \zeta$  with $0\leq t_1<t_2<\infty$ and $\zeta\col\Om\to U$ $\F_{t_1}$-measurable (\eqref{eq: def stoch int elementary} extends due to the It\^o isomorphism).  
Hence, for any $f$ as mentioned above \eqref{eq: step integral}, by linearity in $f$ of both sides of \eqref{eq: step integral}, we obtain as claimed:  
\[
\tilde{I}_{n,m}(f,W)(t)=\tilde{I}_{n,m}(P_nf,W)(t)=\int_0^t \<{P}_n\circ P_nf^m,\dd W\?_U=\int_0^t \<{P}_nf^m,\dd W\?_U\text{ \as}
\]  
Next, for each $n\in\N$ and $f$, $W$ as above, we have
\begin{equation}\label{eq: Pnfm conv} 
\lim_{m\to\infty}\int_0^\cdot \<{P}_nf^m,\dd W\?_U=\int_0^\cdot\<{P}_nf,\dd W\?_U \text{ in probability in }C(\bar{I}_T). 
\end{equation} 
Indeed, continuity of $f$ and of ${P}_n$ give $\lim_{m\to\infty}{P}_nf^m(t,\om)={P}_nf(t,\om)$ in $U$  for every $t\in \bar{I}_T$ and $\om\in\Om$. 
The Dominated Convergence Theorem yields $\lim_{m\to\infty}{P}_nf^m(\cdot,\om)={P}_nf(\cdot,\om)$ in $L^2_\loc(\bar{I}_T;U)$ for every $\om$, in particular in probability. Now the It\^o isomorphism yields \eqref{eq: Pnfm conv}. 

Moreover, since $f\in C({\bar{I}_T};U)$ everywhere on $\Om$,   \eqref{eq: step integral} and \eqref{eq: Pnfm conv} yield 
\[
\lim_{m\to\infty}\tilde{I}_{n,m}(f,W)=\int_0^\cdot\<{P}_nf,\dd W\?_U \text{ in probability in }C(\bar{I}_T). 
\] 
On the other hand, for each $n\in\N$, by \cite[Prop.\ 5.32]{kallenberg21}, there exists a measurable map $\tilde{I}_n\col C({\bar{I}_T};U)\times \W\times \PP(C({\bar{I}_T};U)\times \W)\to C({\bar{I}_T}) $ such that for any $C({\bar{I}_T};U)\times \W$-valued random variable $(\xi,\eta)$ on any probability space, it holds that 
\[
\tilde{I}_{n,m}(\xi,\eta) \text{ converges in probability  in }C(\bar{I}_T;Y)\iff \lim_{m\to\infty}\tilde{I}_{n,m}(\xi,\eta)=\tilde{I}_{n}(\xi,\eta,\law(\xi,\eta))\, \text{ a.s.}
\]
In particular, for each $n\in\N$ and $f$, $W$ as above with $f\in C({\bar{I}_T};U)$ everywhere,
\begin{equation}\label{eq: rep In}
\int_0^\cdot\<{P}_nf,\dd W\?_U =\tilde{I}_{n}(f,W,\law(f,W))\, \text{ a.s.}
\end{equation} 
Next, note that ${P}_nf\to f$ in $U$ on $\bar{I}_T\times\Om$, so by the Dominated Convergence Theorem, ${P}_nf\to f$ in $L^2_\loc(\bar{I}_T;U)$ on $\Om$, and by the It\^o isomorphism  \cite[Th.\ 5.5, (5.4)]{NVW07}:
\[
\lim_{n\to\infty}\int_0^\cdot\<{P}_nf,\dd W\?_U=\int_0^\cdot \<f,\dd W\?_U \text{ in probability  in }C(\bar{I}_T),
\]
Hence, invoking the functional representation \cite[Prop.\ 5.32]{kallenberg21} again, there exists a measurable map $\tilde{I}\col C({\bar{I}_T};U)\times \W\times \PP(C({\bar{I}_T};U)\times \W)\to C({\bar{I}_T})$ such that \eqref{eq: tildeI} is satisfied for all $f$, $W$ as above.
  
It remains to extend the result to integrand processes $f$ lying in $L^2_{\loc}(\bar{I}_T;U)$  
instead of $C(\bar{I}_T;U)$.  
To this end, define
\begin{align*}
&J_n\col L^2_\loc(\bar{I}_T;U)\to C({\bar{I}_T};U),\, J_n(h)(t)\ceqq 
n\!\int_{(t-\frac{1}{n})\vee 0}^t h(s)\dd s,\\ 
&{I}_n\col L^2_\loc(\bar{I}_T;U)\times \W\times \PP(L^2_\loc(\bar{I}_T;U)\times \W)\to C({\bar{I}_T})\col(g,w,\mu)\mapsto \tilde{I}(J_n\circ g,w,(J_n,\Id_\W)\#\mu).
\end{align*} 
Note that for all $t\in \bar{I}_T$: $\|J_n(h)-J_n(g)\|_{C(\bar{I}_t;U)}\leq{n}^{1/2}\|h-g\|_{L^2(I_t;U)}$ by H\"older's inequality, thus each $J_n$ is continuous, hence Borel measurable.  
Moreover,  $\mu\mapsto(J_n,\mathrm{Id}_{\W})\#\mu$ is  
measurable by Lemma \ref{lem: Sigma_S}. 
Consequently, each $I_n$ is measurable.

Furthermore, the maps $J_n$ satisfy
\begin{equation}\label{eq:J_n pointwise}
\lim_{n\to\infty}J_nh= h \text{ in }  L^2_\loc(\bar{I}_T;U), \quad\qquad h\in L^2_\loc(\bar{I}_T;U).
\end{equation}
This follows from the $L^p$-convergence of approximate identities in $L^p(\R^d;E)$ with $E$ Banach. 
Putting $\varphi\ceqq \one_{[0,1]}\in L^1(\R)$ and $\varphi_n(t)\ceqq n\varphi(nt)$, we obtain from \cite[Prop.\ 1.2.32]{HNVWvolume1}:
\[
\lim_{n\to\infty}\varphi_n\ast \tilde{h}=\tilde{h} \,\text{ in }L^2(\R;U),\quad \text{ for any $\tilde{h}\in L^2(\R;U)$}, 
\]
where $\ast$ denotes the convolution. Now, let $t\in \bar{I}_T$ and $h\in L^2_\loc(\bar{I}_T;U)$ be  arbitrary. Define $\tilde{h}\ceqq \one_{I_t}h\in L^2(\R;U)$. For all $r\in I_t$ we have: 
\begin{align*}
\varphi_n\ast \tilde{h}(r)=n\int_\R \one_{[0,1/n]}(r-s)\tilde{h}(s)\dd s
&=n\int_{r-1/n}^r \tilde{h}(s)\dd s
= n\int_{(r-1/n)\vee 0}^r  {h}(s)\dd s=J_n(h)(r).
\end{align*}
Therefore,
\begin{align*}
  \|J_n(h)-h\|_{L^2(I_t;U)} &= \|(\varphi_n\ast \tilde{h}-\tilde{h})\one_{I_t}\|_{L^2(\R;U)} \leq \|\varphi_n\ast \tilde{h}-\tilde{h}\|_{L^2(\R;U)}\to 0, 
\end{align*}
proving \eqref{eq:J_n pointwise}. 

Suppose that $f\col \bar{I}_T\times \Om \to U$ is  $\BB(\bar{I}_T)\otimes\F/\BB(U)$-measurable and $(\F_t)$-adapted process  satisfying $f\in L^2_\loc(\bar{I}_T;U)$ on $\Om$. 
Then the convergence in \eqref{eq:J_n pointwise} holds pointwise on $\Om$, hence in probability, 
and the It\^o homeomorphism \cite[Th.\ 5.5]{NVW07} and \eqref{eq: tildeI} yield
\[
{I}_n(f,W,\law(f,W))\overset{\text{a.s.}}{=}\int_0^\cdot \<J_n\circ f,\dd W\?_U\to \int_0^\cdot \<f,\dd W\?_U \text{ in probability in }C({\bar{I}_T})  
\]
as $n\to \infty$. 
On the other hand, \cite[Prop.\ 5.32]{kallenberg21} gives a measurable map $I\col L^2_\loc(\bar{I}_T;U)\times \W\times \PP(L^2_\loc(\bar{I}_T;U)\times \W)\to C({\bar{I}_T})$ such that for any measurable $(f,W)\col \Om\to L^2_\loc(\bar{I}_T;U)\times \W$:  
\begin{align*}
(I_n(f,W,\law(f,W)))_n &\text{ converges in probability }\\
&\iff I_n(f,W,\law(f,W)){\to} I(f,W,\law(f,W)) \text{  in probability}. 
\end{align*}  
For every  $\BB(\bar{I}_T)\otimes\F/\BB(U)$-measurable, $(\F_t)$-adapted process $f\col \bar{I}_T\times \Om \to U$ satisfying $f\in L^2_\loc(\bar{I}_T;U)$ on $\Om$, we have $\law(f,W)=\law(\bar{f},W)$, and we conclude that   \eqref{eq:tildeI weak} holds.  
\end{proof} 

It might be that Lemma \ref{lem:stoch int rep weak} can also be proved with `$\law(f,W)$' replaced by `$\law(f)$'. However, this is irrelevant for the Yamada--Watanabe theory of this paper. 

Using the above representation for scalar valued stochastic integrals, we are also able to represent the stochastic integral in M-type 2 spaces and in UMD spaces. The proof relies on having a sequence of functionals in the dual space that separates the points. Such functionals give all necessary information for the desired measurability.  
For UMD spaces, recall the conditions for stochastic integrability directly above and below Definition \ref{def: UMD process2}. 

\begin{theorem}\label{th:stoch int rep type 2 or UMD} 
Let $U$ be a separable real Hilbert space, let $(e_k)$ be an orthonormal basis for $U$ and let $Y$  
be a separable real Banach space.  

If $Y$ has M-type 2 and $T\in(0,\infty]$, then there exists a measurable map 
$${I}\col L^2_\loc(\bar{I}_T;\gamma(U, Y))\times \W\times \PP(L^2_\loc(\bar{I}_T;\gamma(U, Y))\times \W)\to C({\bar{I}_T}; Y)$$ 
such that for any filtered probability space $(\Om,\F,\P,(\F_t))$, for any  $U$-cylindrical Brownian motion $W$ on $(\Om,\F,\P,(\F_t))$, and for any measurable, $(\F_t)$-adapted $f\col \bar{I}_T\times \Om \to \gamma(U,Y)$ such that $f\in L^2_\loc(\bar{I}_T;\gamma(U, Y))$ a.s., we have 
\begin{equation}\label{eq:tildeI}
{I}\big(f(\om),W(\om),\law(\bar{f},W)\big)=\Big(\int_0^\cdot f \dd W \Big)(\om)\; \text{ in }C({\bar{I}_T}; Y),\; \text{ for $\P$-\aev } \om\in\Om,
\end{equation} 
where $\bar{f}\ceqq f\one_{L^2_\loc(\bar{I}_T;\gamma(U, Y))}(f)$.

If $Y$ is UMD and $T\in(0,\infty)$, then there exists a measurable map 
$${I}\col \gluy\times \W\times \PP\big(\gluy\times \W\big)\to C({\bar{I}_T}; Y)$$ 
such that for any filtered probability space $(\Om,\F,\P,(\F_t))$, for any  $U$-cylindrical Brownian motion $W$ on $(\Om,\F,\P,(\F_t))$, and for any stochastically integrable process $f\col \bar{I}_T\times \Om \to \LL(U,Y)$ representing   $\xi_f\col\Om\to\gluy$, we have 
\begin{equation}\label{eq:tildeI no type 2}
{I}\big(\xi_f(\om),W(\om),\law(\xi_f,W)\big)=\Big(\int_0^\cdot f \dd W \Big)(\om)\; \text{ in }C({\bar{I}_T}; Y),\; \text{ for $\P$-\aev } \om\in\Om. 
\end{equation}

In both cases, we identify $W=(W^k)$  through \eqref{eq: def Rinfty BM} using $(e_k)$.  
Moreover, all spaces are equipped with their Borel $\sigma$-algebra, except for $\PP(L^2_\loc(\bar{I}_T;\gamma(U, Y))\times \W)$ and $\PP\big(\gluy\times \W\big)$, which are equipped with their corresponding $\sigma$-algebra \eqref{eq: def Sigma_S}. 
\end{theorem}
\begin{proof}
  We write $I_{\R}$ for the map obtained in Lemma \ref{lem:stoch int rep weak} (corresponding to the case $Y=\R$). 
  For any separable Banach space $Y$, we can pick a sequence $(y_k^*)\subset Y^*$ which separates the points in $Y$ \cite[Prop.\ B.1.11]{HNVWvolume1}. Then 
   \begin{equation}\label{eq: iota}
   \iota\col C(\bar{I}_T;Y)\to C(\bar{I}_T;\R^\infty)\col f\mapsto (y_k^*\circ f)_k
   \end{equation}
   is a continuous injection between Polish spaces, hence $\iota(C(\bar{I}_T;Y))\in\BB(C(\bar{I}_T;\R^\infty))$ and $\iota$ is a Borel isomorphism between $C(\bar{I}_T;Y)$ and $\iota(C(\bar{I}_T;Y))$ \cite[Cor.\ 15.2]{kechris95}. 
   
   For the M-type 2 case, we let $R\col U^*\cong U$ be the Riesz isomorphism and we define  
   
   \noindent
   $I^k\col L^2_\loc(\bar{I}_T;\gamma(U, Y))\times \W\times \PP(L^2_\loc(\bar{I}_T;\gamma(U, Y))\times \W)\to C({\bar{I}_T})$ by  
   $$
     I^k(h,w,\mu)\ceqq I_\R(R\circ (y_k^*\circ h(\cdot)),w,(R\circ y_k^*,\Id_{\W})\#\mu).
   $$
   Then $I^k$ is measurable (recall Lemma \ref{lem: Sigma_S}) and for $(f,W)$ as in the statement, we have 
   \begin{equation}\label{eq: I^k}
   I^k(f,W,\law(\bar{f},W))=\int_0^\cdot \<R\circ y_k^*\circ f,\dd W\?_U= \int_0^\cdot y_k^*\circ f\dd W=y_k^*\circ\int_0^\cdot f\dd W 
   \end{equation}  
   \as in $C(\bar{I}_T)$. As in the proof of Lemma \ref{lem:stoch int rep weak},  $\law(\bar{f},W)$ is well-defined. 
   Note that $(I^k)_k$ is then a measurable $C(\bar{I}_T;\R^\infty)$-valued map, by Lemma \ref{lem: generating maps} applied with the projections onto the components $C(\bar{I}_T;\R^\infty)\to C(\bar{I}_T)$. Consequently,  $I\col L^2_\loc(\bar{I}_T;\gamma(U, Y))\times \W\times \PP(L^2_\loc(\bar{I}_T;\gamma(U, Y))\times \W)\to C({\bar{I}_T};Y)$ given by   
   $$
   I(h,w,\mu)\ceqq \begin{cases}
   \iota^{-1}((I^k(h,w,\mu))_k), \quad & (I^k(h,w,\mu))_k\in\mathrm{ran}(\iota),\\
   0, &\text{otherwise},\end{cases}
   $$ 
   is measurable. Taking a countable intersection over $k\in\N$ of subsets of $\Om$ with measure 1 on which \eqref{eq: I^k} holds, and using \eqref{eq: iota}, we obtain \eqref{eq:tildeI}. 
   
  Now let $Y$ be UMD. Note that $\gamma(L^2(I_T;U),\R)=L^2(I_T;U)^*\cong L^2(I_T;U)$ and let $\tilde{R}$ denote the latter Riesz isomorphism. 
    Define $I^k\col \gluy\times \W\times \PP\big(\gluy\times \W\big)\to C({\bar{I}_T})$ by $I^k(L,w,\tilde{\mu})\ceqq I_\R(\tilde{R}\circ (y_k^*\circ L),w,(\tilde{R}\circ y_k^*,\Id_{\W})\#\tilde{\mu})$. 
    Note that $(\tilde{R}\circ y_k^*,\Id_{\W})\#\tilde{\mu} \in \mathcal{P}(L^2(I_T;U) \times \W)$, matching the last component of the domain of $I_{\R}$. 
    
    Since $f$ represents $\xi_f$ (see \eqref{eq: def represent}), we have $\tilde{R}\circ (y_k^*\circ \xi_f)=R\circ (y_k^*\circ f(\cdot))$, thus  $I^k(\xi_f,W,\law(\xi_f,W))=\int_0^\cdot \<\tilde{R}\circ (y_k^*\circ \xi_f),\dd W\?_U=y_k^*\circ\int_0^\cdot f\dd W$ \as in $C(\bar{I}_T)$. As before,  $I(L,w,\tilde{\mu})\ceqq \iota^{-1}((I^k(L,w,\tilde{\mu}))_k)$ if $(I^k(L,w,\tilde{\mu}))_k\in\mathrm{ran}(\iota)$  and $I(L,w,\tilde{\mu})\ceqq0$ otherwise, defines the desired map.  
\end{proof}

\begin{remark}\label{rem: mathbbW choice}
Some authors view $W$ as $C(\bar{I}_T;U_1)$-valued for a Hilbert space $U_1$, rather than $\W=C(\bar{I}_T;\R^\infty)$-valued as we do here. 
However, by virtue of Lemma \ref{lem: sigma algs W and W_1 new}, this makes no difference for Lemma \ref{lem:stoch int rep weak} and Theorem \ref{th:stoch int rep type 2 or UMD}. Indeed, one can replace $W=(W^k)$ by the $Q_1$-Wiener process $W_1$ defined by \eqref{eq: def W_1} and replace $\W$ by $\W_1\ceqq C(\bar{I}_T;U_1)$. By Lemma \ref{lem: sigma algs W and W_1 new}, we have $W|_{\bar{I}_T}=\phi_1(W_1|_{\bar{I}_T})$ $\P$-a.s. for a measurable map $\phi_1\col \W_1\to \W$ that is independent of $W$ and  $W_1$. 
For  $I$ as in Theorem \ref{th:stoch int rep type 2 or UMD} and $\Id$ the identity map on $L^2_\loc(\bar{I}_T;\gamma(U, Y))$ or $\gamma(L^2(I_T;U),Y)$,  
$I^1\ceqq I\circ (\Id(\cdot),\phi_1(\cdot),(\Id,\phi_1)\#(\cdot))$ is measurable and satisfies ${I}^1\big(f,W_1,\law(\bar{f},W_1)\big)= \int_0^\cdot f \dd W $ \as in $C({\bar{I}_T}; Y)$. 
\end{remark}

As a simple consequence, we obtain measurable representations of the integrals in the equations  of conditions \ref{it:integrated eq}, \ref{it:integrated eq ana weak}, \ref{it:integrated eq mild} and \ref{it:integrated eq mild weak}, as we will apply the next corollary to different pairs $(b,\sigma)$.

\begin{corollary}\label{cor: diffusion drift mble rep} 
Let $\B$ be a Polish space, let $U$ be a separable Hilbert space, let $Y$ and $Z$ be real separable Banach spaces with $Y\into Z$ and let  $b$ and $\sigma$ satisfy \eqref{eq:b sigma mble} and \eqref{eq:b sigma adapted}. 
Let $Y$ be of M-type 2 and let $T\in(0,\infty]$. 

There exists a measurable map $I^{b,\sigma}\col \B\times \W \times \PP(\B\times \W) \to C({\bar{I}_T}; Z)$ such that for any filtered probability space $(\Om,\F,\P,(\F_t))$, for any  $U$-cylindrical Brownian motion $W$ on $(\Om,\F,\P,(\F_t))$ and any process $u\col \bar{I}_T\times \Om\to Z$ that is  $\BB(\bar{I}_T)\otimes\F/\BB(Z)$-measurable, $(\F_t)$-adapted and satisfies condition \ref{it:sol integrability}, we have  for $\P$-\aev  $\om\in\Om$: 
\begin{equation}\label{eq:functional I^sigma}
{I}^{b,\sigma}(u(\om),W(\om),\law(\bar{u},W)) =\int_0^\cdot b(s,u(\om))\dd s+\Big(\int_0^\cdot \sigma(s,u)\dd W(s)\Big)(\om)\,\text{ in $C({\bar{I}_T}; Z)$},
\end{equation}
where $\bar{u}=u\one_\B(u)$,  $W=(W^k)$ through \eqref{eq: def Rinfty BM} and $\PP(\B\times \W)$ is equipped with the $\sigma$-algebra \eqref{eq: def Sigma_S}. 

If $Y$ is a real separable UMD Banach space and $T\in(0,\infty)$, then the same statement holds with  condition  \ref{it:sol integrability} replaced by \ref{it:sol integrability UMD}. 
\end{corollary}
\begin{proof} Recall that $\bar{u}$ is a $\F$-measurable by Lemma \ref{lem: mble L^p subsets} and $W$ is $\overline{\F}$-measurable, so we may indeed write $\law(\bar{u},W)$ for $\overline{\P}\#(\bar{u},W)$. 

Suppose that $Y$ is of M-type 2 and $T\in(0,\infty]$. Let $I$ be the corresponding map from Theorem \ref{th:stoch int rep type 2 or UMD}. Define 
\begin{align*}
&A^\sigma\ceqq\{v\in \B:\sigma(\cdot, v)\in L^2_\loc(\bar{I}_T;\gamma(U, {Y}))\},\\
&\Phi^\sigma\col\B\to 
L^2_\loc(\bar{I}_T;\gamma(U, {Y}))\col v\mapsto \one_{A^\sigma}(v)\sigma(\cdot, v),\\
&{I}^\sigma\col \B\times \W\times\PP(\B\times \W)\to C({\bar{I}_T}; Z)\col(v,w,\mu)\mapsto \one_{A^\sigma}(v){I}(\sigma(\cdot,v),w,(\Phi^\sigma,\mathrm{Id}_{\W})\#\mu).
\end{align*} 
We claim that $I^\sigma$ is measurable. Note that ${I}$ is also measurable with respect to codomain $C({\bar{I}_T};Z)$ instead of $C({\bar{I}_T};{Y})$ by Kuratowski's theorem. 
Moreover, $A^\sigma\in\BB(\B)$ by Lemma \ref{lem: mble L^p subsets}  and $\Phi^\sigma$ is measurable by \eqref{eq:tilde sigma} and \eqref{eq:Lp kuratowski}. 
Lemma \ref{lem: Sigma_S} now yields that $\PP(\B\times \W)\to\PP(L^2_\loc(\bar{I}_T;\gamma(U, {Y}))\times \W)\col \mu\mapsto(\Phi^\sigma,\mathrm{Id}_{\W})\#\mu$ is measurable. Measurability of $I^\sigma$ follows.  

Let $u\col \bar{I}_T\times \Om \to Z$ be any  $\BB(I_T)\otimes\F/\BB(Z)$-measurable, $(\F_t)$-adapted process  satisfying \ref{it:sol integrability}. Then for $f(t,\om)\ceqq \sigma(t,u(\om))$ we have  $f\in L^2_\loc(\bar{I}_T;\gamma(U, {Y}))$ a.s., by \ref{it:sol integrability}. Moreover, $f$ is measurable and $(\F_t)$-adapted by Lemma \ref{lem: mble L^p subsets}. 
Thus \eqref{eq:tildeI} yields 
\[
{I}^\sigma(u,W,\law(\bar{u},W))=\int_0^\cdot \sigma(s,u)\dd W(s) \text{ a.s.}
\]  
Next, consider $I^b\col \B\times \W \times \PP(\B\times \W) \to C({\bar{I}_T};Z)\col (v,w,\mu)\mapsto \int_0^\cdot b(s,v)\dd s$. Note that $I^b=\Psi\circ\Phi^b$, where
\begin{align*}
&\Phi^b\col\B\to L^1_\loc(I_T;Z)\col v\mapsto \one_{A^b}(v)b(\cdot, v),\quad \Psi\col L^1_\loc(I_T;Z)\to C({\bar{I}_T};Z)\col v\mapsto \int_0^\cdot v(s)\dd s,
\end{align*}
with $A^b\ceqq \{v\in \B:b(\cdot, v)\in L_\loc^1(I_T;Z)\}$. 
Similar to $\Phi^\sigma$, $\Phi^b$ is measurable by Lemma \ref{lem: mble L^p subsets}, \eqref{eq:tilde sigma} and \eqref{eq:Lp kuratowski}. Also,  $\Psi$ is continuous, so $I^b$ is measurable. Putting $I^{b,\sigma}\ceqq I^b+I^\sigma$ completes the proof. 

If $Y$ is UMD and $T\in(0,\infty)$, we can copy the proof above, using the map $I$ corresponding to the UMD case from Theorem \ref{th:stoch int rep type 2 or UMD} and replacing $A^\sigma$ and $\Phi^\sigma$ by respectively:
\begin{align*}
&A^{\sigma}\ceqq\{v\in \B:\sigma(\cdot, v)\in \gamma(L^2(I_T;U),Y)\},\, 
&\Phi^{\sigma}\col\B\to\gamma(L^2(I_T;U),Y))\col v\mapsto \one_{A^{\sigma}}(v)\sigma(\cdot, v).
\end{align*}   
\end{proof}

At last, we express the equations appearing in conditions \ref{it:integrated eq}, \ref{it:integrated eq ana weak}, \ref{it:integrated eq mild} and \ref{it:integrated eq mild weak} in terms of the joint distribution of $u$ and $W$. 
  
\begin{corollary}\label{cor:SPDE distributional unified}
Let Assumption \ref{ass} hold. 
There exist sets 
$A_{_{\scriptscriptstyle\mathrm{M2}}}, A_{\scriptscriptstyle\mathrm{UMD}}, A_{w},A_{\scriptscriptstyle\mathrm{M2}}^S,A_{\scriptscriptstyle\mathrm{UMD}}^S,A^S_{w}\subset \PP(\B\times \W)$ 
such that whenever $(\Om,\F,\P,(\F_t))$ is a filtered probability space,   $W$ is a  $U$-cylindrical Brownian motion on $(\Om,\F,\P,(\F_t))$  and $u\col \bar{I}_T\times \Om\to Z$ is  $\BB(\bar{I}_T)\otimes\F/\BB(Z)$-measurable, $(\F_t)$-adapted and $u\in\B$ $\P$-a.s., we have:
  \begin{align*}
  &\text{If \ref{it:sol integrability} holds, then: }\qquad\law(\bar{u},W)\in 
  A_{\scriptscriptstyle\mathrm{M2}} 
  \iff \ref{it:integrated eq} \,\text{ holds.}\\
  &\text{If \ref{it:sol integrability UMD} holds, then: }\qquad\law(\bar{u},W)\in 
  A_{\scriptscriptstyle\mathrm{UMD}} 
  \iff \ref{it:integrated eq} \,\text{ holds.}\\
  &\text{If \ref{it:sol integrability ana weak} holds, then: }\qquad\law(\bar{u},W)\in 
  A_{\mathrm{weak}} 
  \iff \ref{it:integrated eq ana weak} \,\text{ holds.}\\
  &\text{If \ref{it:sol integrability mild} holds, then: }\qquad\law(\bar{u},W)\in 
  A_{\scriptscriptstyle\mathrm{M2}}^S 
  \iff \ref{it:integrated eq mild}\text{ holds.}\\
  &\text{If \ref{it:sol integrability mild UMD} holds, then: }\qquad\law(\bar{u},W)\in 
   A_{\scriptscriptstyle\mathrm{UMD}}^S 
  \iff \ref{it:integrated eq mild}\text{ holds.}\\
  &\text{If \ref{it:sol integrability mild weak} holds, then: }\qquad\law(\bar{u},W)\in 
  A_{\mathrm{weak}}^S 
  \iff \ref{it:integrated eq mild weak} \,\text{ holds.}
  \end{align*} 
  Here, $\bar{u}\ceqq u\one_\B(u)$ and  $W=(W^k)$ through \eqref{eq: def Rinfty BM}.
\end{corollary}

\begin{proof} 
Let $(b_{z^*},\sigma_{z^*})$, $({b}^t,{\sigma}^t)$ and $({b}_{z^*}^t,{\sigma}_{z^*}^t)$ be as in the proof of Corollary \ref{cor: mble L^p subsets}. There, it was shown that these coefficients satisfy the conditions of Lemma \ref{lem: mble L^p subsets}, with respectively ($Y=Z=\R$), ($T=t$, $Y=Z$) and  ($T=t$, $Y=Z=\R$). 
These conditions coincide with those of Corollary \ref{cor: diffusion drift mble rep}, thus the latter yields measurable maps $I^{b,\sigma}_{\scriptscriptstyle \mathrm{M2}}$, $I^{b,\sigma}_{\scriptscriptstyle \mathrm{UMD},t}$, $I^{b_{z^*},\sigma_{z^*}}$, $I^{b^t,\sigma^t}_{\scriptscriptstyle \mathrm{M2},t}$, $I^{b^t,\sigma^t}_{\scriptscriptstyle \mathrm{UMD},t}$ and $I^{b_{z^*}^t,\sigma_{z^*}^t}_t$ for $z^*\in Z^*$ and  $t\in \bar{I}_T$.  The subscripts M2 and UMD indicate whether $Y$ is assumed to be M-type 2 or UMD. If absent, Corollary \ref{cor: diffusion drift mble rep} is applied with $Y=\R$ and we use the corresponding M-type 2 map (of course the UMD map is in this case the same up to Riesz isomorphisms). The subscripts $t$ indicate that $T\in(0,\infty]$ is replaced by $t(<\infty)$ in the application of Corollary \ref{cor: diffusion drift mble rep}.  
In addition, consider maps  from $\B\times\W\times\PP(\B\times\W)$ to respectively $C(\bar{I}_T;Z)$, $C(\bar{I}_T)$, $Z$ and $\R$,  given by $I_0(v,w,\mu)\ceqq v(\cdot) -v(0)$, $I_0^{z^*}\ceqq \<I_0(\cdot),z^*\?$,  $I_{0,t}(v,w,\mu)\ceqq v(t) -S(t,0)v(0)$ and $I_{0,t}^{z^*}\ceqq \<I_{0,t}(\cdot),z^*\?$. These maps are continuous, hence measurable. 
Also, for fixed $\mu\in \PP(\B\times\W)$, 
$\phi_\mu\col \B\times\W\to \B\times\W \times\PP(\B\times\W)\col (v,w)\mapsto (v,w,\mu)$  
is trivially $\BB(\B\times\W)/\BB(\B\times\W)\otimes \Sigma_{\B\times\W}$-measurable. 
Combining this with the measurability from Corollary \ref{cor: diffusion drift mble rep} and using the notation \eqref{eq: point evaluation} (with $T=t$), we find that for any fixed $\mu\in \PP(\B\times\W)$, $t\in\bar{I}_T$ and $z^*\in Z^*$, the following sets belong to $\BB(\B\times \W)$: $B_{\scriptscriptstyle\mathrm{M2},\mu}\ceqq ((I_{\scriptscriptstyle\mathrm{M2}}^{b,\sigma}-I_0)\circ \phi_\mu)^{-1}(\{0\})$, $B_{\scriptscriptstyle\mathrm{UMD},\mu,t}\ceqq (\pi_Z^t\circ ( I_{\scriptscriptstyle\mathrm{UMD},t}^{b,\sigma}-I_{0})\circ  \phi_\mu)^{-1}(\{0\})$, $B_{\mathrm{weak},\mu,z^*}\ceqq ((I^{b_{z^*},\sigma_{z^*}}-I_0^{z^*})\circ \phi_\mu)^{-1}(\{0\})$,    $B_{\mathrm{weak},\mu,t,z^*}^S\ceqq ((\pi_Z^t\circ I^{b_{z^*}^t,\sigma_{z^*}^t}-I_{0,t}^{z^*})\circ \phi_\mu)^{-1}(\{0\})$ and $B_{\bullet,\mu,t}^S\ceqq ((\pi_Z^t\circ  I_{\bullet,t}^{b^t,\sigma^t}-I_{0,t})\circ \phi_\mu)^{-1}(\{0\})$ for $\bullet=\mathrm{M2},\mathrm{UMD}$. 
It remains to put
\begin{align*}
&
A_{\scriptscriptstyle\mathrm{M2}}
\ceqq \{\mu\in\PP(\B\times \W):\mu(B_{\scriptscriptstyle\mathrm{M2},\mu})=1\}, \\ 
&
A_{\scriptscriptstyle\mathrm{UMD}}
\ceqq \cap_{t\in\bar{I}_T}\{\mu\in\PP(\B\times \W):\mu(B_{\scriptscriptstyle\mathrm{UMD},\mu,t})=1\}, \\
&
A_{\mathrm{weak}}
\ceqq \cap_{z^*\in D}\{\mu\in\PP(\B\times \W):\mu(B_{\mathrm{weak},\mu,z^*})=1\}, \\
&
A_{\scriptscriptstyle\mathrm{M2}}^S
\ceqq \cap_{t\in\bar{I}_T}\{\mu\in\PP(\B\times \W):\mu(B_{\scriptscriptstyle\mathrm{M2},\mu,t}^S)=1\},  \\
&
A_{\scriptscriptstyle\mathrm{UMD}}^S
\ceqq \cap_{t\in\bar{I}_T}\{\mu\in\PP(\B\times \W):\mu(B_{\scriptscriptstyle\mathrm{UMD},\mu,t}^S)=1\}, \\
&
A_{\mathrm{weak}}^S
\ceqq \cap_{t\in\bar{I}_T,z^*\in D}\{\mu\in\PP(\B\times \W):\mu(B_{\scriptscriptstyle\mathrm{weak},\mu,t,z^*}^S)=1\}.
\end{align*} 
\end{proof}  

\subsection{Proof of the main results}\label{sub:main results}

Since Theorem \ref{th:YW} already provides the core of the Yamada--Watanabe--Engelbert theorem, it remains to define $S_1, S_2, \Gamma,\CC$ and $\nu$ in such a way, that the set $\SS_{\Gamma,\CC,\nu}\subset \PP(S_1\times S_2)$ consists precisely of all laws of $\cond$-weak solutions, for a given collection of conditions $\cond$.  
From now on, let
\begin{equation}\label{eq: S1 S2}
S_1\ceqq \B, \qquad S_2\ceqq Z\times \W.
\end{equation} 
Recall that we identify $\B$ with a subspace of $C(\bar{I}_T;Z)$ since $\B\into C(\bar{I}_T;Z)$ by Assumption \ref{ass}. 
For $s\in \bar{I}_T$, we define   
\begin{equation}\label{eq: def pi_1 pi_2}
\pi^s_{1}\col\B\to Z\col v\mapsto v(s),\quad \pi^s_{2}\col S_2\to Z\times\R^\infty\col (x,w)\mapsto (x,w(s)).
\end{equation} 
For $t\in\bar{I}_T$ and $i=1,2$, we define a temporal compatibility structure:
\begin{equation}\label{eq:temporal}
  \BB_t^{S_i}\ceqq \sigma(\pi^s_{i}:s\in \bar{I}_t), \quad  \CC\ceqq\{(\BB_t^{S_1},\BB_t^{S_2}):t\in \bar{I}_T\}.
\end{equation}
Lastly, for any measurable $u\col (\Om,\F)\to S_1$, $\eta=(u_0,W)\col(\Om,\F)\to S_2=Z\times \W$, we let 
\begin{align}\label{eq:F_t^eta} 
\begin{split}
&\F_t^u\ceqq \overline{\sigma(u^{-1}(B):B\in\BB_t^{S_1})}= \overline{\sigma(u(s)^{-1}(A):A\in\BB(Z),s\in\bar{I}_t))},\\
&\F_t^\eta\ceqq \overline{\sigma(\eta^{-1}(B):B\in\BB_t^{S_2})}
=\overline{\sigma((u_0,(W^k(s))_{k\in\N}):s\in\bar{I}_t)},\\
&\F_t^W\ceqq\overline{\sigma((W^k(s))_{k\in\N} : s\in\bar{I}_t))}.   
\end{split}
\end{align} 
The notations $\F_t^u$ and $\F_t^\eta$  are consistent with Definition \ref{def:compatible} ($\A=\bar{I}_T$).

\begin{remark}
From  Lemma \ref{lem: sigma algs W and W_1 new}\ref{it:W (e_k)-indep}, we find that $(\F_t^{W})$ is independent of the chosen orthonormal basis $(e_k)$, even though $(W^k)$ is defined through $(e_k)$. 
\end{remark}

Observe that $\BB_t^{S_i}$ contains all information of paths up till time $t$. This will in particular ensure that weak solutions for $(\Gamma,\CC,\nu)$ become adapted processes with respect to suitable filtrations. Conversely, SPDE solutions are automatically $\CC$-compatible (`$\com$', see Definition \ref{def:compatible}), as the next lemma shows.  
The proof is a special case of the computation in \cite[p.\ 7]{kurtz14}. A bit more detail is added for convenience. The independent increments of ($U$-cylindrical) Brownian motions are used in a crucial way here. Similar arguments can be given for L\'evy processes, but then, $\B$ should embed continuously  into a  Skorohod space, rather than into $C(\bar{I}_T;Z)$. 

\begin{lemma}\label{lem:SPDE sol is compatible}
  Let Assumption \ref{ass} hold. 
  Let $(\Om,\F,\P,(\F_t))$ be a filtered probability space, let $u\col I_T\times\Om\to Z$ be measurable,  $(\F_t)$-adapted and such that $u(\cdot,\om)\in\B$ for all $\om\in\Om$. Let $W$ be a $U$-cylindrical Brownian motion on $(\Om,\F,\P,(\F_t))$. Then $u\com(u(0),W)$, where  $\CC$ is defined by \eqref{eq:temporal} and $W=(W^k)$ through \eqref{eq: def Rinfty BM}.  
\end{lemma}
\begin{proof}
 Let $\eta\ceqq(u(0),W)\col\Om\to Z\times \W$. 
 We will decompose $\eta$ into an $\F_t^\eta$-measurable part $\eta_t$, and an $\F_t^\eta$-independent part $\eta^t$. 
  For $t\in \bar{I}_T$, put $R_t\ceqq Z\times C(\bar{I}_t;\R^\infty)$, $R^t\ceqq C(\bar{I}_{T-t};\R^\infty)$ and let   $\eta_t\ceqq (u(0),W|_{\bar{I}_t})\col \Om\to R_t$ and $\eta^t\ceqq (W(t+\cdot)-W(t))\big|_{\bar{I}_{T-t}}\col\Om\to  R^t$ (with $\infty-t\ceqq\infty$). 
  Applying Lemma \ref{lem: generating maps} with (time) point evaluation maps on $R_t$ and $R^t$, and using that $W$ is a $U$-cylindrical Brownian motion with respect to $(\F_t)$ and $u$ is $(\F_t)$-adapted, we find 
   \begin{align}
   &\sigma(\eta)=\sigma(\eta_t,\eta^t), \label{eq: eta split1} \\
   &\sigma(\eta_t)
   \subset \F_t^\eta\subset \F_t^u\vee\F_t^\eta, \label{eq: eta split2} \\
   &\sigma(\eta^t)
   \indep \overline{\F_t}\supset \F_t^u\vee\F_t^\eta\supset \F_t^\eta. \label{eq: eta split3}   
   \end{align}     
The independence in \eqref{eq: eta split3} holds by Lemma \ref{lem:sol on completion} and  measurability of the (continuous) restriction $\cdot|_{\bar{I}_{T-t}}\col C(\R_+;\R^\infty)\to R^t$. 
By \eqref{eq: eta split1}, there exists a Borel measurable map $g\col R_t\times R^t\to S_2$ such that $\eta=g(\eta_t,\eta^t)$ \cite[Lem.\ 1.14]{kallenberg21}. Combining with \eqref{eq: eta split2} and \eqref{eq: eta split3} yields for all $h\in\BM(S_2,\BB(S_2))$:
\begin{align*}
   \E[h(\eta)|\F_t^u\vee\F_t^\eta] =\E[h\circ g(\eta_t,\eta^t)|\F_t^u\vee\F_t^\eta] 
   = \int_{R^t}h\circ g(\eta_t,w) \P\#\eta^t (\ddd w) 
   &= \E\big[ h\circ g(\eta_t, \eta^t)  |\F_t^\eta\big]\\
   &= \E\big[ h(\eta)  |\F_t^\eta\big],
\end{align*}
where we use that $h\circ g(\eta_t,\eta^t)$ has one measurable component and one independent component, for the $\sigma$-algebras $\F_t^u\vee\F_t^\eta$ and $\F_t^\eta$. 
\end{proof} 

The following lemma shows how joint compatibility may be of use. Recall that $\wien$ is the law of a continuous $\R^\infty$-Brownian motion (see Remark \ref{rem:wien measure}). 
\begin{lemma}\label{lem:joint compatible indep incr}
Let Assumption \ref{ass} hold. Let $(\Om,\F,\P)$ be a probability space. 
 Let  $u_0\col\Om\to Z$ and $W\col \Om\to \W$ be random variables with  $\law(u_0,W)=\mu\otimes\wien$. Put $\eta\ceqq(u_0,W)$. Then:
\begin{enumerate}[label=\textup{(\roman*)},ref=\textup{(\roman*)}]
  \item \label{it:joint compatible indep incr1}
  $W$ is an $\R^\infty$-Brownian motion with respect to $(\Om,\overline{\F},\overline{\P},(\F_t^\eta))$.  
\end{enumerate}

In addition, let $\CC$ be defined by \eqref{eq:temporal} and let $u,\wtu\col \Om\to \B$ be random variables such that $(u,\wtu)\com \eta$. Then:
\begin{enumerate}[label=\textup{(\roman*)},ref=\textup{(\roman*)},resume]
\item \label{it:joint compatible indep incr2} $\F_s^u\vee\F_s^{\wtu} \vee\F_s^\eta \indep W(t)-W(s)\col\Om\to\R^\infty$ for all $0\leq s\leq t\in\bar{I}_T$. 
    \end{enumerate}
\end{lemma}
\begin{proof}
\ref{it:joint compatible indep incr1}: Since $\law(u_0,W)=\mu\otimes\wien$, we have $u_0\indep W$ and $\law(W)=\wien$. Therefore, $W$ has the distributional properties of an $\R^\infty$-Brownian motion  and thus also the same independence structure with respect to $(\F_t^W)$. In particular, $W(t)-W(s)\indep \F_s^W$ for all $0\leq s<t$.  
Moreover,  $u_0\indep W$ gives $\sigma(W(t)-W(s))\vee \F_s^W\indep u_0$. 
Combining these yields $\sigma(W(t)-W(s))\indep \F_s^W\vee\sigma(u_0)$ and recalling \eqref{eq:F_t^eta}, we find: 
\begin{equation}\label{eq:indepGamma2}
  W(t)-W(s)\indep \F_s^\eta \text{ for all }0\leq s<t. 
\end{equation}
Since $\R^\infty\to\R\col x\mapsto x_k$  is measurable, \eqref{eq:indepGamma2}   implies that  $W^k(t)-W^k(s)\indep \F_s^\eta$, so each $W^k$ is a standard real-valued  $(\F_t^\eta)$-Brownian motion. We conclude that  \ref{it:joint compatible indep incr1} holds. 

\ref{it:joint compatible indep incr2}:
Write $\F_s\ceqq \F_s^u\vee\F_s^{\wtu} \vee\F_s^\eta$. 
For all $B\in\BB(\R^\infty)$, $A_1\in\F_s^{u}$, $A_2\in\F_s^{\wtu}$ and $D\in\F_s^\eta$, we have 
\begin{align*}
  \E[\one_B(W(t)-W(s))\one_{A_1}\one_{A_2}\one_D] &=\E[\E[\one_B(W(t)-W(s))\one_{A_1}\one_{A_2}\one_D|\F_s]]\notag\\
  &=\E[\one_{A_1}\one_{A_2}\one_D\E[\one_B(W(t)-W(s))|\F_s]]\notag\\
  &=\E[\one_{A_1}\one_{A_2}\one_D\E[\one_B(W(t)-W(s))|\F_s^\eta]]\notag\\
  &=\E[\one_B(W(t)-W(s))]\E[\one_{A_1}\one_{A_2}\one_D]. 
\end{align*}
Here, joint  $\CC$-compatibility is used in the third line and \eqref{eq:indepGamma2} is used in the fourth line.  A  monotone class argument now yields $W(t)-W(s)\indep \F_s$. 
\end{proof}

We now prove the Yamada--Watanabe--Engelbert theorem for SPDEs, Theorem \ref{th:YW SPDE}. 
Recall Definition \ref{def: SPDE C-sol} for the notions of solution and uniqueness. 

\begin{proof}[\textbf{Proof of Theorem \ref{th:YW SPDE}}]
Let us write  $\eq\ceqq$ \eqref{SPDE} if $\cond$ is a subcollection of conditions \ref{it:sol integrability}--\ref{it:as cond}, and   $\eq\ceqq$  \eqref{SPDE mild} if $\cond$ is a subcollection of conditions \ref{it:moment cond}--\ref{it:integrated eq mild weak}.   
Throughout the proof we identify $U$-cylindrical Brownian motions $W$ with $(W^k)$ through \eqref{eq: def Rinfty BM}, recalling that $(e_k)$ is fixed by Assumption \ref{ass}. 
We will apply Theorem \ref{th:YW} with a suitable choice of $(\Gamma, \CC,\nu)$. 

Let  $S_1, S_2, \BB_t^{S_i}, \CC$ be as in \eqref{eq: S1 S2}, \eqref{eq:temporal} and put
\begin{equation*} 
   \nu\ceqq\mu\otimes \wien.
\end{equation*} 
Now we define  $\Gamma\subset \PP(S_1\times S_2)$ in such a way that `$\law(u\one_\B(u),(u(0),W))\in\Gamma$' ensures that the conditions in $\cond$ are satisfied. Let $B$, $D$, $D_t^\gamma$ be as in Lemma \ref{lem: mble L^p subsets} and let $B_{z^*}$, $D_{z^*}$, $B_{t}^S$, $D_{t}^{S}$, $D_{t}^{\gamma,S}$, $B_{z^*,t}^S$, $D_{z^*,t}^S$ be as in Corollary \ref{cor: mble L^p subsets}. 
Let $\Pi_{1}\col S_1\times S_2\to \B\times \W\col (v,(x,w))\mapsto v$, let $\Pi_{2}\col S_1\times S_2\to \B\times \W\col (v,(x,w))\mapsto x$ and let $\Pi_{13}\col S_1\times S_2\to \B\times \W\col (v,(x,w))\mapsto (v,w)$. Note that these maps are Borel measurable. 
Put
\begin{equation}\label{eq: def Gamma} 
\Gamma\ceqq \Gamma_0\cap \bigcap_{(k)\in \cond, \, 1\leq k\leq 12}\Gamma_{(k)}, 
\end{equation} 
\begin{itemize}
\item $\Gamma_0\ceqq\{\tilde{\mu}\in\PP(S_1\times S_2):\tilde{\mu}(\pi_1^0\circ\Pi_1=\Pi_2)=1\}$,
\item $\Gamma_{\ref{it:sol integrability}}\ceqq\{\tilde{\mu}\in\PP(S_1\times S_2):\Pi_1\#\tilde{\mu}(B\cap D)=1\}$, 
\item $\Gamma_{\ref{it:sol integrability UMD}}\ceqq\cap_{t\in\bar{I}_T}\{\tilde{\mu}\in\PP(S_1\times S_2):\Pi_1\#\tilde{\mu}(B\cap D_t^\gamma)=1\}$, 
\item $\Gamma_{\ref{it:sol integrability ana weak}}\ceqq\cap_{z^*\in D}\{\tilde{\mu}\in\PP(S_1\times S_2):\Pi_1\#\tilde{\mu}(B_{z^*}\cap D_{z^*})=1\}$, 
\item $\Gamma_{\ref{it:sol integrability mild}}\ceqq\cap_{t\in \bar{I}_T}\{\tilde{\mu}\in\PP(S_1\times S_2):\Pi_1\#\tilde{\mu}(B_{t}^{S}\cap D_{t}^{S})=1\}$, 
\item $\Gamma_{\ref{it:sol integrability mild UMD}}\ceqq\cap_{t\in \bar{I}_T}\{\tilde{\mu}\in\PP(S_1\times S_2):\Pi_1\#\tilde{\mu}(B_{t}^{S}\cap D_{t}^{\gamma,S})=1\}$, 
\item $\Gamma_{\ref{it:sol integrability mild weak}}\ceqq\cap_{t\in \bar{I}_T,z^*\in D}\{\tilde{\mu}\in\PP(S_1\times S_2):\Pi_1\#\tilde{\mu}(B_{z^*,t}^S\cap D_{z^*,t}^S)=1\}$, 
\item for $(k)=$\ref{it:integrated eq},\ref{it:integrated eq ana weak},\ref{it:integrated eq mild},\ref{it:integrated eq mild weak}: $\Gamma_{(k)}\ceqq\{\tilde{\mu}\in\PP(S_1\times S_2):\Pi_{13}\#\tilde{\mu}\in A_k\}$,  where $A_k\subset\PP(\B\times \W)$ denotes the relevant set of Corollary \ref{cor:SPDE distributional unified},
\item $\Gamma_{\ref{it:moment cond}}\ceqq \{\tilde{\mu}\in\PP(S_1\times S_2):\int_{\B} |f|^p\dd (\Pi_1\#\tilde{\mu})<\infty\}$,  
\item $\Gamma_{\ref{it:as cond}}\ceqq \{\tilde{\mu}\in\PP(S_1\times S_2): \Pi_1\#\tilde{\mu}(E)=1\}$. 
   \end{itemize}  
   
Now, the actual proof will be given in Steps 1--5 below. 
Recall that $\SS_{\Gamma,\CC,\nu}\subset \PP(S_1\times S_2)$ is defined by \eqref{eq: def Kurtz set}.   See Definition \ref{def: kurtz solution} for the notions of weak and strong solutions for $(\Gamma,\CC,\nu)$. In Steps 1--3, we relate these to the $\cond$-weak and $\cond$-strong solutions of  Definition \ref{def: SPDE C-sol}. In Steps 4--5, we derive the theorem. 
In what follows, we use the notations $\F_t^u,\F_t^\eta,\F_t^W$ from \eqref{eq:F_t^eta}. 

\vspace{.15cm}

\textbf{Step 1} (Role of $\nu$)

\noindent
For any random variable  $\eta=(u_0,W)\col \Om\to Z\times \W$ with $\law(\eta)=\nu$, we have $\law(u_0)=\mu$ and $W$ is an $\R^\infty$-Brownian motion with respect to $(\Om,\overline{\F},\overline{\P},(\F_t^\eta))$, due to Lemma \ref{lem:joint compatible indep incr}\ref{it:joint compatible indep incr1}. 

Conversely,  if $u_0\in L^0_{\F_0}(\Om;Z)$, $\law(u_0)=\mu$ and $W$ is a $U$-cylindrical Brownian motion with respect to $(\F_t)$, then Lemma \ref{lem:sol on completion} gives $\sigma(u_0)\subset \F_0 \indep W(0+\cdot)-W(0)$. As $W(0+\cdot)-W(0)=W$ a.s., we conclude that $u_0\indep W$, hence  $\law(u_0,W)=\mu\otimes{\wien}=\nu$.     

\vspace{.15cm}

\textbf{Step 2} (Relations between weak solutions)

\noindent
For any probability space $(\Om,\F,\P)$ and measurable $(u,(u_0,W))\col\Om\to S_1\times S_2$, we have: 
\begin{align}\label{eq:weaksol2}
\begin{split}
&\law(u,(u_0,W))\in \SS_{\Gamma,\CC,\nu}\implies \\
&\qquad\quad (u,W,(\Om,\overline{\F},\overline{\P},(\F_t^u\vee\F_t^W)))  \text{ is a }\cond\text{-weak solution to $\eq$ with }\law(u(0))=\mu. 
\end{split}
\end{align} 
Indeed, suppose that $\law(u,(u_0,W))\in\SS_{\Gamma,\CC,\nu}$.   
Trivially, $u$ and $W$ are $(\F_t^u\vee\F_t^W)$-adapted. 
Moreover,  $\law(u_0,W)=\nu$, so $\law(u_0)=\mu$ and $\Gamma_0$ yields $u(0)=u_0$ a.s., hence  $\law(u(0))=\mu$. 
By construction of $\Gamma$, the conditions in $\cond$ hold if   $W$ is a $U$-cylindrical Brownian motion with respect to $(\Om,\overline{\F},\overline{\P},(\F_t^u\vee\F_t^W))$. It remains to prove the latter. 
By Step 1, $W$ is an $\R^\infty$-Brownian motion with respect to $(\Om,\overline{\F},\overline{\P},(\F_t^\eta))$,  thus it extends uniquely to a $U$-cylindrical Brownian motion $W$ on $(\Om,\overline{\F},\overline{\P},(\F_t^\eta))$ satisfying \eqref{eq: def Rinfty BM} (see the lines below \eqref{eq: def Rinfty BM}). 
It suffices to check that $W(t)-W(s)\indep \F_s^u\vee\F_s^\eta$. But this follows from the $\CC$-compatibility $u\com (u_0,W)$ and application of Lemma \ref{lem:joint compatible indep incr}\ref{it:joint compatible indep incr2}  with  $\tilde{u}=u$.  
This completes the proof of \eqref{eq:weaksol2}.  

Conversely, we have:
\begin{align}\label{eq:weaksol}
\begin{split}
&(u,W,(\Om,\F,\P,(\F_t))) \text{  is a $\cond$-weak solution to $\eq$ with } \law(u(0))=\mu \implies\\
&  \qquad\quad \law(\bar{u},(\bar{u}(0),W))\in \SS_{\Gamma,\CC,\nu} \text{ for $\bar{u}\ceqq u\one_{\B}(u)$}.
\end{split}
\end{align}
If $(u,W,(\Om,\F,\P,(\F_t)))$ is a $\cond$-weak solution to $\eq$ and $\law(u(0))=\mu$, then we have $\law({u}(0),W)=\nu$ by Step 1. Furthermore,  
$\law(\bar{u},(\bar{u}(0),W))\in \Gamma$ by construction of $\Gamma$ and $\bar{u}\com (\bar{u}(0),W)$ by Lemma \ref{lem:SPDE sol is compatible}. Hence, $\law(\bar{u},(\bar{u}(0),W))\in \SS_{\Gamma,\CC,\nu}$ as claimed. 

\vspace{.15cm}

{\textbf{Step 3} (Coincidence of strong solutions and uniqueness notions) }

\noindent
We show the following:  
\begin{enumerate}[label=\textup{(\Roman*)},ref=\textup{(\Roman*)}]  
\item \label{it:strongsol} For any filtered probability space $(\Om,\F,\P,(\F_t))$, for any measurable $u\col \Om\to S_1$ with $u(0)\in L^0_{\F_0}(\Om;Z)$ and for any $U$-cylindrical Brownian motion $W$ on $(\Om,\F,\P,(\F_t))$:  

\noindent
    $(u,W,(\Om,\overline{\F},\overline{\P},(\overline{\F_t})))$ is a $\cond$-strong solution to $\eq$ with $\law(u(0))=\mu$  
    
    \hspace{3.2cm}$\iff$ $(u,(u(0),W))$ is a strong solution for $(\Gamma,\CC,\nu)$.  
\item \label{it:path} Pathwise uniqueness holds for $\cond$-weak solutions to $\eq$ with initial law $\mu$ $\iff$ pointwise uniqueness holds in $\SS_{\Gamma,\CC,\nu}$. 
\item \label{it:jointweakunique} Joint weak uniqueness holds for $\cond$-weak solutions to $\eq$ with initial law $\mu$  $\iff$ $\#\SS_{\Gamma,\CC,\nu}\leq 1$ (\ie joint uniqueness in law holds).
\end{enumerate}  

\ref{it:strongsol} `$\Leftarrow$':  
 Since $(u,(u(0),W))$ is a strong solution for $(\Gamma,\CC,\nu)$, we have a  Borel measurable map   $F\col S_2\to S_1$ such that $F(u(0),W)=u$ a.s. It remains to show that $(u,W,(\Om,\F,\P,(\F_t)))$ is a $\cond$-weak solution to $\eq$. Since  $\law(u,(u(0),W))\in \SS_{\Gamma,\CC,\nu}\subset \Gamma$ and  $W$ is assumed to be a $U$-cylindrical Brownian motion with respect to $(\F_t)$, $\cond$ holds. 
 Moreover,  $W$ is $(\F_t)$-adapted and $u(0)$ is $\F_0$-measurable, so $\F_t^{(u(0),W)}\subset \overline{\F_t}$ (recall \eqref{eq:F_t^eta}). Also, \cite[Prop.\ 2.13]{kurtz14} gives $\F_t^u\subset \F_t^{(u(0),W)}$ since $(u,(u(0),W))$ is a strong, $\CC$-compatible solution. 
 Thus, 
 $\sigma(u(s):s\in\bar{I}_t)\subset \F_t^u\subset \F_t^{(u(0),W)}\subset \overline{\F_t}$, \ie $u$ is $(\overline{\F_t})$-adapted. We conclude that $(u,W,(\Om,\overline{\F},\overline{\P},(\overline{\F_t})))$ is a $\cond$-strong solution to $\eq$. 

\ref{it:strongsol} `$\Rightarrow$': Let $(u,W,(\Om,\overline{\F},\overline{\P},(\overline{\F_t})))$ be  a $\cond$-strong solution with initial law $\mu$ and $u=F(u(0),W)$ a.s. for some measurable $F\col S_2\to S_1$. Then by \eqref{eq:weaksol}, $(u,(u(0),W))$ is a weak solution for $(\Gamma,\CC,\nu)$, and automatically strong by the same map $F$.

\ref{it:path} `$\Rightarrow$': By Corollary \ref{cor:compatible}, it suffices to prove pointwise uniqueness for \emph{jointly} $\CC$-compatible solutions. Thus, let $(u,\eta)$ and $({\wtu},\eta)$ be solutions for $(\Gamma,\CC,\nu)$ on the same probability space $(\Om,\F,\P)$ with $(u,{\wtu})\com \eta$. We have to show that $u={\wtu}$ a.s. Write $\eta=(u_0,W)$ with random variables $u_0\col\Om\to Z$ and $W\col \Om\to \W$ and note that $u(0)=u_0=\wtu(0)$ \as due to $\Gamma_0$. 
Define $\F_t\ceqq \F_t^{u}\vee \F_t^{{\wtu}}\vee\F_t^\eta$. We show that $(u,W,(\Om,\overline{\F},\overline{\P},(\F_t)))$ is a $\cond$-weak solution, and similarly for $\wtu$. Then the assumed pathwise uniqueness yields $u={\wtu}$ a.s. as needed. Adaptedness of $u$, ${\wtu}$ and $W$ is clear by definition of $\F_t$. Also, $W$ has independent increments with respect to $(\F_t)$ by Lemma \ref{lem:joint compatible indep incr}\ref{it:joint compatible indep incr2}  and $\law(W)=\wien$, so $W$ is a $U$-cylindrical Brownian motion on  $(\Om,\overline{\F},\overline{\P},(\F_t))$. By construction of $\Gamma$, $\cond$ holds, completing the argument. 

\ref{it:path} `$\Leftarrow$': Let $(u_i,W,(\Om,\F,\P,(\F_t)))$ be a $\cond$-weak solution to $\eq$ for $i=1,2$, with $u_1(0)=u_2(0)$ a.s. and $\law(u_i(0))=\mu$. Then \eqref{eq:weaksol} gives $\law(\bar{u_i},(\bar{u_i}(0),W))\in\SS_{\Gamma,\CC,\nu}$ for $i=1,2$. Hence pointwise uniqueness in $\SS_{\Gamma,\CC,\nu}$ yields a.s. $u_1=\bar{u_1}=\bar{u_2}=u_2$.
 
\ref{it:jointweakunique} '$\Rightarrow$': Let $\mu_1,\mu_2\in\SS_{\Gamma,\CC,\nu}$. Pick any measurable $(u_i,(u_{0,i},W_{0,i}))\col \Om^i\to S_1\times S_2$ on  probability spaces $(\Om^i,\F^i,\P^i)$ such that $\mu_i=\law(u_i,(u_{0,i},W_{0,i}))$ for $i=1,2$. 
Thanks to \eqref{eq:weaksol2}, 
$(u_i,W_{0,i},(\Om^i,\overline{\F^i},\overline{\P^i}, (\F_t^{u_i}\vee\F_t^{W_{0,i}})))$  are  $\cond$-weak solutions with $\law(u_i(0))=\mu$. Joint weak uniqueness yields $\law(u_1,W_{0,1})=\law(u_2,W_{0,2})$. 
Since $\Gamma\subset\Gamma_0$, we have $u_i(0)=u_{0,i}$ a.s., thus   $\mu_1=\law(u_1,(u_1(0),W_{0,1}))=\law(u_2,(u_2(0),W_{0,2}))=\mu_2$. 
We conclude that $\#\SS_{\Gamma,\CC,\nu}\leq 1$.   

\ref{it:jointweakunique} '$\Leftarrow$': Suppose that $\#\SS_{\Gamma,\CC,\nu}\leq 1$. Let $(u_i,W_{0,i},(\Om^i,{\F^i},{\P^i}, (\F_t^i)))$ be $\cond$-weak solutions such that $\law(u_i(0))=\mu$ for $i=1,2$. By \eqref{eq:weaksol}, $\law(\bar{u_i},(\bar{u_i}(0),W_{0,i})) \in \SS_{\Gamma,\CC,\nu}$, 
thus  $\law(\bar{u_1},(\bar{u_1}(0),W_{0,1}))=\law(\bar{u_2},(\bar{u_2}(0),W_{0,2}))$ and hence $\law(u_1,W_{0,1})=\law(u_2,W_{0,2})$. 

\vspace{.15cm}

{\textbf{Step 4} (Equivalence of  \ref{it:YW1SPDE}, \ref{it:YW2SPDE} and \ref{it:YW3SPDE})}

\noindent
Using Steps 2 and 3, we  prove the equivalences claimed in the theorem. 

\ref{it:YW1SPDE}$\Rightarrow$\ref{it:YW3SPDE}: 
By \ref{it:YW1SPDE} and \eqref{eq:weaksol}, we have $\SS_{\Gamma,\CC,\nu}\neq\varnothing$. Also, \ref{it:YW1SPDE} and \ref{it:path}`$\Rightarrow$' imply that pointwise uniqueness holds in $\SS_{\Gamma,\CC,\nu}$. Thus Theorem \ref{th:YW}\ref{it:YW1Kurtz}$\Rightarrow$\ref{it:YW3Kurtz} yields that joint uniqueness in law holds, hence, by \ref{it:jointweakunique}`$\Leftarrow$', joint weak uniqueness holds. Also, Theorem \ref{th:YW}\ref{it:YW1Kurtz}$\Rightarrow$\ref{it:YW3Kurtz} yields existence of a measurable map $F\col S_2\to S_1$ such that for any measurable $\eta\col\Om\to S_2$ with $\law(\eta)=\nu$: $(F(\eta),\eta)$ is a strong solution for $(\Gamma,\CC,\nu)$. Hence, for any filtered probability space $(\Om,\F,\P, (\F_t))$,  any $U$-cylindrical Brownian motion $W$ on $(\Om,\F,\P, (\F_t))$ and any  $u_0\in  L^0_{\F_0}(\Om;Z)$ with $\law(u_0)=\mu$, we can put $\eta\ceqq(u_0,W)$, $u\ceqq F(u_0,W)$ and $F_\mu\ceqq F$, to  obtain that  $(F_{\mu}(u_0,W),(u_0,W))$ is a strong solution for $(\Gamma,\CC,\nu)$, noting that $\law(u_0,W)=\nu$ by Step 1.  
Thus, \ref{it:strongsol}`$\Leftarrow$' yields that $(F_{\mu}(u_0,W),W,(\Om,\overline{\F},\overline{\P}, (\overline{\F_t})))$ is a $\cond$-strong solution to $\eq$ with initial law $\mu$. Due to $\Gamma_0$ in the construction of $\Gamma$, we have $u(0)=F_{\mu}(u_0,W)(0)=u_0$ a.s. We conclude that 
the second (hence the first) statement of \ref{it:YW3SPDE} holds.

\ref{it:YW3SPDE}$\Rightarrow$\ref{it:YW2SPDE}: Trivial. 

\ref{it:YW2SPDE}$\Rightarrow$\ref{it:YW1SPDE}: 
By \ref{it:YW2SPDE}, there exists a $\cond$-strong solution $(u,W,(\Om,\overline{\F},\overline{\P}, (\overline{\F_t})))$, which is trivially also a $\cond$-weak solution. 
Moreover, \ref{it:jointweakunique}`$\Rightarrow$' yields joint uniqueness in law and \ref{it:strongsol}`$\Rightarrow$' yields that $(u,(u(0),W))$ is a strong solution for $(\Gamma,\CC,\nu)$. Now Theorem \ref{th:YW}\ref{it:YW2Kurtz}$\Rightarrow$\ref{it:YW1Kurtz} gives pointwise uniqueness. Then, 
\ref{it:path}`$\Leftarrow$' yields pathwise uniqueness for solutions with initial law $\mu$ and we conclude that \ref{it:YW1SPDE} holds.

\vspace{.15cm}

{\textbf{Step 5} (Proof of \ref{it:YW_Fmu}: additional measurability of $F_\mu$)}

\noindent
Let $t\in\bar{I}_T$. We show that $F_\mu$ is $\overline{\BB_t^{S_2}}^\nu/\BB_t^{S_1}$-measurable. 
By definition of $\BB_t^{S_1}$, this holds if and only if 
$\pi_1^s\circ F_\mu$ is $\overline{\BB_t^{S_2}}^\nu/\BB(Z)$-measurable for all $s\in\bar{I}_t$. 
The latter is equivalent to having for $\B^t\ceqq C(\bar{I}_t;Z)$ and  $\Pi_t\col \B\to \B^t\col u\mapsto u|_{\bar{I}_t}$: 
\begin{align}\label{eq:ts proj2}
\Pi_t\circ F_\mu \text{ is }\overline{\BB_t^{S_2}}^\nu/\BB(\B^{t})\text{-measurable}. 
\end{align}
Indeed, ${\pi}_{1,t}^s\col \B^t\to Z\col v\mapsto v(s)$ satisfies  ${\pi}_{1,t}^s(\Pi_t\circ u)=\pi_1^s(u)$ for $s\in\bar{I}_t$, and $\{{\pi}_{1,t}^s:s\in\bar{I}_t\}$ separates the points in the Polish space $\B^t$, so  Lemma \ref{lem: generating maps} gives that \eqref{eq:ts proj2} is   equivalent.

Let $\eta\col\Om\to S_2$ be a random variable on a probability space $(\Om,\F,\P)$, with $\law(\eta)=\mu\otimes \wien$. Consider $\eta_t\col \Om\to (S_2,\BB_t^{S_2})\col \om\mapsto \eta(\om)$. Note that $\eta_t$ is measurable since $\BB_t^{S_2}\subset \BB(S_2)$. Define $u\ceqq F_\mu(\eta)$, so that $(u,\eta)$ is a strong solution for $(\SS,\Gamma,\nu)$ (recall that we put $F_\mu\ceqq F$ in the proof of \ref{it:YW1SPDE}$\Rightarrow$\ref{it:YW3SPDE} above). 
Moreover, by \cite[Prop.\ 2.13]{kurtz14} and \eqref{eq:F_t^eta}:  $\sigma(\Pi_t\circ u)\subset\F_t^u\subset \F_t^\eta=\overline{\sigma(\eta_t)}$, \ie  $\Pi_t\circ u\col(\Om,\overline{\sigma(\eta_t)})\to(\B^t,\BB(\B^t))$ is measurable. As $\B^t$ is Polish, it is a Borel space \cite[Th.\ 1.8]{kallenberg21}. Hence, by first applying Lemma \ref{lem: mble version of map} and then \cite[Lem.\ 1.14]{kallenberg21}, we obtain a measurable map $F_\mu^t\col (S_2,\BB_t^{S_2})\to(\B^t,\BB(\B^t))$ such that $\Pi_t\circ u=F_\mu^t(\eta_t)$ $\P$-a.s. 
Then, since $\Pi_t\circ F_\mu(\eta)=\Pi_t\circ u\overset{\text{a.s.}}{=}F_\mu^t(\eta_t)=F_\mu^t(\eta)$, it follows that 
\begin{align}\label{eq:nu-version}
\nu(\Pi_t\circ F_\mu=F_\mu^t)=\law(\eta)(\Pi_t\circ F_\mu=F_\mu^t)=\P(\Pi_t\circ F_\mu(\eta)=F_\mu^t(\eta))=1.
\end{align}
Here we use that $\{\Pi_t\circ F_\mu=F_\mu^t\}\in\BB(S_2)$, due to the fact that $\{(v,v):v\in \B^t\}\in\BB(\B^t)\otimes\BB(\B^t)$ since $\B^t$ is a separable metric space. 
Thus by \eqref{eq:nu-version}, $F_\mu^t$ is a $\BB_t^{S_2}/\BB(\B^t)$-measurable $\nu$-version of $\Pi_t\circ F$. 
Lemma \ref{lem: mble version of map} now yields that \eqref{eq:ts proj2} holds, from which we conclude that $F_\mu$ is $\overline{\BB_t^{S_2}}^\nu/\BB_t^{S_1}$-measurable, by the earlier considerations. 
\end{proof}

Although not needed for the results in \cite{kurtz14} (and for Theorem \ref{th:YW}), one can verify that the set $\Gamma$ is in fact convex, as was required in the earlier version \cite{kurtz07}. For $\Gamma_{(k)}$ with $k=4,5,10,11$, one may use the approximations of Lemma \ref{lem:stoch int rep weak} and revisit Theorem \ref{th:stoch int rep type 2 or UMD} and Corollaries \ref{cor: diffusion drift mble rep}, \ref{cor:SPDE distributional unified} to show convexity of the sets $A_k$ (see also \cite[Ex.\ 3.18, (3.11)]{kurtz07}).

\begin{remark}\label{rem: mathbbW choice YW} 
In the notion of $\cond$-strong solution and in Theorem \ref{th:YW SPDE} above, we can replace the $\R^\infty$-Brownian motion by the $Q_1$-Wiener process $W_1$ induced by $W$ through \eqref{eq: def W_1} and replace $\W$  by $\W_1\ceqq C(\bar{I}_T;U_1)$. 

Indeed, let $\phi$ and $\phi_1$ be as in Lemma \ref{lem: sigma algs W and W_1 new} and define $G_1\ceqq(\Id_Z,\phi_1)\col Z\times\W_1\to Z\times\W$ and $G\ceqq (\Id_Z,\phi)\col Z\times\W\to Z\times\W_1$. By Lemma \ref{lem: sigma algs W and W_1 new} (and since $\BB(Z\times \W_1)=\BB(Z)\otimes \BB(\W_1)$ by separability),  
 $G$ and $G_1$ are measurable and satisfy   $(u_0,W|_{\bar{I}_T})=G_1(u_0,W_1|_{\bar{I}_T})$ \as and 
$(u_0,W_1|_{\bar{I}_T})=G(u_0,W|_{\bar{I}_T})$ a.s. Thus, the notion of $\cond$-strong solution is equivalent for $W=(W^k)$ and $W_1$ (keeping $W$ in the stochastic integrals), and Theorem \ref{th:YW SPDE} holds in the second case as well, with $(W,\W)$ replaced by $(W_1,\W_1)$ in Theorem \ref{th:YW SPDE}\ref{it:YW3SPDE}. For part \ref{it:YW3SPDE}, we can use  $F_\mu\circ G_1$ as solution map, because $\phi_1$ is constructed independently of $W$ and $W_1$ in Lemma \ref{lem: sigma algs W and W_1 new}. The proof of \ref{it:YW_Fmu} gives $\overline{\BB_t^{Z\times \W_1}}^{\mu\otimes \mathrm{P}^{Q_1}}/\BB_t^{S_1}$-measurability of this solution map, where $\mathrm{P}^{Q_1}$ denotes the $Q_1=JJ^*$-Wiener measure (see Remark \ref{rem: Q-wiener}).  
\end{remark}

The Yamada--Watanabe--Engelbert theorem (Theorem \ref{th:YW SPDE}) is stated for an arbitrary fixed initial law $\mu$ and yields a $\mu$-dependent strong solution map $F_\mu$. In contrast, the classical Yamada--Watanabe theorem \cite[Th.\ 1.1 Chap.\ 4]{ikedawatanabe81} is a statement for all fixed initial laws together and it gives a `unique strong solution' in a stronger sense: see \cite[Def.\ 1.6 Chap.\ 4, p.\ 149]{ikedawatanabe81} (SDEs) 
and \cite[Def.\ 1.9]{rock08} (SPDEs).   

We now  derive an analog of the classical Yamada--Watanabe theorem. 
Namely, we suppose that \ref{it:YW1SPDE} of Theorem \ref{th:YW SPDE} is satisfied for all initial laws $\mu$ in a certain subcollection. Then it gives that all corresponding $\CC$-weak solutions are up to null sets determined by a \emph{single, $\mu$-independent  map }$F$. Here, a space $Z_0$ is introduced, since   random initial values in $Z$ may not have enough regularity to give rise to a solution (recall Assumption \ref{ass}). Also, one may for example only consider initial laws satisfying some moment conditions, thus we state the result for $M\subset \PP(Z_0)$.  

The proof was partly inspired by \cite[Lem.\ 2.11]{rock08}. We identify  $\mu\in \PP(Z_0)$ with its trivial extension $\mu(\cdot\cap Z_0)\in \PP(Z)$ (\eg when writing $F_\mu$), using that $Z_0\into Z$ and identifying $Z_0$ with its image.  
 
\begin{corollary} \label{cor: single F} 
Let the assumptions of Theorem \ref{th:YW SPDE} hold.  
Let $Z_0$ be a topological space such that $Z_0\into Z$.  
Let $M\subset \PP(Z_0)$ be such that $\delta_{z}\in M$ for all $z\in Z_0$. Suppose that for all $\mu\in M$, \ref{it:YW1SPDE} (hence \ref{it:YW3SPDE}) is satisfied. Let $F_\mu$ be the corresponding map of \ref{it:YW3SPDE}. Then:
\begin{enumerate}[label=\emph{(\roman*)},ref=\textup{(\roman*)}]
  \item\label{it:unistrong1} For any $\cond$-weak solution $(u,W,(\Om,\F,\P, (\F_t)))$ to \eqref{SPDE}, respectively \eqref{SPDE mild}, with $\law(u(0))\in M$, we have: 
      $u=F_{\law(u(0))}(u(0),W)$ $\P$-a.s.
  \item\label{it:unistrong2} There exists a map $F\col Z_0\times \W\to \B$ such that $F(z,\cdot)$ is $\overline{\BB_t(\W)}^{{\wien}}/\BB_t(\B)$-measurable for all $z\in Z_0$ and $t\in \bar{I}_T$, and  for all $\mu\in M$, we have for $\mu$-\aev $z\in Z_0  $: 
      
      \noindent
      $F_\mu(z,\cdot)=F(z,\cdot)$ ${\wien}$-a.e.   
\end{enumerate} 
\end{corollary}
\begin{proof}
Claim \ref{it:unistrong1} follows by pathwise uniqueness. 

For \ref{it:unistrong2}, the following map will meet the conditions:
\[
F(z,w)\ceqq F_{\delta_z}(z,w),\qquad z\in Z_0,w\in\W.
\]
By the $\overline{\BB_t^{S_2}}^{\delta_z\otimes\wien}/\BB_t^{S_1}$-measurability of $F_{\delta_z}$ from Theorem \ref{th:YW SPDE} and by Fubini's theorem, $F(z,\cdot)$ is $\overline{\BB_t(\W)}^{{\wien}}/\BB_t(\B)$-measurable. 

Fix arbitrary $\mu\in M$. It remains to show that for $\mu$-\aev $z\in Z$: $F_\mu(z,\cdot)=F(z,\cdot)$ ${\wien}$-a.e.  Let $\nu=\mu\otimes\wien$, $\Gamma$, $\CC$ and $F_\mu$ be as in the proof of Theorem \ref{th:YW SPDE}. 
Define for $t\in\R_+$ and $z\in Z_0$:
\[
\Om\ceqq \B\times Z\times\W,\quad \F_t\ceqq \F\ceqq \BB(\Om),\quad Q_z(\ddd b,\ddd \tilde{z},\ddd w)\ceqq\delta_{F_\mu(\tilde{z},w)}(\ddd b)\wien(\ddd w)\delta_z(\ddd \tilde{z}).
\]
Let $u\col\Om\to \B$ and $(u_0,W)\col\Om\to Z\times \W$ be given by
\[
u(b,\tilde{z},w)\ceqq b, \quad u_0(b,\tilde{z},w)\ceqq \tilde{z},\quad W(b,\tilde{z},w)\ceqq w.
\]
These maps are measurable and by construction,  
\begin{align}
&\law(u_0,W)=(u_0,W)\#Q_z=\delta_z \otimes\wien\eqqc \nu_z. \label{eq:delta_z x wien}
\end{align} 
Then $\law(u_0)=\delta_z$ and Lemma \ref{lem:joint compatible indep incr}\ref{it:joint compatible indep incr1} gives that $W$ is an $\R^\infty$-Brownian motion with respect to  $(\Om,\overline{\F},\overline{Q_z},(\F_t^{(u_0,W)}))$. By Theorem \ref{th:YW SPDE}\ref{it:YW3SPDE},  
$(F_{\delta_z}(u_0,W),W,(\Om,\overline{\F},\overline{Q_z},(\F_t^{(u_0,W)})))$ is a $\cond$-strong  solution  with $F_{\delta_z}(u_0,W)(0)=u_0$ \as 
 and \eqref{eq:weaksol} gives $\law(F_{\delta_z}(u_0,W),(u_0,W))\in\SS_{\Gamma,\CC,\nu_z}$  under $\overline{Q_z}$, hence under $Q_z$, for every $z\in Z_0$.   

We show that for $\mu$-\aev $z\in Z_0$, $\law(u,(u_0,W))\in \SS_{\Gamma,\CC,\nu_z}$ holds as well. Recall \eqref{eq: def Kurtz set} for the definition of $\SS_{\Gamma,\CC,\nu_z}$. Note that \eqref{eq:delta_z x wien} implies $\law(u,(u_0,W))\in\PP_{\nu_z}(\B\times(Z\times \W))$. 
It remains to show that $\law(u,(u_0,W))\in\Gamma$ and $u\com (u_0,W)$ under $Q_z$. Let  
\[
\phi^z\col Z\to Z\col \tilde{z}\mapsto z, \qquad Q^\mu(\ddd b,\ddd \tilde{z},\ddd w)\ceqq\delta_{F_\mu(\tilde{z},w)}(\ddd b)\wien(\ddd w) \mu(\ddd \tilde{z}).
\]
We have $\delta_z=\phi^z\#\mu$, so by change of variables, 
we have for all $B\in \BB(\Om)$: 
\begin{align}
Q_z(B)=\int_\Om \one_B\delta_{F_\mu(\tilde{z},w)}(\ddd b)\wien(\ddd w)(\phi^z\#\mu)(\ddd \tilde{z})
&=\int_\Om \one_B\circ (\Id_{\B},\phi^z,\Id_{\W})\delta_{F_\mu(\tilde{z},w)}(\ddd b)\wien(\ddd w) \mu(\ddd \tilde{z})\notag\\
&=(\Id_{\B},\phi^z,\Id_{\W})\#Q^\mu(B). \label{eq: Qz Qmu}
\end{align}
Viewing $(u_0,W)$ defined above as an adapted process on $(\Om,\F,Q^\mu,(\F_t^{(u_0,W)})$ and noting that its law equals $(u_0,W)\#Q^\mu=\mu\otimes\wien=\nu$ by definition of $Q^\mu$, 
Theorem \ref{th:YW SPDE}\ref{it:YW3SPDE} and \eqref{eq:weaksol} yield  that $\law(F_{\mu}(u_0,W),(u_0,W))\in\SS_{\Gamma,\CC,\nu}$.  Hence, $Q^\mu=\law(F_{\mu}(u_0,W),(u_0,W))\in\Gamma$. Recall the definitions of $\Gamma_{(k)}$ (see \eqref{eq: def Gamma}) 
and note that by \eqref{eq: Qz Qmu}, 
\begin{align*}
&\Pi_{13}\# Q_z=(\Pi_{13}\circ(\Id_{\B},\phi^z,\Id_{\W}))\#Q^\mu=\Pi_{13}\#Q^\mu, \\
&\Pi_{1}\# Q_z=(\Pi_{1}\circ(\Id_{\B},\phi^z,\Id_{\W}))\#Q^\mu=\Pi_{1}\#Q^\mu. 
\end{align*} 
Thus, using that $Q^\mu\in\Gamma$, we find that $Q_z\in\Gamma_{(k)}$ for all $(k)\in \cond$ and all $z\in Z_0$. 

Next, we show that for $\mu$-\aev $z\in Z_0$: $Q_z\in\Gamma_0$.  
Note that $\psi\col Z\times\W\to Z\times Z\col (z,w)\mapsto (z,F_\mu(z,w)(0))$ is joint measurable and $D\ceqq \{(z,z):z\in Z\}\in \BB(Z\times Z)$ since $Z$ is Polish. 
Therefore, Fubini's theorem gives measurability of 
$$
Z\to [0,1]\col z\mapsto \int_{\W}\one_D\circ \psi(z,\cdot) \dd\wien=\wien(B_z),
$$
where $B_z\ceqq \{z=F_\mu(z,\cdot)(0)\}\in\BB(\W)$. Integrating over $Z$ and using that $Q^\mu\in \Gamma\subset\Gamma_0$ yields 
\[
\int_Z \wien(B_z)\dd\mu(z)=\int_Z\int_{\W}\one_{\{z=F_\mu(z,w)\}} \dd\wien(w)\dd\mu(z)=Q^\mu(\Pi_2=\pi_1^0\circ\Pi_1)=1. 
\]
Since $\wien(B_z)\in[0,1]$ on $Z$, we conclude that $\wien(B_z)=1$ for $\mu$-\aev $z\in Z_0$. Consequently, for $\mu$-\aev $z\in Z_0$, we have 
\begin{align*}
Q_z(\pi_0^1\circ\Pi_1=\Pi_2)&= \int_Z\int_\W\int_\B \one_{\{\pi_0^1\circ\Pi_1=\Pi_2\}}(b,\tilde{z},w) \dd\delta_{F_\mu(\tilde{z},w)}(b)\dd\wien(w)\dd\delta_z(\tilde{z})=\wien(B_z)=1,
\end{align*}
\ie $Q_z\in\Gamma_0$. 
We conclude that for $\mu$-\aev $z\in Z_0$, $(u,u_0,W)\# Q_z=Q_z\in\Gamma$.

Lastly, we prove $\CC$-compatibility under $Q_z$, for every $z\in Z_0$. We will apply Lemma \ref{lem:SPDE sol is compatible} with filtration $(\F_t^u \vee\F_t^W)$. Adaptedness of $u$ and $W$ is clear. It remains to show that $W(t)-W(s)\indep  \F_s^u \vee\F_s^W$ under $Q_z$. 
By definition of $u$ and $W$, $\F_s^u \vee\F_s^W=\overline{\G}^{Q_z}$ (recall \eqref{eq:F_t^eta}), where 
\begin{align}\label{eq: gen sigma alg}
\G \ceqq\sigma\big((\pi_1^{r_1})^{-1}(A_1)\times Z\times  (\pi_\W^{r_2})^{-1}(A_2) :0\leq r_1,r_2\leq s, A_1\in\BB(Z),A_2\in\BB(\R^\infty)\big)
\end{align}
with $\pi_\W^t$ defined by \eqref{eq: point evaluation}. 
It suffices to show that $W(t)-W(s)\indep D$ under $Q_z$ for 
\[
D=\cap_{j=1}^nD_j,\qquad D_j=(\pi_1^{r_1^j})^{-1}(A_1^j)\times Z\times  (\pi_\W^{r_2^j})^{-1}(A_2^j),
\] 
as these intersections form a $\pi$-system generating $\G$ and null sets do not affect independence.  
Since $(u,u_0,W)\#Q^\mu=Q^\mu\in\SS_{\Gamma,\CC,\nu}$, Lemma \ref{lem:joint compatible indep incr} (with $\wtu=u$) yields   $W(t)-W(s)\indep D$ under $Q^\mu$. Moreover, 
$
(\Id_\B,\phi^z,\Id_\W)^{-1}(D)=D,  
$ 
so for all $B\in\BB(\R^\infty)$  we find 
\begin{align*}
Q_z\big(\{W(t)-W(s)\in B\}\cap D\big)&= Q^\mu\big((\Id_\B,\phi^z,\Id_\W)^{-1}\big(\{W(t)-W(s)\in B\}\cap D\big)\big)\\
&=Q^\mu\big( \{W(t)-W(s)\in B\}\cap {D}\big)\\
&=Q^\mu(\{W(t)-W(s)\in B\})Q^\mu({D})\\
&=Q_z(\{W(t)-W(s)\in B\})Q_z({D}),
\end{align*}
proving that $W(t)-W(s)\indep \F_s^u \vee\F_s^W$ under $Q_z$. Lemma \ref{lem:SPDE sol is compatible} yields $u\com (u_0,W)$ under $Q_z$. 

We conclude that $\law(u,u_0,W)\in \SS_{\Gamma,\CC,\nu_z}$ under $Q_z$, for $\mu$-\aev $z\in Z_0$. 
Pointwise uniqueness (by \ref{it:path}) yields $u=F_{\delta_z}(u_0,W)$ $Q_z$-a.e.,  for $\mu$-\aev $z\in Z_0$. 
The latter and the definition of $Q_z$  
give for all $B\in\BB(\Om)$ and $\mu$-\aev $z\in Z_0$:  
\begin{align}
  Q_z(B)  =\int_\Om \one_B(u,u_0,W)\dd Q_z &=\int_\Om \one_B(F_{\delta_z}(u_0,W),u_0,W)\dd Q_z\notag\\
  &= 
  \int_Z\int_\W\int_\B \one_B(F_{\delta_z}(\tilde{z},w),\tilde{z},w) \dd\delta_{F_\mu(\tilde{z},w)}(b)\dd\wien(w)\dd\delta_z(\tilde{z})\notag\\
  &=\int_Z\int_\W \one_B(F_{\delta_z}(\tilde{z},w),\tilde{z},w) \dd\wien(w)\dd\delta_z(\tilde{z})\notag\\
  &= \int_\W \one_B(F_{\delta_z}({z},w),{z},w) \dd\wien(w). \label{eq:Q_z}
\end{align}
On the other hand, for all $B\in\BB(\Om)$ and $z\in Z_0$, we have by the definition of $Q_z$,
\begin{align*}
  Q_z(B) =\int_Z\int_\W\int_\B \one_B(b,\tilde{z},w) \dd\delta_{F_\mu(\tilde{z},w)}(b)\dd\wien(w)\dd\delta_z(\tilde{z})
  &= \int_\W \one_B(F_\mu({z},w),z,w)\dd\wien(w).
\end{align*}
Note that the integrand on the right-hand side is measurable as a map $Z\times\W\to [0,1]$. 
Thus, by Fubini's theorem and restriction to $Z_0$ (recall that $Z_0\into Z$), 
\[
Z_0\to [0,1]\col z\mapsto Q_z(B) \text{ is measurable.}
\]
Now we may integrate \eqref{eq:Q_z} with respect to $\mu$ to obtain for all $B\in\BB(\Om)$: 
\begin{align*} 
\int_{Z_0}\int_\W \one_B(F_{\delta_z}({z},w),{z},w)&  \dd\wien(w) \dd\mu(z)
=\int_{Z_0}  Q_z(B)   \dd\mu(z)\\ 
&=\int_{Z_0}\int_\W \one_B(F_\mu({z},w),z,w)\dd\wien(w) \dd\mu(z).
\end{align*}
Thus for $\mu$-\aev $z\in Z_0$, we must have $F(z,w)=F_{\delta_z}(z,w)=F_\mu(z,w)$ for $\wien$-\aev $w\in\W$. 
\end{proof}

Let us recall that several examples of applications of Theorem \ref{th:YW SPDE} and Corollary \ref{cor: single F} were already given in Example \ref{ex:applic}. To conclude, let us discuss how one can rederive the Yamada--Watanabe theory of Kunze and Ondrej\'at from our results. 

\begin{example}\label{ex:kunze setting} 
Theorem \ref{th:YW SPDE} reproves Kunze's result \cite[Th.\ 5.3, Cor.\ 5.4]{kunze13} for analytically weak solutions. 
We choose $\cond=\{\ref{it:sol integrability mild weak},\ref{it:integrated eq mild weak}\}$, $D=Z^*$.    With $E$, $\tilde{E}$, $A\in \LL(D(A),\tilde{E})$, $F$ and $G$ as in \cite[Hyp.\ 3.1]{kunze13}, we put
\begin{align*}
&Y\ceqq Z\ceqq \tilde{E}, \quad \B\ceqq C([0,T];E),\\ 
&b(t,v)\ceqq F(v(t))\in \tilde{E}, \quad \sigma(t,v)\ceqq G(v(t))\in\LL(U,Z) \qquad t\in[0,T], v\in \B.
\end{align*}
Here, we recall Remark \ref{rem: L(U,Y) valued} and Lemma \ref{lem:coeffs markov} and we use that the weak solution notion of Kunze is equivalent to our mild weak solution notion, see \cite[Def.\ 6.1, Def.\ 3.3, Prop.\ 6.3]{kunze13}. Thus Theorem \ref{th:YW SPDE} and Corollary \ref{cor: single F} apply.

Alternatively, one could also redefine   $\Gamma$ in the proof of Theorem \ref{th:YW SPDE}, using a set similar to $A_{\mathrm{weak}}$ of Corollary \ref{cor:SPDE distributional unified}  directly encoding \cite[Def.\ 3.3]{kunze13}: 
for all $z^*\in D(A^*)$, $\P$-a.s.: for all $t\in {\bar{I}_T}$,
$$\<u(t),z^*\? =\<u(0),z^*\?+\int_0^t \<u(s),A^*z^*\?\dd s+\int_0^t \<b(s,u),z^*\?\dd s+\int_0^t z^*\circ \sigma(s,u)\dd {W}(s)\,\text{ in $Z$.}
$$ 
\end{example}

\begin{example}\label{ex:ondrejat setting} 
Theorem \ref{th:YW SPDE} recovers Ondrej\'at's results \cite{ondrejat04} for mild solutions. 
Except for one minor difference explained below, Ondrej\'at's notion of a solution in \cite[p.\ 8]{ondrejat04} corresponds to a $\cond$-weak solution to \eqref{SPDE mild} with $\cond=\{\ref{it:sol integrability mild},\ref{it:integrated eq mild}\}$. Here,  $S(t,s)=S(t-s)$, $Y$ is a separable 2-smooth Banach space (in particular also M-type 2) and $b$ and $\sigma$ are of the  Markovian form \eqref{eq: b sigma markov}. 
We define our spaces $\tilde{X}$, $Y$ and $Z$ as Ondrej\'at's spaces $X$, $X$ and $X_1$, respectively. Moreover, we put $b(\cdot)=f(\cdot)$, $\sigma(\cdot)=g(\cdot)\circ Q^{1/2}$, where $Q\in \LL(U)$ is the covariance operator of the $Q$-Wiener process used for the noise in \cite{ondrejat04}. Note that $Q^{1/2}\in\LL_2(U,U_0)=\gamma(U,U_0)$  and $g$ is $\LL(U_0,X_1)$-valued, so that by the ideal property \cite[Th.\ 9.1.10]{HNVWvolume2}, $\sigma$ is $\gamma(U,X_1)$-valued. By \cite[Lem.\ 2.5]{NVW07}, $U_0$-strong measurability of $g$ (as is required in \cite{ondrejat04}) is equivalent to strong measurability of $\sigma$ (see also \cite[Note 2.6]{ondrejat04}). Lastly, note that in \cite{ondrejat04}, $S$ is $\LL(X_1,X)$-valued and $X_1$-strongly measurable, with $i\col X\into X_1$. Then, $\LL(X_1,X)\into \LL(X_1)=\LL(Z)$, so $S$ is $\LL(Z)$-valued and  \eqref{eq: mble evol fam}  holds since $i$ is Borel measurable, \ie $S$ satisfies the conditions of Definition \ref{def: evol fam}. We use Remark \ref{rem: integrability coeffs flex} to include the integrability condition \cite[(0.2)]{ondrejat04}. 

The only minor difference is that a $\cond$-weak  solution $u$ has to belong \as to $\B$, with $\B\into C(\bar{I}_T;Z)$, whereas in \cite[p.\ 8]{ondrejat04}, it is only required that $\<u,y_n^*\?$ has a continuous, adapted modification on $\bar{I}_T$ for all $n\in\N$, with  $(y_n^*)\subset Y^*$ a sequence separating the points in $Y$. 
However, $S$ is usually a strongly continuous semigroup on $X=X_1=Z$; then, if for example  $b(\cdot,u(\cdot))\in L^1(0,T;Z)$ and $\sigma(\cdot,u(\cdot))\in L^p(0,T;\guz)$ a.s.\ with $p>2$, any solution in Ondrej\'at's sense automatically has a modification with paths in $C(\bar{I}_T;Z)$, see \cite[Th.\ 4.5]{NV20maximal}.   
Moreover, if the previous integrability holds with $p=2$, then \cite[Th.\ 12(1)]{ondrejat04} gives that solutions in Ondrej\'at's sense have a modification with paths in $C([0,T];X_{-1})$. In the latter,  $X_{-1}$ is an extrapolation space defined as the completion of $X$ under  $\|x\|_{X_{-1}}\ceqq \|R_{\lambda} x\|_X$, where  $R_{\lambda}$ is the resolvent of the generator of $S$ for some $\lambda$ in the resolvent set. One can then apply our results with $Z$ redefined as $X_{-1}$. Solutions in our framework then also satisfy Ondrej\'at's   continuity condition with a sequence of functionals that separate the points in $X$, thanks to \cite[Th.\ 12(2)]{ondrejat04}.

It was also pointed out in \cite[p.\ 9]{ondrejat04} that the results in \cite{ondrejat04} remain true if one considers only adapted solutions with norm continuous paths. 

For Ondrej\'at's setting, note that $\Gamma$ in the proof of Theorem \ref{th:YW SPDE} can also be constructed directly, using \cite[Th.\ 6]{ondrejat04}  instead of  Corollaries \ref{cor: mble L^p subsets} and \ref{cor:SPDE distributional unified}, defining 
$$
\Gamma\ceqq \{\mu\in\PP(S_1\times S_1): \exists\, \cond\text{-weak solution } (u,W,(\Om,\F,\P,(\F_t))) \text{ with }\law(u,u(0),W)=\mu\}. 
$$ 
\end{example}

\appendix
\section{Measure theory}\label{appendix1}
We start by giving sufficient conditions for a collection of maps to generate the Borel $\sigma$-algebra on a Polish space. Lemma \ref{lem: generating maps} is a simple consequence of \cite[Prop.\ 1.4, Th.\ 1.2]{vakhania87}, where the result is proved for  a collection of real-valued functions $\Gamma$. Recall that completely Hausdorff (T2) spaces are topological spaces in which any two distinct points can be separated by a continuous real-valued function. This class contains all metric spaces and all locally compact Hausdorff spaces, the latter due to Urysohn's lemma. Also, recall  that a metric space is second-countable if and only if it is separable. 

\begin{lemma}\label{lem: generating maps}
Let $S$ be a Polish space and let $Y$ be a completely Hausdorff space. Let $\Gamma$ be a collection of continuous maps from $S$ to $Y$ that separate the points of $S$. Then $\BB(S)=\sigma(\Gamma)$. 
  
Let $S$ be a Polish space and let $Y$ be a  second-countable Hausdorff space. Let $\Gamma_0$ be a \emph{countable} collection of Borel measurable maps from $S$ to $Y$ that separate the points of $S$, then $\BB(S)=\sigma(\Gamma_0)$.
\end{lemma}
\begin{proof}
Trivially, $\sigma(\Gamma)\subset \BB(S)$. For all pairs of distinct points $x,y\in Y$, pick a continuous map $f_{x,y}\col Y\to\R$ with $f(x)\neq f(y)$. Put $\Gamma'\ceqq\{f_{x,y}\circ\gamma:x,y\in Y, x\neq y, \gamma\in \Gamma\}$. Note that $\sigma(\Gamma')\subset \sigma(\Gamma)$ since each $f_{x,y}$ is Borel measurable. 
Also, $\Gamma'\subset C(S;\R)$ separates the points of $S$, so by \cite[Th.\ 1.2]{vakhania87}, $\BB(S)=\sigma(\Gamma')$. We conclude that $\BB(S)=\sigma(\Gamma)$. 

In the second case, again trivially, $\sigma(\Gamma_0)\subset \BB(S)$. Take a countable base $(V_k)$ for the topology on $Y$ and consider the countable collection 
\[
\mathcal{J}\ceqq\{\gamma^{-1}(V_j):\gamma\in\Gamma_0,j\in\N\}\subset \BB(S). 
\]
Since $V_j\in\BB(Y)$, we have $\sigma(\mathcal{J})\subset \sigma(\Gamma_0)$. Thus it suffices to show that $\BB(S)=\sigma(\mathcal{J})$. 

The sets in $\mathcal{J}$ separate the points of $S$. Indeed, for distinct $s,t\in S$ there exists $\gamma\in\Gamma_0$ such that $\gamma(s)\neq \gamma(t)$ in $Y$. Pick disjoint open sets $U_s,U_t\subset Y$ containing $\gamma(s)$ and $\gamma(t)$, respectively. Since $(V_k)$ is a base, there exist $k(s),k(t)\in\N$ such that $\gamma(s)\in V_{k(s)}\subset U_s$ and $\gamma(t)\in V_{k(t)}\subset U_t$. 
Now $\gamma^{-1}(V_{k(s)})\in\mathcal{J}$ and $\gamma^{-1}(V_{k(t)})\in \mathcal{J}$ indeed separate $s$ and $t$. 
By \cite[Prop.\ 1.4]{vakhania87}, it follows that $\BB(S)=\sigma(\mathcal{J})$. 
\end{proof}

The following lemma is well-known, but $X$ and $S$ are usually assumed to be complete. Therefore, we include a proof.  

\begin{lemma}\label{lem: mble version of map}
  Let $(\Om,\F,\P)$ be a probability space with completion $(\Om,\overline{\F},\overline{\P})$, let $X$ be a metric space and let  
  \[
  f\col(\Om,\overline{\F},\overline{\P})\to X
  \]
  be strongly measurable. Then there exists a strongly measurable map $g\col(\Om,{\F},{\P})\to X$ such that $f=g$ ${\P}$-a.s.

  If $S$ is a second-countable topological space and $f\col(\Om,\overline{\F},\overline{\P})\to S$ is measurable, then there exists a measurable map $g\col(\Om,{\F},{\P})\to S$ such that $f=g$ ${\P}$-a.s.
\end{lemma}
\begin{proof} 
We start with the second part. 
Let $(V_k)$ be a countable basis for the topology on $S$. For each $j\in\N$ we have $f^{-1}(V_j)\in \overline{\F}$, so we can write $f^{-1}(V_j)=A_j\dotcup Z_j$ with $A_j\in\F$, $Z_j\subset Z^0_j\in\F$, $\P(Z^0_j)=0$. Put $Z\ceqq\cup_{j\in\N}Z^0_j\in \F$ and note that $\P(Z)=0$. Fix any $s_0\in S$ and let
$g(\om)\ceqq f(\om)\one_{Z^c}(\om)+s_0\one_{Z}(\om)$. 
Then $f=g$ $\P$-a.s. and it remains to show that $g$ is measurable. Any open set is a countable union of sets in $(V_k)$, so it suffices to show that $g^{-1}(V_j)\in\F$ for all $j\in\N$. If $s_0\notin V_j$, then $g^{-1}(V_j)=Z^c\cap f^{-1}(V_j)=Z^c\cap Z_j^c\cap f^{-1}(V_j)=Z^c\cap A_j\in \F$. If $s_0\in V_j$, then $g^{-1}(V_j)=Z\cup f^{-1}(V_j)=Z\cup A_j \cup Z_j=Z\cup  A_j\in \F$. 
  
The first part follows from the second part. Let $f$ be strongly measurable. By Pettis' measurability lemma \cite[Prop.\ 1.9]{vakhania87}, $f$ is $\overline{\F}/\BB(X)$-measurable and separably valued. Hence $S\ceqq \overline{f(\Om)}\subset X$ is a separable metric space (\ie second-countable). The above yields an $\F/\BB(S)$-measurable map $g_0\col\Om\to S$ with $f=g_0$ a.s., with $\BB(S)=\{A\cap S:A\in\BB(X)\}$. Thus $g\col\Om\to X,\, {g}(\om)=g_0(\om)$ is $\F/\BB(X)$-measurable and satisfies $f=g$ a.s.  
\end{proof}

For a measure space  $(S_1,\A_1,\mu_1)$ and a Banach space $E$, we let $L^0(S_1;E)$ be the vector space of strongly $\mu_1$-measurable functions, with identification of $\mu_1$-\aev equal functions. We say that a sequence $(f_k)$ \emph{converges in measure to} $f$ whenever  for all $A\in \A_1$ with $\mu_1(A)<\infty$ and for all $\delta>0$ it holds that 
\[
\lim_{k\to\infty}\mu_1(A\cap \{s\in S_1:\|f_k(s)-f(s)\|_E>\delta\})=0.
\]

\begin{remark}\label{rem:L0}
If  $\mu_1$ is $\sigma$-finite, then the above sequential convergence notion fully determines the topology of convergence in measure on $L^0(S_1;E)$ as defined in \cite[245A]{fremlin2}. In this case, the topology of convergence in measure is  completely metrizable by the metric given in \cite[Prop.\ A.2.4]{HNVWvolume1} (replace $|\cdot|$ by $\|\cdot\|_E$), and  see also \cite[245E(b)(c), 211L]{fremlin2}. Moreover, \as convergence of a sequence implies convergence in $L^0(S_1;E)$  \cite[Lem.\ A.2.2]{HNVWvolume1}.    

If $(S_1,\A_1,\mu_1)=(I_T,\BB(I_T),\ddd t)$ with $T\in (0,\infty]$ and if $E$ is a separable Banach space, then $L^0(S_1;E)$ is Polish. Separability can be proved similarly as in \cite[245Y(j), 244I,  242O(ii)]{fremlin2}. 
\end{remark} 
\begin{remark}
The notion of convergence in measure introduced above is the local version and not the global one, whose (metrizable) topology is more ill-behaved if $\mu_1(S_1)=\infty$ (see \cite[245B(c), 245Y(e)]{fremlin2}). 
\end{remark}
\begin{lemma}\label{lem:L^0}
Let $(S_1,\A_1,\mu_1)$ be a $\sigma$-finite measure space, let $(S_2,\A_2)$ be a measurable space and let $E$ be a Banach space. Let $f\col S_1\times S_2\to E$ be strongly measurable. Then $\hat{f}\col S_2\to L^0(S_1;E)\col x\mapsto f(\cdot,x)$ is strongly measurable. 
\end{lemma}

\begin{proof}
Observe that  $f(\cdot,x)\in L^0(S_1;E)$ for any fixed $x\in S_2$, \ie $\hat{f}$ is well-defined. 
Since $f$ is strongly measurable, we can approximate $f$ by $\A_1\otimes\A_2$-simple functions. For $A\in\A_1\otimes \A_2$, $\one_A(\cdot,x)\col S_1\to\R$ is measurable by Fubini's theorem.  By taking linear combinations in $E$ and pointwise limits, it follows that $f(\cdot,x)$ is strongly measurable, hence strongly $\mu_1$-measurable \cite[Prop.\ 1.1.16]{HNVWvolume1}. So $\hat{f}$ is well-defined. It remains to prove strong measurability of $\hat{f}$. 

First consider real-valued functions. Define
\begin{align*}
\mathcal{H}\ceqq \{ h\col S_1\times S_2\to\R\,:\;& h \text{ is bounded, measurable and }\\
&\hat{h}\col S_2\to L^0(S_1;\R)\col x\mapsto h(\cdot,x) \text{ is strongly measurable}\}.
\end{align*}
Note that $\mathcal{H}$ is a linear space which is closed under bounded increasing pointwise (on $S_1\times S_2$) convergence, since the latter implies pointwise convergence on $S_2$ in $L^0(S_1;\R)$.  
Moreover, $\mathcal{H}\supset \{\one_{A_1\times A_2}:A_i\in\A_i\}$. Indeed, $S_2\to L^0(S_1;\R)\col x\mapsto\one_{A_1\times A_2}(\cdot,x)=\one_{A_1}(\cdot)\one_{A_2}(x)$ is itself an $\A_2$-simple function. 
By the monotone class theorem, it follows that $\mathcal{H}$ contains all bounded, measurable functions  $S_1\times S_2\to\R$. In particular, $\mathcal{H}\supset\{\one_A:A\in\A_1\otimes\A_2\}$. 
Now put
\begin{align*}
\mathcal{H}'\ceqq \{ f\col S_1\times S_2\to E\,:\;& h \text{ is strongly measurable and }\\
&\hat{f}\col S_2\to L^0(S_1;\R)\col x\mapsto f(\cdot,x) \text{ is strongly measurable}\}.
\end{align*}
Note that $\mathcal{H}'\supset\{h\otimes y:h\in\mathcal{H},y\in E\}$ and $\mathcal{H}'$ is closed under linear combinations, thus $\mathcal{H}'$ contains all $\A_1\otimes\A_2$-simple functions $S_1\times S_2\to E$. 
Furthermore, $\mathcal{H}'$ is closed under pointwise convergence on $S_1\times S_2$, again since the latter implies pointwise convergence on $S_2$ in $L^0(S_1;E)$. 
Thus $\mathcal{H}'$ contains all strongly measurable $f\col S_1\times S_2\to E$, completing the proof.
\end{proof}

The following result is standard. 

\begin{lemma}\label{lem: mble composition} 
  Let $(S,\mathcal{A})$ be a measurable space, let $Y$ be a separable metric space and let $X$ be separable Banach space. Let $\Phi\col S\times Y\to X$ be such that $\Phi(s,\cdot)\col Y\to X$ is continuous for each $s\in S$. Then \ref{it: mble comp 1} and \ref{it: mble comp 2} are equivalent:
  \begin{enumerate}[label=\textup{(\roman*)},ref=\textup{(\roman*)}] 
    \item\label{it: mble comp 1} $\Phi(\cdot, y)\col S\to X$ is $\mathcal{A}/\BB(X)$-measurable for each $y\in Y$,
    \item\label{it: mble comp 2} $\Phi\col S\times Y\to X$ is $\mathcal{A}\otimes \BB(Y)/\BB(X)$-measurable,
  \end{enumerate}
 In addition, suppose that $Y$ is a Banach space. 
  If \ref{it: mble comp 1} or \ref{it: mble comp 2} holds, then
  \begin{enumerate}[resume,label=\textup{(\roman*)},ref=\textup{(\roman*)}] 
    \item\label{it: mble comp 3} if $\phi\col S\to Y$ is strongly $\mathcal{A}$-measurable, then $\Phi(\cdot,\phi(\cdot))\col S\to X$ is strongly $\mathcal{A}$-measurable.
  \end{enumerate}
\end{lemma}

\section{Cylindrical Brownian motion}\label{appendix2}

We consider  a $U$-cylindrical Brownian motion $W$ in the sense of Definition \ref{def: cylindrical BM}. Some properties are discussed in Subsection \ref{sub:stoch integration}. In this Appendix \ref{appendix2}, we prove some additional lemmas.

We let $\F_t^W$ be as in \eqref{eq:F_t^eta}  and  $\overline{\F_t}\ceqq \F_t\cup\mathcal{N}$, where $\mathcal{N}$ is the collection of all $(\Om,\F,\P)$-null sets. In addition, we define $\bar{\F}_{t^+}\ceqq \cap_{h>0}\sigma(\overline{\F_{t+h}})$. 
The next lemma shows that $W$ is also a $U$-cylindrical Brownian motion with respect to $(\overline{\F}_{t})$, $(\bar{\F}_{t^+})$ and $({\F}_{t}^W)$. 
This is standard for real-valued Brownian motions. 
In the cylindrical case, independence aspects are slightly more subtle, so we include a proof. 
To consider $W$ as a $U$-cylindrical Brownian motion on a completed probability space, note that  $L^2(\Om,\overline{\F})\cong L^2(\Om,\F)$ (but we do not have any inclusion as sets). Indeed, for $f,g$ as in Lemma \ref{lem: mble version of map}, $\iota([f]_{\overline{\P}})\ceqq[g]_{\P}$ defines an isomorphism, where $[\cdot]_{\overline{\P}}$, $[\cdot]_\P$ denote equivalence classes of \as equal functions, so we may identify $W$ with $\iota^{-1}\circ W$ and the latter will fulfill the conditions of Definition \ref{def: cylindrical BM} with the filtrations $(\overline{\F}_{t})$, $(\bar{\F}_{t^+})$ and $({\F}_{t}^W)$. 
  
\begin{lemma}\label{lem:sol on completion} 
Let $(e_k)$ be an orthonormal basis for $U$ and let $W$ be a $U$-cylindrical Brownian motion with respect to $(\Om,{\F},{\P}, (\F_t))$. Identify $W=(W^k)$ through \eqref{eq: def Rinfty BM}. 

Then for all $s\in\R_+$:  $W(s+\cdot)-W(s)\indep \F_s$ as a $C(\R_+;\R^\infty)$-valued random variable. In particular, $W(t)-W(s)\indep \F_s$ for all $0\leq s<t$, as an $\R^\infty$-valued random variable. 

Moreover, $W$ is a $U$-cylindrical Brownian motion with respect to $(\Om,\overline{\F},\overline{\P}, (\overline{\F}_{t}))$, $(\Om,\overline{\F},\overline{\P}, (\bar{\F}_{t^+}))$ and $(\Om,\overline{\F},\overline{\P},({\F}_{t}^W))$. 
\end{lemma} 

\begin{proof} 
Let  $s\in\R_+$. 
By Lemma \ref{lem: generating maps}, we have $\BB(C(\R_+;\R^\infty))=\sigma(\Pi_n^r:n\in\N,r\in\R_+)$, with $\Pi_n^r\col C(\R_+;\R^\infty)\to\R^\infty\col w\mapsto (w_1(r),\ldots,w_n(r),0,0,\ldots)$. Writing $\eta_k^s\ceqq W^k(s+\cdot)-W^k(s)$ for $k\in\N$, it follows that
\begin{align*}
  \sigma(W(s+\cdot)-W(s))
  =\sigma\Big(\bigcap_{j=1}^N\big(\eta_1^s(s_j),\ldots,\eta_n^s(s_j)\big)^{-1}(A_j): N,n\in\N, s_j\in \R_+,A_j\in\BB(\R^n) \Big).
\end{align*}   
The intersections on the right-hand side form a $\pi$-system, thus it suffices to show that 
\begin{equation}\label{eq: pi set}
\bigcap_{j=1}^N\big(\eta_1^s(s_j),\ldots,\eta_n^s(s_j)\big)^{-1}(A_j)\indep {\F_s}.
\end{equation}
The $W^k$ are independent standard real-valued $({\F_t})$-Brownian motions, thus $B^n\ceqq (W^1,\ldots,W^n)$ is an $n$-dimensional standard $({\F_t})$-Brownian motion, as follows from L\'evy's characterization \cite[Th.\ 3.16]{karatzas98}. Hence,  
$\big[B^n(s+s_1)-B^n(s);B^n(s+s_2)-B^n(s+s_1);\ldots;B^n(s+s_N)-B^n(s+s_{N-1})\big]\indep{\F_s}$ (viewing $B^n$ as a row vector here). Putting $\phi\col \R^{N\times n}\to\R^{N\times n},\, (\phi(A))_{jk}\ceqq \sum_{i=1}^jA_{ik}$, we obtain 
\begin{align*} 
 &{            \big(\eta_k^s(s_j)\big) _{j=1} ^{N} }
  { \vphantom{ \big(\eta_k^s(s_j)\big) } } _{k=1} ^{n}
  =\big[B^n(s+s_1)-B^n(s);\ldots;B^n(s+s_N)-B^n(s)\big]\\
&\quad=\phi\big(\big[B^n(s+s_1)-B^n(s);B^n(s+s_2)-B^n(s+s_1);\ldots;B^n(s+s_N)-B^n(s+s_{N-1})\big]\big)\indep {\F_s}, 
\end{align*}
using that $\phi$ is measurable. In particular  \eqref{eq: pi set} holds, hence $W(s+\cdot)-W(s)\indep\F_s$. 
Also, we have $W(t)-W(s)=\pi_{\R^\infty}^{t-s}\circ (W(s+\cdot)-W(s))\indep \F_s$, since $\pi_{\R^\infty}^{t-s}$ defined by \eqref{eq: point evaluation} is measurable. 

Now let us prove the last part of the lemma. 
The case $(\Om,\overline{\F},\overline{\P},(\overline{\F}_{t}))$ is trivial, noting that null sets do not affect independence in Definition \ref{def: cylindrical BM}\ref{it:bm3}. 

The case $(\Om,\overline{\F},\overline{\P},(\F_t^W))$ is then   almost trivial, noting that the $\F_t^W$-measurability of Definition \ref{def: cylindrical BM}\ref{it:bm2} holds if $f=\one_{(0,s]}\otimes e_k$ with $s\leq t$ (recall \eqref{eq:F_t^eta}) and then using linearity and continuity (in $f$) of $W$ and density in $L^2(\R_+;U)$ of the linear span of such $f$. 

Regarding $(\Om,\overline{\F},\overline{\P}, (\bar{\F}_{t^+}))$, the non-trivial part is that  $Wf\indep \bar{\F}_{t^+}$ for all $f\in L^2(\R_+;U)$ with $\supp(f)\subset[t,\infty)$. For such $f$, define $f_n\ceqq f\one_{[t+1/n,\infty)}$. Since $\supp(f_n)\subset [t+1/n,\infty)$, we have $Wf_n\indep \overline{\F_{t+1/n}}\supset \bar{\F}_{t^+}$ for all $n\in\N$.  Furthermore, $f=\lim_{n\to\infty}f_n$ \aev on $\R_+$, hence in $L^2(\R_+;U)$  by the Dominated Convergence Theorem. Consequently, $Wf_n\to Wf$ in $L^2(\Om)$ and $Wf_{n_j}\to Wf$ \as for a subsequence. Let $A\in \bar{\F}_{t^+}$ and let $V\subset \R$ be open. Then  $\one_V(Wf_{n_j})\to\one_V(Wf)$ a.s., so   the Dominated Convergence Theorem gives
\[
\E[\one_V(Wf)\one_A]=\lim_{j\to\infty}\E[\one_V(Wf_{n_j})\one_A]=\lim_{j\to\infty}\E[\one_V(Wf_{n_j})]\E[\one_A]=\E[\one_V(Wf)]\E[\one_A]. 
\]
As the $\pi$-system of open sets generates $\BB(\R)$, we conclude that $Wf\indep\bar{\F}_{t^+}$. 
\end{proof} 

 Finally, we relate the $\sigma$-algebras generated by $W$ and $W_1$, when  $W_1$ is a $Q$-Wiener process of the form \eqref{eq: def W_1} induced by $W$. 

Part \ref{it:W W1} of the next lemma shows that up to null sets, $W$ and $W_1$ share the same information. 
Combined with Lemma \ref{lem: mble version of map} and \cite[Lem.\ 1.14]{kallenberg21}, this yields measurable maps $\phi$ and $\phi_1$ such that $W_1=\phi(W)$ \as and $W=\phi_1(W_1)$ a.s. Generally, $\phi$ and $\phi_1$ depend on $W$ and $W_1$ with this reasoning (see the proof of \cite[Lem.\ 1.14]{kallenberg21}). The last part of the next lemma shows that in the present case, $\phi$ and $\phi_1$ can even be constructed \emph{independently of  $W$ and $W_1$}. 
For $t\in \R_+\cup\{\infty\}$ we write  $\bar{I}_t\ceqq \overline{(0,t)}$. 

\begin{lemma}\label{lem: sigma algs W and W_1 new} 
Let $W$ be a $U$-cylindrical Brownian motion on $(\Om,\F,\P, (\F_t))$ and let $(e_k)$ be an orthonormal basis for $U$.  
Let  $W=(W^k)\col\Om\to C(\R_+;\R^\infty)$ and $W_1\col \Om\to C(\R_+;U_1)$ be defined by \eqref{eq: def Rinfty BM} and  \eqref{eq: def W_1}, respectively. For $t\in\R_+\cup\{\infty\}$, define  
\begin{align}\label{eq: def F_t^W new}
\begin{split}
&\F_t^W\ceqq \overline{\sigma((W^k(s))_{k\in\N}:s\in\bar{I}_t)},\qquad\F_t^{W_1}\ceqq \overline{\sigma(W_1(s):s\in\bar{I}_t)}. \end{split}
\end{align} 
Then, the following holds for all $t\in\R_+\cup\{\infty\}$: 
\begin{enumerate}[label=\textup{(\roman*)},ref=\textup{(\roman*)}]\setlength\itemsep{.3em}
\item \label{it:W (e_k)-indep}
${\F_t^W}=\overline{\sigma(W(\one_{(0,s]}\otimes v):s\in\bar{I}_t,v\in U)}$,  
\item\label{it:W W1} 
$\F_t^W= \F_t^{W_1}$.   
\end{enumerate}
Furthermore, there exist measurable maps $\phi\col C(\R_+;\R^\infty)\to C(\R_+;U_1)$ and $\phi_1\col C(\R_+;U_1)\to C(\R_+;\R^\infty)$, \emph{independent of $W,W_1$ and $(\Om,\F,\P, (\F_t))$}, such that
\[
W_1=\phi(W) \text{ \as in }C(\R_+;U_1),\qquad W=\phi_1(W_1) \text{ \as in }C(\R_+;\R^\infty).
\]
Such maps also exist if $\R_+$ is replaced by $[0,T]$ with $T\in(0,\infty)$.
\end{lemma}
\begin{proof} 
Using Lemma \ref{lem: generating maps} with continuous maps $\R^\infty\to\R\col x\mapsto x_k$ ($k\in\N$), we deduce   
\begin{equation}\label{eq:s k split}
  \F_t^W=\overline{\sigma(W^k(s): s\in\bar{I}_t,k\in\N)}.
\end{equation}

\ref{it:W (e_k)-indep}: `$\subset$': This follows directly from \eqref{eq:s k split} and  \eqref{eq: def Rinfty BM}.

`$\supset$':  Let $s\in\bar{I}_t$ and $v\in U$ be arbitrary. By the Dominated Convergence Theorem, we have $\one_{(0,s]}\otimes v=\sum_{j\in\N}\<e_j,v\?_U\one_{(0,s]}\otimes e_j$ in $L^2(\R_+;U)$. Thus, as   $W\in\mathcal{L}(L^2(\R_+;U);L^2(\Om))$, we find
\begin{equation}\label{eq: L2 lim}
  W(\one_{(0,s]}\otimes v)=\sum_{j\in\N}\<e_j,v\?_UW(\one_{(0,s]}\otimes e_j)
\end{equation}
in $L^2(\Om)$. In particular, the series in \eqref{eq: L2 lim} converges in probability, and hence a.s. by the It\^o--Nisio theorem \cite[Th.\ 6.4.1]{HNVWvolume2}. For each $N\in\N$, we have 
\[
\sigma\Big(\sum_{j=1}^N\<e_j,v\?_UW(\one_{(0,s]}\otimes e_j)\Big)\subset \F_t^W
\]
since $\R^N\to\R\col x\mapsto \sum_{j=1}^N\<e_j,v\?_Ux_j$ is measurable and by \eqref{eq:s k split}. Consequently, 
$
\Om_0\ceqq \{\om\in\Om:\sum_{j\in\N}\<e_j,v\?_UW(\one_{(0,s]}\otimes e_j)(\om) \text{ converges}\}\in\F_t^W,
$ 
$\P(\Om_0)=1$ and
\begin{equation}\label{eq: on Om_0}
  \sigma(\one_{\Om_0}W(\one_{(0,s]}\otimes v))=\sigma(\one_{\Om_0}\sum_{j\in\N}\<e_j,v\?_UW(\one_{(0,s]}\otimes e_j))\subset\F_t^W,
\end{equation}
where the last inclusion follows from the fact that $\one_{\Om_0}\sum_{j=1}^N\<e_j,v\?_UW(\one_{(0,s]}\otimes e_j)$ is $\F_t^W$-measurable for each $N\in\N$, thus the same holds for the pointwise limit. 
Since $\F_t^W$ contains all null sets, \eqref{eq: on Om_0} implies $\sigma(W(\one_{(0,s]}\otimes v))\subset {\F_t^W}$, and `$\subset$' follows. 

We proceed by proving the final statement, after which we derive \ref{it:W W1} as a corollary. 
Define ${\phi}\col C(\R_+;\R^\infty)\to C(\R_+;U_1)$ by 
\[
 w=(w_k)\mapsto   
\begin{cases}
  \sum_{k\in\N}w_k(\cdot) Je_k,\quad &\text{if the series converges in $C(\R_+;U_1)$,}\\
  0, &\text{otherwise.}
  \end{cases}   
\]
Then, ${\phi}$ is Borel measurable, since ${\phi}^n\col w\mapsto \sum_{k=1}^n w_k(\cdot) Je_k$ is  continuous, hence measurable,  and since ${\phi}(w)=\lim_{n\to\infty} \one_{\{({\phi}^k) \text{ converges}\}}(w){\phi}^n(w)$.  
Since the series in \eqref{eq: def W_1} converges a.s.\ in $C(\R_+;U_1)$ (see Remark \ref{rem: Q-wiener}), it follows from the definition of $\phi$ that $W_1 = \phi(W)$ a.s.

To construct $\phi_1$, we exploit a  singular value decomposition of $J$. 
Since $J$ is a Hilbert--Schmidt operator, it is compact. By \cite[Th.\ 9.2]{NeervenFA},  it has a singular value decomposition 
\begin{equation}\label{eq:sing val dec}
Jv=\sum_{j\in \N} \tau_j\<f_j,v\?_Ug_j\quad \text{in }U_1,\quad \text{ for all } v\in U,  
\end{equation} 
with $(f_j)$ an orthonormal sequence in $U$, $(g_j)$ an orthonormal sequence in $U_1$ and $\tau_j> 0$ for all $j$.  Since $J$ is injective, $(J^*J)^{1/2}$ only has strictly positive eigenvalues. Thus by \cite[Th.\ 9.2, Th.\ 9.1(1)]{NeervenFA}, we find that $\mathrm{span}\{f_j:j\in \N\}$ is dense in $U$, so $(f_j)$ is an orthonormal basis for $U$. 

For $i\in \N$ and fixed $t\in\R_+$, the \as convergence of \eqref{eq: def W_1} (which in particular holds in $U_1$ for fixed $t$) and  \eqref{eq:sing val dec}  yield a.s.
\begin{align}\label{eq:sing val dec W1}
\<W_1(t),g_i\?_{U_1}=\sum_{k\in\N}W(\one_{(0,t]}\otimes e_k) \<Je_k,g_i\?_{U_1}&=\sum_{k\in\N}\sum_{j\in \N}\tau_j\<f_j,e_k\?_U W(\one_{(0,t]}\otimes e_k)\<g_j,g_i\?_{U_1}\notag\\
&=\tau_i\sum_{k\in\N} \<f_i,e_k\?_UW(\one_{(0,t]}\otimes e_k)\notag\\
&=\tau_iW(\one_{(0,t]}\otimes f_i). 
\end{align}
In the last equality we used the \as convergence of \eqref{eq: L2 lim}. 

As $\tau_j>0$ for all $j$, 
we can define $\phi_1\col C(\R_+;U_1)\to C(\R_+;\R^\infty)$ by 
  \[
  \phi_1(w)_k\ceqq  
  \begin{cases} 
 \sum_{j\in\N}\frac{1}{\tau_j}\<e_k,f_j\?_U\<w(\cdot),g_j\?_{U_1},\quad &\text{if the series converges in $C(\R_+;\R)$,}\\
  0, &\text{otherwise.}
  \end{cases}
  \]
  The map $\phi_1$ is Borel measurable by the same argument as for $\phi$, and by \eqref{eq:sing val dec W1} we have for all $t\in\R_+$ and $k\in\N$: a.s.  
\begin{align*} 
\sum_{j\in\N}\frac{1}{\tau_j}\<e_k,f_j\?_U\<W_1(t),g_j\?_{U_1} 
= \sum_{j\in\N} \<e_k,f_j\?_U  W(\one_{(0,t]}\otimes f_j)
=W(\one_{(0,t]}\otimes e_k) =W^k(t).
\end{align*}
In the second-last equality, we used that $W\in\mathcal{L}(L^2(\R_+;U),L^2(\Om))$ and applied the It\^o--Nisio theorem in $\R$. Thus, the series on the left-hand side converges a.s.\ to $W^k(t)$ for each fixed $t$. 
The summands $\frac{1}{\tau_j}\<e_k,f_j\?_U\<W_1(\cdot),g_j\?_{U_1}$ are independent, symmetric, and $C(\R_+;\R)$-valued, so their restrictions to $[0, T]$ take values in the separable Banach space $C([0,T];\R)$. Moreover, $W^k$ is a.s.\ continuous. Therefore, the It\^o--Nisio theorem for separably valued random variables \cite[Th.\ V.2.4]{vakhania87} upgrades the pointwise-in-time convergence above to a.s.\ convergence in $C([0,T];\R)$ for every $T\in(0,\infty)$. This implies a.s.\ convergence in $C(\R_+;\R)$ (equipped with the topology of uniform convergence on compact sets), for each fixed $k$. Recalling the definition of $\phi_1$, we now conclude that $\phi_1(W_1)=W$ \as

The same constructions of $\phi$ and $\phi_1$ work when $\R_+$ is replaced by $[0,T]$. 

\ref{it:W W1}: Viewing $W|_{\bar{I}_t}$ and $W_1|_{\bar{I}_t}$ as  $C(\bar{I}_t;\R^\infty)$- and $C(\bar{I}_t;U_1)$-valued random variables respectively, it holds that
  \begin{equation}\label{eq:F_t^W equiv}
  \F_t^{W}=\overline{\sigma(W|_{\bar{I}_t})},\quad \F_t^{W_1}= \overline{\sigma(W_1|_{\bar{I}_t})}.
  \end{equation}
Indeed, for $S=\R^\infty$ and $S= U_1$, the continuous maps  $\pi^s_{S,t}\col C(\bar{I}_t;S)\to S\col w\mapsto w(s)$, $s\in\bar{I}_t$  separate the points in $C(\bar{I}_t;S)$, hence $\BB(C(\bar{I}_t;S))=\sigma(\pi^s_{S,t}:s\in \bar{I}_t)$ by Lemma \ref{lem: generating maps}.  
Recalling that $\pi^s_{\R^\infty,t}\circ W=W(s)$ and $\pi^s_{U_1,t}\circ W_1=W_1(s)$ yields \eqref{eq:F_t^W equiv}. \ref{it:W W1} now immediately follows from the measurable maps $\phi$ and $\phi_1$ (with $\R_+$ replaced by $[0,t]$). 
\end{proof}

\printbibliography%[heading=bibintoc]

\end{document}